\documentclass[onefignum,onetabnum]{siamart190516}

\makeatletter
\newcommand*{\addFileDependency}[1]{% argument=file name and extension
  \typeout{(#1)}
  \@addtofilelist{#1}
  \IfFileExists{#1}{}{\typeout{No file #1.}}
}
\makeatother

\newcommand*{\myexternaldocument}[1]{%
    \externaldocument{#1}%
    \addFileDependency{#1.tex}%
    \addFileDependency{#1.aux}%
}

\usepackage{multicol}
\usepackage{pgfplots}
\usepackage{lipsum}
\usepackage{amsfonts,amsmath,amssymb,mathtools,bbm,dsfont,stmaryrd,aligned-overset,bm}
\usepackage{graphicx}
\usepackage{wrapfig}
\usepackage{epstopdf}
\usepackage{amsopn}
\usepackage{cancel}
\usepackage{physics}
\usepackage{mdframed}
\usepackage{subcaption}
\usepackage{comment}
\usepackage{booktabs}
\usepackage{microtype}
\usepackage{enumitem}\setlist[enumerate,1]{label={(\roman*)}}
\ifpdf
  \DeclareGraphicsExtensions{.eps,.pdf,.png,.jpg}
\else
  \DeclareGraphicsExtensions{.eps}
\fi
\newcommand{\nth}[1]{{#1}^{\mathrm{th}}}

\newbool{arxiv}
\boolfalse{arxiv}

\usepackage{multirow}

\usepackage{tikz,pgfplots}
\usetikzlibrary{patterns,hobby,calc,intersections}
\pgfplotsset{compat=1.16}
\usepackage{pgfplotstable}
\usepgfplotslibrary{fillbetween}
\usepackage{graphics/tikzart}

\let\norm\relax
\newcommand{\naturals}                       {\mathbb{N}}

\newcommand{\reals}                          {\mathbb{R}}
\newcommand{\realsBar}                       {\bar{\reals}}
\newcommand{\nonnegativeReals}               {\reals_{\geq 0}}
\newcommand{\nonnegativeRealsBar}            {\realsBar_{\geq 0}}

\DeclareMathOperator{\id}                    {id}
\DeclareMathOperator{\proj}                  {proj}

\newcommand{\st}                             {\,|\,}

\renewcommand{\d}{\,\mathrm{d}}

\newcommand{\indexedVar}[3]                 {{#1}_{#2}^{#3}}

\newcommand{\norm}[1]                        {\left\Vert #1\right\Vert}

\newcommand{\expectedValue}[2]               {\mathbb{E}^{#1}\ifstrempty{#2}{}{\left[#2\right]}}

\newcommand{\diracMeasure}[1]                {\delta_{#1}}

\renewcommand{\d}[0]                         {\,\mathrm{d}}

\newcommand{\Pp}[2]                          {\mathcal{P}_{#1}(#2)}

\newcommand{\spaceProbabilityBorelMeasures}[1]{\mathcal{P}\left(#1\right)}

\DeclareMathOperator{\supp}                  {supp}

\newcommand{\wassersteinDistance}[3]         {W_{#1}\ifstrempty{#2}{}{(#2,#3)}}

\newcommand{\plan}[1]{\bm{#1}}
\newcommand{\setPlans}[2]                    {\Gamma(#1\ifstrempty{#2}{}{,#2})}
\newcommand{\SetPlans}[2]                    {\Gamma\left(#1\ifstrempty{#2}{}{,#2}\right)}

\newcommand{\pushforward}[1]                 {{#1}_{\#}}

\renewcommand{\ae}                           {\text{ a.e.}}

\newcommand{\Cb}[1]                          {C_b\ifstrempty{#1}{}{(#1)}}

\newcommand{\comp}[2]                        {#2 \circ #1}%{#1 \fatsemi #2}

\newcommand{\statespace}[1]{X_{#1}}
\newcommand{\otherspace}[1]{Y_{#1}}
\newcommand{\refVar}{r}
\newcommand{\refPVar}{\rho}
\newcommand{\refSpace}[1]{R_{#1}}
\newcommand{\anotherspace}[1]{Z_{#1}}
\newcommand{\inputspace}[1]{U_{#1}}
\newcommand{\probabilityInputSpace}[1]{\mathcal{U}_{#1}}
\newcommand{\noisespace}[1]{W_{#1}}

\newcommand{\costtogo}[1]{J_{#1}}
\newcommand{\stagecost}[1]{G_{#1}}
\newcommand{\costtogosmall}[1]{j_{#1}}
\newcommand{\stagecostsmall}[1]{g_{#1}}
\newcommand{\dynamics}[1]{f_{#1}}

\newcommand{\dynamicsbig}[1]{F_{#1}}
\newcommand{\terminalcost}{G_N}
\newcommand{\terminalcostsmall}{g_N}

\newcommand{\projectionFromTo}[2]{\proj^{#1}_{#2}}
\newcommand{\identityOn}[1]{\id_{#1}}

\newcommand{\inIndexSet}[3]{#1 \in \{#2, \ldots, #3\}}
\newcommand{\inputVariable}[4]{u_{#4}^{#3}\ifstrempty{#1}{}{[#1,#2]}}
\newcommand{\inputMap}[4]{u_{#4}^{#3}\ifstrempty{#1}{}{(#1,#2)}}
\newcommand{\epsilonPlan}[3]{\plan{#1}^{#2}\ifstrempty{#3}{}{_{#3}}}
\newcommand{\transportMap}[2]{T^{#1}_{#2}}
\newcommand{\probabilityInput}[2]{\lambda^{#1}_{#2}}

\newcommand{\kantorovich}[4]{\mathcal{K}[#1]\ifstrempty{#2}{}{(#2\ifstrempty{#4}{\ifstrempty{#3}{}{,}#3}{\ldots,#4})}}

\newcommand{\freemarginals}[4]{\mathcal{J}[#1]\ifstrempty{#2}{}{(#2\ifstrempty{#4}{\ifstrempty{#3}{}{,}#3}{\ldots,#4})}}

\NewDocumentCommand{\statcirc}{ O{#2} m }{%
    \begin{tikzpicture}
    \fill[#2] (0,0) circle (.7ex); % Fill circle with base colour (arg#2)
    \fill[#1] (0,0) -- (215:.6ex) arc (215:395:.6ex) -- cycle; % Fill a half circle filled with second colour (arg#1), if specified
    \end{tikzpicture}
}
\NewDocumentCommand{\statrect}{ O{#2} m }{%
    \begin{tikzpicture}
      \fill[#2]   (0,0) rectangle  ++ (.2,.2);
      \fill[#1]   (.01,.01) -- ++ (.18,.18) |- cycle;
    \end{tikzpicture}
}
\NewDocumentCommand{\statdiam}{ O{#2} m }{%
    \begin{tikzpicture}[xscale=.7,yscale=.9]
      \fill[#2,rotate=45]   (0,0) rectangle  ++ (.2,.2);
      \fill[#1,rotate=45]   (.01,.01) -- ++ (.18,.18) |- cycle;
    \end{tikzpicture}
}
\NewDocumentCommand{\stattrian}{ O{#2} m }{%
    \begin{tikzpicture}
      \fill[#2]   (0,0) -- (.2,.0) -- (.1,.2) -- (.0,.0);
      \fill[#1]   (.1,.01) -- (.19,.01) -- (.1,.19) -- (.1,.01);
    \end{tikzpicture}
}
\usepackage{microtype}
\usepackage[acronym]{glossaries}
\newacronym{acr:amod}{AMoD}{autonomous mobility-on-demand}
\newacronym{acr:av}{AV}{autonomous vehicle}
\newacronym{acr:are}{ARE}{Algebraic Riccati Equation}
\newacronym{acr:dare}{DARE}{Discrete-time Algebraic Riccati Equation}
\newacronym{acr:dro}{DRO}{distributionally robust optimization}
\newacronym{acr:dpa}{DPA}{Dynamic Programming Algorithm}
\newacronym{acr:ethz}{ETH Zürich}{Eidgenössische Technische Hochschule Zürich}
\newacronym{acr:kl}{KL}{Kullback-Leibner}
\newacronym{acr:lsc}{lsc}{lower semi-continuous}
\newacronym{acr:oc}{OC}{Optimal Control}
\newacronym{acr:pmp}{PMP}{Pontryagin Maximum Principle}
\newacronym{acr:lqr}{LQR}{Linear Quadratic Regulator}
\newacronym{acr:rl}{RL}{Reinforcement Learning}
\newacronym{acr:spp}{SPP}{Shortest Path Problems}
\newacronym{acr:usc}{usc}{upper semi-continuous}
\newacronym{acr:costcat}{Cost}{Category of costs}
\newacronym{acr:dpcat}{DP}{Category of dynamic programming problems}

\newsiamthm{assumption}{Assumption}
\newsiamremark{remark}{Remark}
\newsiamremark{example}{Example}
\newsiamremark{counterexample}{Counterexample}
\newlist{hypothesis}{enumerate}{10}
\setlist[hypothesis]{label*=(H\arabic*)}
\crefname{hypothesisi}{Hypothesis}{Hypotheses}
\newlist{property}{enumerate}{10}
\setlist[property]{label*=(P\arabic*)}
\crefname{propertyi}{Property}{Properties}
\newsiamthm{plainproblem}{Problem}
\crefname{plainproblem}{Problem}{Problems}
\newenvironment{problem}%
{\begin{plainproblem*}}%
{\end{plainproblem*}}%\end{mdframed}}
\theoremstyle{nonumberplain}
\theoremheaderfont{\normalfont\sc}
\theorembodyfont{\normalfont\itshape}
\theoremseparator{.}
\theoremsymbol{}
\newtheorem{informaltheorem}{Informal Statement}

\headers{Dynamic Programming in Probability Spaces}{Antonio Terpin, Nicolas Lanzetti, and Florian Dörfler}

\title{Dynamic Programming in Probability Spaces via Optimal Transport\thanks{Published in the \emph{SIAM Journal on Control and Optimization}, 62(2), 1183-1206, 2024.
\funding{This research was supported by the Swiss National Science Foundation under the NCCR Automation (grant 51NF40\_180545).
The authors are affiliated with the Automatic Control Laboratory, ETH Zürich (\email{\{aterpin,lnicolas,dorfler\}@ethz.ch}).}}}

\author{Antonio Terpin%
\thanks{Equal contribution.}
\and Nicolas Lanzetti\footnotemark[2]%
\and Florian Dörfler
}
\usetikzlibrary{external}
\ifpdf
\hypersetup{
  pdftitle={OptCon},
  pdfauthor={Antonio Terpin, Nicolas Lanzetti, and Florian Dörfler}
}
\fi

\myexternaldocument{supplement}

\allowdisplaybreaks

\begin{document}

\maketitle

% REQUIRED
\begin{abstract}\noindent%
We study discrete-time finite-horizon optimal control problems in probability spaces, whereby the state of the system is a probability measure. 
We show that, in many instances, the solution of dynamic programming in probability spaces results from two ingredients: (i) the solution of dynamic programming in the “ground space” (i.e., the space on which the probability measures live) and (ii) the solution of an optimal transport problem. From a multi-agent control perspective, a separation principle holds: The “low-level control of the agents of the fleet” (how does one reach the destination?) and “fleet-level control” (who goes where?) are decoupled.
\end{abstract}
% REQUIRED

% REQUIRED
\section{Introduction}
\label{section:introduction}
Many optimal control problems of stochastic or large-scale dynamical systems can be framed in the probability space, whereby the state is a \emph{probability measure}. We provide three examples, starting with a pedagogical case: 
% Many optimal control problems in stochastic or large-scale systems can be formulated in the probability space, whereby the state of the dynamical system is a \emph{probability measure}. We provide three examples, starting with a pedagogical case: 
\begin{example}[Deterministic optimal control]
\label{example:deterministic-optimal-control}
Consider a discrete-time dynamical system with state space $\statespace{k}$, input space $\inputspace{k}$, and dynamics $\dynamics{k}: \statespace{k}\times \inputspace{k}\to \statespace{k+1}$.
The problem of steering the system from an initial state $x_0\in\statespace{0}$ along an $N \in \naturals$ time-steps long target trajectory $\{\refVar{}_k \in \statespace{k}\}_{k = 0}^N$ (e.g., $\refVar{}_0 = \ldots = \refVar{}_N = 0$ for a regulation problem) while minimizing the sum of the terminal cost $\terminalcostsmall: \statespace{N}\times\statespace{N}\to\nonnegativeRealsBar$ and the stage costs $\stagecostsmall{k}: \statespace{k}\times\inputspace{k}\times\statespace{N}\to\nonnegativeRealsBar$ reads as
\begin{equation}\label{equation:deterministic-optimal-control}
% \costtogosmall{}(x_0, \refVar{}_N)
% =
\inf_{\inputMap{}{}{}{k}: \statespace{k} \to \inputspace{k}}
\terminalcostsmall(x_N, \refVar{}_N)
+
\sum_{k = 0}^{N-1}
\stagecostsmall{k}(x_k, \inputMap{}{}{}{k}(x_k), \refVar{}_k),
\end{equation}
subject to the dynamics. The costs $\stagecostsmall{k}$ and $\terminalcostsmall$ measure the ``closeness'' between the state $x_k$ and the reference $\refVar_k$, as well as the input effort. For instance, when all the spaces are $\reals^n$, they may be defined as $\stagecostsmall{k}(x_k, \inputVariable{}{}{}{k}, \refVar_k) = \norm{x_k - \refVar_k}^2 + \norm{\inputVariable{}{}{}{k}}^2$ and $\terminalcostsmall(x_N, \refVar_N) = \norm{x_N - \refVar_N}^2$. It is instructive to capture this setting via probability measures. At each time-step $k$, consider the Dirac's delta probability measure $\mu_k = \diracMeasure{x_k}$, and let $\refPVar_k = \diracMeasure{\refVar{}_k}$. The relation between $\mu_{k+1}$ and $\mu_k$ is the ``pushforward'' operation $\mu_{k+1} = \diracMeasure{x_{k+1}} = \diracMeasure{\dynamics{k}(x_k, \inputMap{}{}{}{k}(x_k))} = \pushforward{\dynamics{k}(\cdot, \inputMap{}{}{}{k}(\cdot))}{\mu_k}$, a dynamics in the probability space.
An equivalent formulation to \eqref{equation:deterministic-optimal-control} is
\begin{equation}\label{equation:deterministic-optimal-control:probability-space}
\begin{aligned}
% \costtogosmall{}(x_0, \refVar{}_N)
% =
% \costtogo{}(\mu_0, \refPVar_N)
% &=
\inf_{\inputMap{}{}{}{k}: \statespace{k}\to \inputspace{k}}
&
\underbrace{
\int_{\statespace{N}}\int_{\statespace{N}}
\terminalcostsmall(x_N, \refVar{}_N)
\d\mu_N(x_N)\d\refPVar_N(\refVar{}_N)
}_{\terminalcost(\mu_N, \refPVar_N)}
\\&\qquad\qquad
+
\sum_{k = 0}^{N-1}
\underbrace{
\int_{\statespace{k}}\int_{\statespace{k}}
\stagecostsmall{k}(x_k, \inputMap{}{}{}{k}(x_k), \refVar{}_k)
\d\mu_k(x_k)\d\refPVar_k(\refVar{}_k)
}_{\stagecost{k}(\mu_k, \inputMap{}{}{}{k}, \refPVar_k)},
% \\&=
% \inf_{\inputMap{}{}{}{k}: \statespace{k}\to  \inputspace{k}}
% \terminalcost(\mu_N, \refPVar_N)
% +
% \sum_{k = 0}^{N - 1}
% \stagecost{k}(\mu_k, \inputMap{}{}{}{k}, \refPVar_N),
\end{aligned}
\end{equation}
where $
%\costtogo{}, 
\terminalcost$ and $\stagecost{k}$ have the same meaning as the lower-case counterparts in \eqref{equation:deterministic-optimal-control} but are defined in the probability space.
%The formal similarity between \eqref{equation:deterministic-optimal-control} and \eqref{equation:deterministic-optimal-control:probability-space} suggests that one can climb the abstraction ladder and frame~\eqref{equation:deterministic-optimal-control:probability-space} as an optimal control problem, with state $\mu_k$ and inputs $\inputMap{}{}{}{k}$ \cite{Chen2021}. 
\end{example}
\begin{example}[Distribution steering]
Assume now that the initial condition $x_0$ in \cref{example:deterministic-optimal-control} is unknown, but its realization is distributed according to $\mu_0\in\Pp{}{\statespace{0}}$, with $\Pp{}{\statespace{0}}$ being the space of probability measures over $\statespace{0}$. The input to apply to each ``particle'' having state $x_k \in \statespace{k}$ is given by the (deterministic) feedback map $\inputMap{}{}{}{k}:\statespace{k}\to\inputspace{k}$, and the dynamical evolution is $x_{k+1} = \dynamics{k}(x_k, \inputMap{}{}{}{k}(x_k))$. Similarly to \cref{example:deterministic-optimal-control}, the dynamics in the probability space are then $\mu_{k+1} = \pushforward{\dynamics{k}(\cdot, \inputMap{}{}{}{k}(\cdot))}{\mu_k}$. The same formalism of \cref{example:deterministic-optimal-control} can then be used to ensure that the terminal state $x_N$ is distributed closely to a desired probability measure $\refPVar_N \in \Pp{}{\statespace{N}}$:
\begin{equation}\label{equation:example:distribution-steering}
% \costtogo{}(\mu_0, \refPVar_N)
% \coloneqq
\inf_{\inputMap{}{}{}{k}: \statespace{k}\to \inputspace{k}} \;
\sum_{k = 0}^{N}
\mathcal{C}_k(\mu_k, \refPVar_N) 
+
\int_{\statespace{k}}
\norm{\inputMap{}{}{}{k}(x_k)}^2\d\mu_k(x_k),
\end{equation}
where $\mathcal{C}_{k}$ measures closeness between $\mu_k$ and $\refPVar_N$, akin to $\stagecost{k}$ and $\terminalcost{}$ in \eqref{equation:deterministic-optimal-control:probability-space}.
\end{example}

\begin{example}[Large-scale multi-agent systems]
\label{example:multi-agent}
The optimal steering of a fleet of $M$ identical agents, with dynamics $x_{k + 1} ^{(i)} = \dynamics{k}(x_k^{(i)}, \inputMap{}{}{}{k}(x_k^{(i)}))$, from an initial configuration $\{x_k^{(i)}\}_{i = 1}^M$ to a desired one $\refPVar_N$ can be cast as 
\begin{equation}\label{equation:example:multi-agent}
    % \costtogo{}(\{x_0^{(i)}\}_{i = 1}^M, \refPVar_N) 
    % \coloneqq 
    \inf_{\inputMap{}{}{}{k}: \statespace{k} \to \inputspace{k}} 
    \; 
    \sum_{k =  0}^{N} 
    \mathcal{C}_k(\{x_k^{(j)}\}_{j = 1}^M, \refPVar_N)
    +
    \frac{1}{M}
    \sum_{i = 1}^M
    \terminalcostsmall(x_N^{(i)})
    +
    \sum_{k =  0}^{N - 1} 
    \stagecostsmall{k}(x_k^{(i)}, \inputMap{}{}{}{k}(x_k^{(i)})),
\end{equation}
where $\mathcal{C}_k$ is a fleet-specific cost (e.g., a cohesion or formation cost), and  $\terminalcostsmall$ and $\stagecostsmall{k}$ are  agent-specific costs (e.g., input effort). 
Often, the interest lies in the \emph{macroscopic behavior} of the fleet. Hence, it is customary to capture the state of the fleet by a probability measure $\mu_k \in \Pp{}{\statespace{k}}$ and the input by a map $\inputMap{}{}{}{k}: \statespace{k} \to \inputspace{k}$ \cite{HudobadeBadyn2021,Krishnan2019}.
The optimization problem in~\eqref{equation:example:multi-agent} can then be written as an optimal control problem, with state $\mu_k$, input $\inputMap{}{}{}{k}$, and dynamics $\mu_{k+1} = \pushforward{\dynamics{k}(\cdot, \inputMap{}{}{}{k}(\cdot))}{\mu_k}$. Overall, 
\begin{equation}\label{equation:example:multi-agent:revised}
% \begin{aligned}
    % \costtogo{}(\mu_0, \refPVar_N) 
    % &\coloneqq 
    \inf_{\inputMap{}{}{}{k}: \statespace{k} \to \inputspace{k}} \; 
    \int_{\statespace{N}} \terminalcostsmall(x_N) \d\mu_N(x_N)
    % \\&\qquad\qquad
    +
    \sum_{k =  0}^{N - 1} 
    \mathcal{C}_k(\mu_k, \refPVar_N)
    +
    \int_{\statespace{k}} \stagecostsmall{k}(x_k, \inputMap{}{}{}{k}(x_k)) \d\mu_k(x_k).
% \end{aligned}
\end{equation}
Such a modeling approach suits robotics~\cite{Terpin2021}, mobility~\cite{Zardini2021}, and social networks~\cite{Albi2015,Huang2020}.
\end{example}

Formally,~\eqref{equation:deterministic-optimal-control:probability-space},~\eqref{equation:example:distribution-steering}, and~\eqref{equation:example:multi-agent:revised} are instances of discrete-time finite-horizon optimal control problems in probability spaces:
\begin{equation}\label{equation:problem-informal}
    % \costtogo{}(\mu_0, \refPVar_0, \ldots, \refPVar_N) \coloneqq
    \inf_{\inputMap{}{}{}{k}: \statespace{k}\to\inputspace{k}} \; 
    \terminalcost(\mu_N, \refPVar_N)
    +
    \sum_{k =  0}^{N - 1}
    \stagecost{k}(\mu_k, \inputMap{}{}{}{k}, \refPVar_k),
\end{equation}
subject to the measure dynamics $\mu_{k+1} = \pushforward{\dynamics{k}(\cdot, \inputMap{}{}{}{k}(\cdot))}{\mu_k}$, where $\refPVar_k$ are (possibly time-dependent) reference probability measures.
In this paper, we consider $\stagecost{k}, \terminalcost$ as \emph{optimal transport discrepancies}: An optimal transport discrepancy measures the effort to transport one probability measure onto another when moving a unit of mass from $x$ to $y$ costs $c(x,y)$% and is therefore well-suited to measure closeness between probability measures
; see \cref{section:probabilityspace}.
To solve~\eqref{equation:problem-informal}, one possibility is the \gls*{acr:dpa}~\cite{Bertsekas2014}. However, its deployment poses several analytical and computational challenges. For example, it is unclear which easy-to-verify assumptions ensure the existence of solutions. Moreover, even if a minimizer exists, its computation suffers the infinite dimensionality of the probability space and the burden of repeated computations of optimal transport discrepancies; see \cref{section:problem}.

This inevitable complexity prompts us to adopt a different perspective: At least formally,~\eqref{equation:problem-informal} resembles a \emph{single} optimal transport problem~\cite{Ambrosio2008,Villani2007}, whereby one seeks to transport one probability measure $\mu_0$ to a final one $\refPVar_N$ while minimizing some transportation cost. If this formal similarity is made rigorous, we can tackle~\eqref{equation:problem-informal} with the tools of optimal transport theory, which has reached significant maturity in recent years, both theoretically~\cite{Ambrosio2008,Santambrogio2015,Villani2007} and numerically~\cite{Cuturi2013,Peyre2019,Taskesen2022}. Moreover, the available computational libraries (see, e.g.,~\cite{Flamary2021,Schmitzer2019}) provide a wealth of methods, mainly relying on the so-called \emph{regularized} optimal transport problem~\cite{Cuturi2013,Peyre2019}, which is significantly faster to solve and recovers the original formulation and solution as a zero-noise limit.

\subsection{Contributions}
We study the optimal control and dynamic programming in probability spaces through the lens of optimal transport theory. Specifically, we show that many optimal control problems in probability spaces can be reformulated and studied as optimal transport problems. Our results reveal a separation principle: The “low-level control of the agents of the fleet” (how does one reach the destination?) and “fleet-level control” (who goes where?) are decoupled.
We complement our theoretical analysis with various examples and counterexamples, which demonstrate that our conditions cannot be relaxed and expose the pitfalls of heuristic approaches. 
The proofs of our results rely on novel stability results for the (multi-marginal) optimal transport problem, which are of independent interest.

\subsection{Previous work}
Most of the literature focuses on continuous time, and it is founded on~\cite{Benamou2000}, which relates the optimal transport problem and fluid mechanics. Through the optimal control lens, this formulation corresponds to an optimal control problem with integrator dynamics: The resulting flow is a time-dependent feedback law~\cite{Chen2021}. An attempt to introduce generic dynamical constraints can be found in~\cite{Bonnet2023,Cavagnari2020,Cavagnari2023}, where the set of possible flows is constrained in a set of admissible ones, induced by the dynamics. Constructive results can be found in the specific setting of linear systems and Gaussian probability measures. In this case and when the control laws are affine, the space of probability measures is implicitly constrained to the space of Gaussian distributions, and closed-form solutions exist~\cite{Chen2016,Chen2016a,Chen2016b,Chen2018c}. All of these works build on traditional optimal control tools. In~\cite{Bonnet2019b,Bonnet2019c,Bonnet2021}, instead, the authors develop a Pontryagin Maximum Principle for optimal control problems in the Wasserstein space (i.e., probability space endowed with the Wasserstein distance). Their analysis combines classical tools in optimal control theory with the ``differential structure'' of the Wasserstein space~\cite{Ambrosio2008,Bonnet2021,Lanzetti2022}. In \cite{Hindawi2010}, the authors study optimal transport when the transportation cost results from the cost-to-go of a \gls*{acr:lqr}. This methodology implicitly assumes that, to steer a fleet of identical particles, one can compute the cost-to-go for the single particle and then ``lift'' the solution to the probability space via an optimal transport problem. While attractive, this approach generally yields suboptimal solutions; see \cref{section:examples}.
 
The discrete-time setting has, instead, received less attention. Towards this direction,~\cite{Bakolas2016,Bakolas2017,Bakolas2018b,Bakolas2018,Balci2021} explore the covariance control problem for discrete-time linear systems, possibly subject to constraints. In~\cite{Krishnan2019}, the authors study the optimal steering of multiple agents from an initial configuration to a final one in a distributed fashion.
%However, their method works only with distances as costs and integrator dynamics. 
In~\cite{HudobadeBadyn2021}, the authors follow an approach similar to \cite{Hindawi2010}, albeit in discrete time. In \cite{Arque2022}, the authors study the problem of mass transportation over a graph, embedding constraints such as the maximum flow on the edges. To do so, they exploit the structure of the ground space, in this case, the transportation graph. Finally, when the evolution is Markovian and the cost is the Kullback-Leibler divergence, the optimal transport problem over a graph has been addressed in~\cite{Chen2017Network,Chen2021,Chen2018Network,Chen2020Network}.
In all these approaches, the distribution/fleet steering problem is a priori formalized as an optimal transport problem and not as an optimal control problem in the probability space. As we shall see in~\cref{section:lifting}, our results bridge these two perspectives and allow us to back up and recover many of the approaches in the literature.

\subsection{Organization}
The paper unfolds as follows. In~\cref{section:probabilityspace,section:problem}, we review the space of probability measures and introduce our problem setting. We present and discuss our main result, \cref{theorem:lifting}, in \cref{section:lifting}. In~\cref{section:examples}, we provide examples to ignite an intuition on our results and expose potential pitfalls. All proofs are in~\cref{section:proofs}. Finally,~\cref{section:conclusion} summarizes our findings and future directions. 

\subsection{Notation}
We denote by $\Cb{\statespace{}}$ the space of continuous and bounded functions $\statespace{} \to \reals$
and by $\nonnegativeRealsBar=[0,+\infty]$ the set of nonnegative extended real numbers.
The identity map on $\statespace{}$ is denoted by $\identityOn{\statespace{}}$, and the projection maps from $\statespace{}\times\otherspace{}$ onto $\statespace{}$ are denoted by $\projectionFromTo{\statespace{}\times\otherspace{}}{\statespace{}}$.
%; e.g., $\projectionFromTo{\statespace{}\times\otherspace{}}{\statespace{}}: \statespace{}\times\otherspace{}\to\statespace{}$, $\projectionFromTo{\statespace{}\times\otherspace{}}{\statespace{}}(x, y) = x$.
%We distinguish two ways of ``stacking'' functions. 
Given the set of maps $\{h_k: \statespace{} \to \statespace{k}\}_{k=i}^j$ we denote by 
% \begin{align*}
$
(h_i, \ldots, h_j):
\statespace{} \to \statespace{i}\times\ldots\times\statespace{j}
$
the map
$
x
\mapsto 
(h_i, \ldots, h_j)(x) \coloneqq (h_i(x), \ldots, h_j(x)).
$
% \end{align*}
\section{The Space of Probability Measures}
\label{section:probabilityspace}

We start with notation and preliminaries in~\cref{subsection:probabilityspace:basics}. Then, in~\cref{subsection:probabilityspace:functionals}, we review optimal transport.  

\subsection{Preliminaries}\label{subsection:probabilityspace:basics}
We assume all spaces to be Polish spaces and all probability measures and maps to be Borel. 
We denote by $\Pp{}{\statespace{}}$ the space of Borel probability measures on $\statespace{}$, and we denote by $\delta_{x}$ the Dirac's delta at $x\in \statespace{}$; i.e., the probability measure defined for all Borel sets $B \subseteq \statespace{}$ as $\delta_{x}(B) = 1$ if $x \in B$ and $\delta_x(B) = 0$ otherwise.
We denote by $\supp(\mu)$ the support of a probability measure $\mu \in \Pp{}{\statespace{}}$. %That is, for every Borel set $B \subseteq \statespace{} \setminus \supp(\mu)$, $\mu(B) = 0$.
The \emph{pushforward} of a probability measure $\mu \in \Pp{}{\statespace{}}$ through $T: \statespace{} \to \otherspace{}$, denoted by $\pushforward{T}{\mu}\in\Pp{}{\otherspace{}}$, is defined by $(\pushforward{T}{\mu})(A) = \mu(T^{-1}(A))$ for all Borel sets $A \subseteq \otherspace{}$. For any $\pushforward{T}{\mu}$-integrable, $\phi: \otherspace{} \to \reals$ it holds
$
\int_{\otherspace{}} \phi \d(\pushforward{T}{\mu}) = \int_{\statespace{}} \comp{T}{\phi} \d\mu.
$
Given $\nu \in \Pp{}{\otherspace{}}$, $T$ is a \emph{transport map} from $\mu$ to $\nu$ if $\pushforward{T}{\mu} = \nu$; to this extent, it suffices that for all $\phi \in \Cb{\otherspace{}}$
$
    \int_{\otherspace{}} \phi \d \nu = \int_{\otherspace{}} \phi \d (\pushforward{T}{\mu}).% = \int_{\statespace{}} \comp{T}{\phi} \d \mu.
$
We say that $T: \statespace{} \to \otherspace{}\times\anotherspace{}$ is a transport map from $\mu \in \Pp{}{\statespace{}}$ to $(\nu_1, \nu_2) \in \Pp{}{\otherspace{}}\times\Pp{}{\anotherspace{}}$ if $\pushforward{(\projectionFromTo{\otherspace{}\times\anotherspace{}}{\otherspace{}})}{(\pushforward{T}{\mu})} = \nu_1$ and $\pushforward{(\projectionFromTo{\otherspace{}\times\anotherspace{}}{\anotherspace{}})}{(\pushforward{T}{\mu})} = \nu_2$.

\subsection{Optimal transport}\label{subsection:probabilityspace:functionals}

Given a nonnegative \emph{transportation cost} $c: \statespace{} \times \otherspace{} \to \nonnegativeRealsBar$, the \emph{optimal transport discrepancy} $\kantorovich{c}{}{}{}: \Pp{}{\statespace{}}\times \Pp{}{\otherspace{}} \to \nonnegativeRealsBar$ between two probability measures $\mu\in\Pp{}{\statespace{}}$ and $\nu\in\Pp{}{\otherspace{}}$ is
\begin{equation}
\label{equation:cost:optimaltransport:fromreference}
% \begin{aligned}
% \kantorovich{c}{}{}{}:\Pp{}{\statespace{}}\times \Pp{}{\otherspace{}}&\to\nonnegativeRealsBar\\
% (\mu, \nu) &\mapsto \kantorovich{c}{\mu}{}{\nu} = \inf_{\plan{\gamma} \in \setPlans{\mu}{\nu}} \int_{\statespace{} \times \otherspace{}} c \d\plan{\gamma},
% \end{aligned}
% \kantorovich{c}{}{}{}:\Pp{}{\statespace{}}\times \Pp{}{\otherspace{}}\to\nonnegativeRealsBar,\quad 
% \Pp{}{\statespace{}}\times \Pp{}{\otherspace{}} \ni
% (\mu, \nu) \mapsto 
\kantorovich{c}{\mu}{\nu}{} \coloneqq \inf_{\plan{\gamma} \in \setPlans{\mu}{\nu}} \int_{\statespace{} \times \otherspace{}} c(x,y) \d\plan{\gamma}(x,y)
% \in \nonnegativeRealsBar
,
\end{equation}
where $\setPlans{\mu}{\nu}
\coloneqq
\left\{
\plan{\gamma}\in\Pp{}{\statespace{}\times\otherspace{}}
\st
\pushforward{(\projectionFromTo{\statespace{}\times\otherspace{}}{\statespace{}})}{\plan{\gamma}} = \mu,
\pushforward{(\projectionFromTo{\statespace{}\times\otherspace{}}{\otherspace{}})}{\plan{\gamma}} = \nu
\right\}$ is the set of couplings.
% with the set of couplings $\setPlans{\mu}{\nu}$ defined as
% \begin{displaymath}
% \setPlans{\mu}{\nu}
% \coloneqq
% \left\{
% \plan{\gamma}\in\Pp{}{\statespace{}\times\otherspace{}}
% \st
% \pushforward{(\projectionFromTo{\statespace{}\times\otherspace{}}{\statespace{}})}{\plan{\gamma}} = \mu,
% \pushforward{(\projectionFromTo{\statespace{}\times\otherspace{}}{\otherspace{}})}{\plan{\gamma}} = \nu
% \right\}.
% \end{displaymath}
A prominent example of optimal transport discrepancy is, for some $p \geq 1$, the $\nth{p}$ power of the $p$-Wasserstein distance, obtained when $\statespace{} = \otherspace{}$ and the transportation cost $c$ is a metric that induces the topology on $\statespace{}$ \cite[\S 7]{Ambrosio2008}.

\begin{remark}\label{remark:background:expectedvalue}
When the transportation cost does not depend on one of the two variables (e.g., there exists $\tilde{c} \in \statespace{} \to \nonnegativeRealsBar$ such that $c(x, y) = \tilde{c}(x)$), the optimal transport discrepancy reduces to an expected value; i.e., $\kantorovich{c}{\mu}{\nu}{}=\expectedValue{\mu}{\tilde c}$.
\end{remark}

We will repeatedly work with a generalization of the optimal transport problem to $k$ marginals. Let $\statespace{} \coloneqq \statespace{1}\times\ldots\times\statespace{k}$, and $c: \statespace{}\to\nonnegativeRealsBar$. The \emph{multi-marginal} optimal transport problem between $k$ probability measures $\{\mu_i \in \Pp{}{\statespace{i}}\}_{i = 1}^k$ reads as
\begin{equation}
\label{equation:cost:multimarginaloptimaltransport}
\kantorovich{c}{\mu_1}{}{\mu_k} \coloneqq \inf_{\plan{\gamma} \in \setPlans{\mu_1, \ldots}{\mu_k}} 
\int_{\statespace{}}
c(x_1,\ldots,x_k)
\d\plan{\gamma}(x_1,\ldots,x_k),
\end{equation}
where
$
\setPlans{\mu_1}{\ldots, \mu_k}
\coloneqq
\left\{
\plan{\gamma} \in \Pp{}{\statespace{}} 
\st 
\pushforward{(\projectionFromTo{\statespace{}}{\statespace{i}})}{\plan{\gamma}} = \mu_i, \, \inIndexSet{i}{1}{k}
\right\}.
$
In general, the infima in~\eqref{equation:cost:optimaltransport:fromreference} and~\eqref{equation:cost:multimarginaloptimaltransport} are not attained, unless mild conditions on the transportation cost hold true (e.g., $c$ lower semicontinuous in \eqref{equation:cost:optimaltransport:fromreference}~\cite[\S 4]{Villani2007}). A transport plan $\epsilonPlan{\gamma}{\varepsilon}{} \in \setPlans{\mu_1, \ldots}{\mu_k}$ is $\varepsilon$-optimal when
\begin{equation}
\label{equation:transport-plan:epsilon-optimal}
\int_{\statespace{}} c(x_1, \ldots, x_k) \d\epsilonPlan{\gamma}{\varepsilon}{}(x_1, \ldots, x_k) \leq \kantorovich{c}{\mu_1}{}{\mu_k} + \varepsilon.
\end{equation}

The formulation in~\eqref{equation:cost:optimaltransport:fromreference} and \eqref{equation:cost:multimarginaloptimaltransport} is the \emph{Kantorovich formulation} of the optimal transport problem, whereby one optimizes over \emph{transport plans} $\plan{\gamma}$. The (stricter) \emph{Monge formulation}\footnote{Historically, the Monge formulation comes first. For a thorough review of the history of optimal transport and its founding fathers, see \cite[\S 1]{Villani2007}.} considers only transport plans $\plan{\gamma} = \pushforward{(\identityOn{\statespace{1}}, \transportMap{}{})}{\mu_1} \in \setPlans{\mu_1,\ldots}{\mu_k}$ induced by a transport map $T:X_1\to X_2\times\ldots X_k,
\pushforward{T}{\mu_1} = (\mu_2, \ldots, \mu_k)$. 
% Namely,
% \begin{displaymath}
% \monge{c}{\mu_1}{}{\mu_k}
% \coloneqq
% \inf_{\substack{
% T:X_1\to X_2\times\ldots X_k
% \\
% \pushforward{T}{\mu_1} = (\mu_2, \ldots, \mu_k)}}
% \int_{\statespace{}} 
% c(x_1,T(x_1))
% \d\mu_1(x_1).
% \end{displaymath}
% If $\int_{\statespace{}}c\d(\pushforward{(\id_{\statespace{1}}, \transportMap{\varepsilon}{})}{\mu_1} \leq \monge{c}{\mu_1}{}{\mu_k} + \varepsilon$, $\transportMap{\varepsilon}{}: \statespace{1}\to\statespace{2}\times\ldots\times\statespace{k}$ is $\varepsilon$-optimal.

\begin{comment}
For any cost $V: \statespace{} \to \nonnegativeRealsBar$, we define its expected value functional by
\begin{equation}
\label{equation:cost:energy}
    \expectedValue{\cdot}{V}: \Pp{}{\statespace{}} \to \nonnegativeRealsBar,
    \quad
    \mu \mapsto \expectedValue{\mu}{V} \coloneqq \int_{\statespace{}} V \d \mu.
\end{equation}
This functional can be used to encode the average/total cost associated with the probability measure:  
\begin{example}[Cumulative effort]
We can model a fleet of $M$ particles as a uniform distribution over the states of the agents; i.e., $\mu = \frac{1}{M}\sum_{i \in \mathcal{I}}\delta_{x_i}$, with $\mathcal{I}$ being some indexing of the particles. Then, the total energy/fuel $e(x_i)$ used by each particle $x_i$ amounts to the expected value of $Me$:
\begin{displaymath}
    \sum_{i \in \mathcal{I}} e(x_i) = \sum_{i \in \mathcal{I}} M\frac{1}{M} e(x_i) = \sum_{i \in \mathcal{I}} M e(x_i) \mu(x_i) = \expectedValue{\mu}{M e}.
\end{displaymath}
\end{example}
\end{comment}
\section{Problem Statement}
\label{section:problem}
Let $\statespace{k}$, $\inputspace{k}$, and $\refSpace{k}$ be Polish spaces, representing the state space, the input space, and the space of references in the ground space, respectively (often, $\refSpace{k}=\statespace{k}$).
We consider dynamical systems whose state is a probability measure over $\statespace{k}$. This approach encompasses continuous approximations of multi-agent systems and systems with uncertain initial conditions (usually captured by absolutely continuous probability measures), as well as finite settings (captured by empirical probability measures).
\begin{example}[Robots in a grid]
\label{example:robots}
Consider $M$ robots in a grid of three cells; i.e., $\statespace{k}=\{\pm 1,0\}$. Suppose that the $\nth{i}$ robot is located at $x_k^{(i)}\in\statespace{k}$ (i.e., has state $x_k^{(i)}$). Then the state of the system is $\mu_k=\frac{1}{M}\sum_{i=1}^{M}\diracMeasure{x_k^{(i)}}$. The same modeling approach applies to $M$ robots in the two-dimensional plane, simply setting $\statespace{k}=\reals^2$.
\end{example}
In this setting, we focus on the following optimal control problem. 
\begin{problem}
[Discrete-time optimal control in probability spaces]
\label{problem:finitehorizon:rigorous}
Let $N\in\naturals_{\geq 1}$.
For dynamics $\dynamics{k}:\statespace{k}\times\inputspace{k}\to\statespace{k+1}$, costs $\stagecostsmall{k}:\statespace{k}\times\inputspace{k}\times\refSpace{k}\to\nonnegativeRealsBar$ and $\terminalcostsmall{}:\statespace{N}\times\refSpace{N}\to\nonnegativeRealsBar$, initial condition $\mu\in\Pp{}{\statespace{0}}$, and reference trajectory $\refPVar=(\refPVar_0,\ldots,\refPVar_N)\in\Pp{}{\refSpace{0}}\times\ldots\times\Pp{}{\refSpace{N}}$, find the joint state-input distribution $\probabilityInput{}{k}\in\Pp{}{\statespace{k}\times\inputspace{k}}$ which solve
\begin{displaymath}
\begin{aligned}
\costtogo{}(\mu,\refPVar) 
= 
\inf_{\substack{
\mu_k\in\Pp{}{\statespace{k}}
\\
\probabilityInput{}{k}\in\Pp{}{\statespace{k}\times\inputspace{k}}
}}\;&
\kantorovich{\terminalcostsmall{}}{\mu_{N}}{\refPVar_N}{}
 + \sum_{k = 0}^{N - 1}\kantorovich{\stagecostsmall{k}}{\probabilityInput{}{k}}{\refPVar_k}{}
\\
\mathrm{s.t.}\qquad &
\mu_{k + 1} = \pushforward{\dynamics{k}}{\probabilityInput{}{k}},
\quad
\mu_0 = \mu, 
\\
&\pushforward{(\projectionFromTo{\statespace{k}\times\inputspace{k}}{\statespace{k}})}{\probabilityInput{}{k}}=\mu_k.
\end{aligned}
\end{displaymath}
\end{problem}
Before presenting our results, we detail our setting. The notation in the ground space is juxtaposed with the one in the measure space in \cref{table:notation}.
\begin{table}[]
    \centering
    \begin{tabular}{c|c|c}
        %\toprule
        & ground space & measure space
        \\ \hline   
        state & $x_k \in \statespace{k}$ & $\mu_k \in \Pp{}{\statespace{k}}$ 
        \\
        reference & $\refVar_k \in \refSpace{k}$ & $\refPVar_k \in \Pp{}{\refSpace{k}}$ 
        \\
        \multirow{2}{*}{state-input distribution} & \multirow{2}{*}{$x \mapsto \inputMap{}{}{}{k}(x_k) \in \inputspace{k}$}
        %or $\inputMap{}{}{}{k}: \statespace{k}\times\refSpace{k}\times\ldots\times\refSpace{N}\to\inputspace{k}$ 
        & $\probabilityInput{}{k} \in \Pp{}{\statespace{k}\times\inputspace{k}}$
        \\
        & & s.t. $\pushforward{(\projectionFromTo{\statespace{k}\times\inputspace{k}}{\statespace{k}})}{\probabilityInput{}{k}}=\mu_k$
        \\
        dynamics & $x_{k+1} = \dynamics{k}(x_k, \inputMap{}{}{}{k}(x_k))$ & $\mu_{k+1} = \pushforward{\dynamics{k}}{\probabilityInput{}{k}}$
        \\
        cost-to-go & $\costtogosmall{k}$ & $\costtogo{k}$
        \\
        stage and terminal costs & $\stagecostsmall{k}$ and $\terminalcostsmall$ & $\kantorovich{\stagecostsmall{k}}{}{}{}$ and $\kantorovich{\terminalcostsmall}{}{}{}$ \\
        %\bottomrule
    \end{tabular}
    \caption{Parallelism between objects in the ground space and in the measure space.}
    \label{table:notation}
    \ifbool{arxiv}{}{\vspace{-.75cm}}
\end{table}
\subsection{State-input distribution}
\label{section:problem:state-input-distribution}
The state-input distribution $\probabilityInput{}{k}\in\Pp{}{\statespace{k}\times\inputspace{k}}$ is a probability measure on $\statespace{k}\times\inputspace{k}$ whose first marginal is $\mu_k$. The semantics is as follows: The probability mass assigned by $\probabilityInput{}{k}$ to the pair $(x_k,\inputMap{}{}{}{k})$ indicates the probability that one particle has state $x_k \in \statespace{k}$ and applies the input $\inputMap{}{}{}{k}\in\inputspace{k}$ or, equivalently, the share of agents which have state $x_k\in\statespace{k}$ and apply the input $\inputMap{}{}{}{k}\in\inputspace{k}$.
When $\probabilityInput{}{k}=\pushforward{(\identityOn{X},\inputMap{}{}{}{k})}{\mu_k}$ for some $u_k:\statespace{k}\to\inputspace{k}$, the input is ``deterministic'': All particles that have state $x_k\in\statespace{}$ apply the input $\inputMap{}{}{}{k}(x_k)\in\inputspace{k}$. % and, accordingly, $\probabilityInput{}{k}(\{(x_k,\inputVariable{}{}{}{k})\}) = \mu_k(\{x_k\})$.
%We additionally fix the first marginal of $\probabilityInput{}{k}$ to be $\mu_k$, so that $\probabilityInput{}{k}$ is a full description of the state and the input of the dynamic system.
% We exemplify our modeling approach as follows: 
\begin{example}[Robots in a grid, continued]
\label{example:robots:input}
Consider again $M$ identical robots on $\statespace{k}=\{\pm 1,0\}$, where at each time-step each robot can either move to the origin ($\inputVariable{}{}{}{k} = 0$) and stay there forever or change position ($\inputVariable{}{}{}{k} = -1$), so that $\inputspace{k}=\{0,1\}$ and $\dynamics{k}(x_k, \inputVariable{}{}{}{k}) = x_k\inputVariable{}{}{}{k}$. Consider the following input-state distributions $\probabilityInput{(1)}{k}$ and $\probabilityInput{(2)}{k}$.
\begin{center}
    \begin{tabular}{c c|c c}
         \multicolumn{2}{c|}{\multirow{2}{*}{$\probabilityInput{(1)}{k}$}} & \multicolumn{2}{c}{$\inputVariable{}{}{}{k}$} \\
         \multicolumn{2}{c|}{} & 0 & $-1$ \\ \hline 
         \multirow{3}{*}{$x_k$} & $-1$ & 0.2 & 0.3 \\
         & $0$  & 0.5 & 0.0 \\
         & $+1$ & 0.0 & 0.0
    \end{tabular}
    \hspace{2cm} 
    \begin{tabular}{c c| c c}
         \multicolumn{2}{c|}{\multirow{2}{*}{$\probabilityInput{(2)}{k}$}} & \multicolumn{2}{c}{$\inputVariable{}{}{}{k}$} \\
         \multicolumn{2}{c|}{} & 0 & $-1$ \\ \hline 
         \multirow{3}{*}{$x_k$} & $-1$ & 0.0 & 0.5\\
         & $0$  & 0.5 & 0.0\\
         & $+1$ & 0.0 & 0.0\\
    \end{tabular}
\end{center}
In the first case (i.e., $\probabilityInput{(1)}{k}$), $20\%$ of the robots are located at $x_k = -1$ and go to the origin ($\inputVariable{}{}{}{k} = 0$), $30\%$ of the robots are located at $x_k = -1$ and switch position ($\inputVariable{}{}{}{k} = -1$), and $50\%$ of the robots are located at $x_k = 0$ and remain there ($\inputVariable{}{}{}{k} = 0$, despite being irrelevant for the dynamics). The input is not deterministic, since not all robots located at $x_k = -1$ apply the same input. From $\probabilityInput{(1)}{k}$ we can also infer the distribution of the robots: $50\%$ of them are located at $x_k = -1$ and the other $50\%$ at $x_k = 0$. In the second case (i.e., $\probabilityInput{(2)}{k}$), the input is deterministic: All robots located at $x_k = -1$ switch position, and all the robots located at $x_k = 0$ stay there. 
\end{example}
\begin{remark}
We have two comments on our modeling choice.
First, since the first marginal of $\probabilityInput{}{k}$ is $\mu_k$, the costs $\kantorovich{\stagecostsmall{k}}{\probabilityInput{}{k}}{\refPVar_k}{}$ are implicitly a function of the state, the input, and the reference trajectory. 
% Second, the joint state-input distribution can be alternatively represented by the conditional probability measures $\{\probabilityInput{}{k,x}\in\Pp{}{\inputspace{k}}\}_{x\in\statespace{k}}$ over the input space $\inputspace{k}$. By disintegration~\cite[Theorem 5.3.1]{Ambrosio2008}, these two formulations are equivalent. 
Second, in multi-agent settings, one may worry that inputs are incompatible with the fleet size. For instance, in~\cref{example:robots:input}, $\probabilityInput{(1)}{k}$ cannot apply to a fleet of two agents, since it is not possible to have 20\% of the agents at a given state and apply a given input.
Nonetheless, optimal inputs never do that: After a state augmentation (i.e., with $M$ copies of the same state, where $M$ is the fleet size), optimal inputs are guaranteed to be deterministic and hence never ``split an individual agent''.
%Still, the more general joint state-input distribution considerably simplifies the analysis, the same way the Kantorovich formulation is more tractable than the Monge formulation in optimal transport theory.
\end{remark}

\subsection{Dynamics}
\label{section:problem:dynamics}
We consider measure dynamics resulting from the pushforward via a function $\dynamics{k}:\statespace{k}\times\inputspace{k}\to\statespace{k+1}$ (typically, the dynamics of the single particles); i.e., $\mu_{k+1}=\pushforward{\dynamics{k}}{\probabilityInput{}{k}}$. In the special case of deterministic inputs (i.e., $\probabilityInput{}{k}=\pushforward{(\identityOn{X},u_k)}{\mu_k}$ for some function $u_k:\statespace{k}\to\inputspace{k}$), the dynamics simplifies to $\mu_{k+1}=\pushforward{\dynamics{k}(\cdot,u_k(\cdot))}{\mu_k}$.
\begin{example}[Robots in a grid, continued]
\label{example:robots:dynamics}
Consider the setting of~\cref{example:robots:input}, where $\dynamics{k}(x_k, \inputVariable{}{}{}{k}) = x_k\inputVariable{}{}{}{k}$. The measure dynamics are $\mu_{k+1}=\pushforward{\dynamics{k}}{\probabilityInput{}{k}}$, and the two inputs of~\cref{example:robots:input} yield $\mu_{k+1}^{(1)}=0.7\diracMeasure{0}+0.3\diracMeasure{1}$ and $\mu_{k+1}^{(2)}=0.5\diracMeasure{0} + 0.5\diracMeasure{1}$. 
\end{example}

\subsection{Cost}
\label{section:problem:cost}
We consider optimal transport discrepancies with, as transportation costs, $\stagecostsmall{k}:\statespace{k}\times\inputspace{k}\times\refSpace{k}\to\nonnegativeRealsBar$ (stage cost) and $\terminalcostsmall{}:\statespace{N}\times\refSpace{N}\to\nonnegativeRealsBar$ (terminal cost). By~\cref{remark:background:expectedvalue}, this modeling assumption includes expected values but not functionals such as the variance of the probability measure or the \acrlong*{acr:kl} divergence from the references $\refPVar_k$, $\refPVar_N$. Our formulation encompasses the terminal constraint $\mu_N = \refPVar_N$: It suffices to set $\terminalcostsmall(x_N, \refVar_N) = +\infty$ if $x_N \neq \refVar_N$. Similarly, state-dependent input constraints $\inputspace{k}(x_k)$ can be encoded setting $\stagecostsmall{k}(x_k, \inputVariable{}{}{}{k}, \refVar_k) = +\infty$ when $\inputVariable{}{}{}{k} \not\in \inputspace{k}(x_k)$. In view of \cref{example:deterministic-optimal-control}, the transportation costs $\stagecostsmall{k}$ and $\terminalcostsmall{}$ may be interpreted as the cost incurred by a single agent.
\begin{example}[Robots in a grid, continued]
\label{example:robots:cost}
Suppose that the goal is to steer $\frac{M}{2}$ robots to $x_k = -1$ and $\frac{M}{2}$ to $x_k = +1$, while minimizing the input. Then $\refPVar_N=\frac{1}{2}\diracMeasure{-1}+\frac{1}{2}\diracMeasure{+1}$ and, for some weight $\alpha>0$, possible costs are $\terminalcostsmall{}(x_N,\refVar_N)=\abs{x_N-\refVar_N}$ and $\stagecostsmall{k}(x_k,v_k,\refVar_k)=\alpha |v_k|$.
This way, the aim is to minimize the (type 1) Wasserstein distance from the reference $\refPVar_N$ at the end of the horizon (i.e., $\kantorovich{\terminalcostsmall}{\mu_N}{\refPVar_N}{}$) and the (weighted) input effort throughout the horizon (i.e., $ \kantorovich{\stagecostsmall{k}}{\probabilityInput{}{k}}{\refPVar_k}{}=\alpha\expectedValue{\probabilityInput{}{k}}{|v_k|}$). The weight $\alpha>0$ arbitrates between these two objectives. The references $\refPVar_k$ for $\inIndexSet{k}{0}{N-1}$ do not enter in the cost and are therefore irrelevant. 
\end{example}

\subsection{\texorpdfstring{\gls*{acr:dpa}}{DPA} for~\texorpdfstring{\cref{problem:finitehorizon:rigorous}}{Problem 3.2}}

\cref{problem:finitehorizon:rigorous} is a discrete-time finite-horizon optimal control problem in abstract spaces~\cite{Bertsekas2014,Bertsekas2017}. It is therefore natural to deploy the \gls*{acr:dpa}.

\begin{definition}[\texorpdfstring{\gls*{acr:dpa}}{DPA}] \label{definition:dpa}
Initialization: Let $\costtogo{N}(\mu_N,\refPVar_N) \coloneqq \kantorovich{\terminalcostsmall}{\mu_N}{\refPVar_N}{}$.
\\\noindent
Recursion: For all $k\in\{N-1,N-2,\ldots,1,0\}$, compute the cost-to-go $\costtogo{k}$:
\begin{equation}\label{equation:dynamicprogrammingalgorithm}
\begin{aligned}
    \costtogo{k}(\mu_k,\refPVar_k,\ldots,\refPVar_N) &\coloneqq
    \inf_{\substack{
        \probabilityInput{}{k} \in \Pp{}{\statespace{k}\times\inputspace{k}}
        \\
        \pushforward{(\projectionFromTo{\statespace{k}\times\inputspace{k}}{\statespace{k}})}{\probabilityInput{}{k}}=\mu_k}}
    \kantorovich{\stagecostsmall{k}}{\probabilityInput{}{k}}{\refPVar_k}{}
    + 
    \costtogo{k+1}(\pushforward{\dynamics{k}}\probabilityInput{}{k},\refPVar_{k+1},\ldots,\refPVar_N).
\end{aligned}
\end{equation}
\end{definition}
Unfortunately, the \gls*{acr:dpa} in probability spaces poses several analytic and computational challenges; we mention two. 
First, it is unclear under which easy-to-verify assumptions minimizers exist. Second, even if they do, their computation remains challenging, if not prohibitive. Already when all sets are finite, and the (generally infinite-dimensional) probability space reduces to the finite-dimensional probability simplex, \eqref{equation:dynamicprogrammingalgorithm} is excruciating. For instance, the mere evaluation of $\costtogo{N}$ involves solving an optimal transport problem with all the related computational difficulties \cite{Cuturi2013,Genevay2016,Peyre2019,Bahar2022}. Thus, the optimization of $\costtogo{N}$, needed to compute $\costtogo{N-1}$, will inevitably be very demanding.

In the following, we show that the solution of~\cref{problem:finitehorizon:rigorous} can be constructed from the solution of the \gls*{acr:dpa} in the ground space (i.e., $\statespace{0},\statespace{1},\ldots$) and a \emph{single} (possibly multi-marginal) optimal transport problem.
In other words, a separation principle holds: The optimal control law results from the combination of optimal low-level control laws (found via \gls*{acr:dpa} in the ground space) and a fleet-level control law (found via an optimal transport problem). 
This way, we bypass the cumbersome application of \gls*{acr:dpa} in probability spaces as well as the repeated evaluation of optimal transport discrepancies. 
% In other words, the ``extra price'' to work in the probability space instead of in the ground space is the solution to a single optimal transport problem.
At least formally, our result generalizes two well-known extreme cases. 
On the one hand, when considering Dirac's delta probability measures, the \gls*{acr:dpa} in the probability space reduces to the \gls*{acr:dpa} in the ground space (see \cref{example:deterministic-optimal-control}); 
on the other hand, when considering trivial dynamics (i.e., $N=1$ and $\dynamics{0}(x_0,\inputVariable{}{}{}{0})=x_0$) and an optimal transport discrepancy as a terminal cost, \cref{problem:finitehorizon:rigorous} reduces to an optimal transport problem. 
Thus, \gls*{acr:dpa} in probability spaces should be at least ``as difficult as'' solving both the \gls*{acr:dpa} in the ground space and an optimal transport problem. As we shall see below, it is not ``more difficult'' than that. 

\subsection{Auxiliary problem: \texorpdfstring{\gls*{acr:dpa}}{DPA} in the ground space}
Before presenting our main results, we introduce an auxiliary optimal control problem in the ground space:
\begin{equation*}
\costtogosmall{}(x, \refVar_0, \ldots, \refVar_N) 
= 
\begin{aligned}[t]
\inf_{\substack{
x_k\in\statespace{k}
\\
\inputVariable{}{}{}{k} \in \inputspace{k}}}
&
\terminalcostsmall(x_N, \refVar_N)
+
\sum_{k = 0}^{N - 1}
\stagecostsmall{k}(x_k, \inputVariable{}{}{}{k}, \refVar_k)
\\
\text{s.t.  } &x_{k+1}=\dynamics{k}(x_k,\inputVariable{}{}{}{k}), \quad x_0=x. 
\end{aligned}
\end{equation*}
Similarly to \eqref{equation:dynamicprogrammingalgorithm}, the \gls*{acr:dpa} provides the \emph{cost-to-go} $\costtogosmall{k}: \statespace{k}\times \refSpace{k}\times\ldots\times\refSpace{N}\to \nonnegativeRealsBar$:
\begin{equation}\label{equation:dp:originalspace:extended}
\begin{aligned}
    \costtogosmall{N}(x_N,\refVar_N) &\coloneqq \terminalcostsmall(x_N,\refVar_N);\\
    \costtogosmall{k}(x_k, \refVar_k, \ldots, \refVar_N) 
    &\coloneqq 
    \inf_{\inputVariable{}{}{}{k} \in \inputspace{k}}
    \stagecostsmall{k}(x_k,\inputVariable{}{}{}{k},\refVar_k) + \costtogosmall{k+1}(\dynamics{k}(x_k,\inputVariable{}{}{}{k}), \refVar_{k+1}, \ldots, \refVar_N).
\end{aligned}
\end{equation}
Specifically, we use lower-case $\costtogosmall{k}$ for the cost-to-go in the ground space and upper-case $\costtogo{k}$ for its probability space twin.
By~\eqref{equation:dp:originalspace:extended}, ($\varepsilon$-)optimal inputs will be feedback law $\inputMap{}{}{}{k}:\statespace{k}\times\refSpace{k}\times\ldots\times\refSpace{N}\to\inputspace{k}$. In particular, an input $\inputMap{}{}{}{k}\in\inputspace{k}$ (or, with a slight abuse of notation, a feedback law $\inputMap{}{}{}{k}:\statespace{k}\times\refSpace{k}\times\ldots\times\refSpace{N}\to\inputspace{k}$) is $\varepsilon$-optimal in \eqref{equation:dp:originalspace:extended} if
\begin{equation}
\label{equation:input:epsilon-optimality}
\stagecostsmall{k}(x_k,\inputVariable{}{}{}{k},\refVar_k) + \costtogosmall{k+1}(\dynamics{k}(x_k,\inputVariable{}{}{}{k}), \refVar_{k+1}, \ldots, \refVar_N)
\leq
\costtogosmall{k}(x_k, \refVar_k, \ldots, \refVar_N) + \varepsilon.
\end{equation}
\section{Main Result}
\label{section:lifting}
In this section, we present our main result. We first provide an informal statement in~\cref{subsection:lifting:informal}. The rigorous version is in~\cref{subsection:lifting:rigorous}.

\subsection{A separation principle in the probability space}\label{subsection:lifting:informal}
Our main result predicates a separation principle:

\begin{informaltheorem}%[\gls*{acr:dpa} in probability spaces via optimal transport, informal]
\label{theorem:informal}
Consider the setting of~\cref{problem:finitehorizon:rigorous}. At every stage $k$, the following hold:
\begin{enumerate}[leftmargin=*]
    \item The cost-to-go $\costtogo{k}$ is a multi-marginal optimal transport problem between the current state $\mu_k$ and the future references $\refPVar_k,\ldots,\refPVar_N$, with transportation cost being the cost-to-go in the ground space $\costtogosmall{k}$. 
    
    \item The optimal state-input distribution $\probabilityInput{\ast}{k}$ results from the following strategy:
    \begin{enumerate}[label=(\arabic*),leftmargin=*]
        \item Find the optimal input $\inputVariable{}{}{\ast}{k}$ in the ground space.
        
        \item Find the optimal transport plan $\plan{\gamma}_k^\ast$ for the cost-to-go $\costtogo{k}$.
        
        \item Dispatch the particles as prescribed by $\plan{\gamma}_k^\ast$, and apply $\inputVariable{}{}{\ast}{k}$ to steer them to their allocated trajectory. 
    \end{enumerate}
\end{enumerate}
\end{informaltheorem}

In words, to solve~\gls*{acr:dpa} in probability spaces, we first solve for the cost-to-go $\costtogosmall{k}$ in the ground space and then construct a multi-marginal optimal transport problem with transportation cost $\costtogosmall{k}$. Moreover, the optimal input for a fleet of identical agents results from the composition of the optimal control strategy for \emph{each} individual agent
% (how to optimally follow the trajectory $\refVar_k,\ldots, \refVar_N$ for an agent with state $x_k$?)
(what is the optimal feedback law for an agent at $x_k$ that follows the trajectory $\refVar_k,\ldots, \refVar_N$?)
and the solution of a multi-marginal optimal transport problem (who has state $x_k$ and follows the trajectory $\refVar_k, \ldots, \refVar_N$?).
Importantly, our result reveals a separation principle: It is optimal to first devise low-level controllers for individual agents (i.e., $\inputMap{}{}{\ast}{k}$) and then solve an assignment problem to allocate agents to their destinations (i.e., $\plan{\gamma}_k^\ast$).

\subsection{A rigorous statement}\label{subsection:lifting:rigorous}
Next, we rigorously formalize the statements in~\cref{subsection:lifting:informal}.

\begin{theorem}[\gls*{acr:dpa} in probability spaces via optimal transport]\label{theorem:lifting}
Consider the setting of~\cref{problem:finitehorizon:rigorous}.
%and let \cref{hypothesis:lifting} hold. 
At every stage $k$, the following hold: 
\begin{enumerate}[leftmargin=*]
\item
The cost-to-go equals the multi-marginal optimal transport discrepancy
\begin{equation}\label{equation:lifting:optimal-transport}
\begin{aligned}
&\costtogo{k}(\mu_k, \refPVar_k,\ldots,\refPVar_N)
= 
\kantorovich{\costtogosmall{k}}{\mu_k}{\refPVar_k,\ldots,\refPVar_N}{}
\\
&= 
\inf_{\plan{\gamma} \in \setPlans{\mu_k}{\refPVar_k,\ldots,\refPVar_N}}
\int_{\statespace{k}\times\refSpace{k}\times\ldots\times\refSpace{N}}
\costtogosmall{k}(x_k,\refVar_k,\ldots,\refVar_N)
\d
\plan{\gamma}(x_k,\refVar_k,\ldots,\refVar_N),
\end{aligned}
\end{equation}
where $\costtogosmall{k}$ is the cost-to-go in the ground space, as in \eqref{equation:dp:originalspace:extended}.
Moreover, the \gls*{acr:dpa} yields the optimal solution $\costtogo{} = \costtogo{0}$.
\item
For $\varepsilon \geq 0$, suppose $\inputMap{}{}{\varepsilon/2}{k}: \statespace{k}\times\refSpace{k}\times\ldots\times\refSpace{N}\to\inputspace{k}$ and $\epsilonPlan{\gamma}{\varepsilon/2}{k} \in \setPlans{\mu_k}{\refPVar_k,\ldots,\refPVar_N}$ are $\frac{\varepsilon}{2}$-optimal in \eqref{equation:dp:originalspace:extended} and \eqref{equation:lifting:optimal-transport}, respectively.
Then 
\begin{equation}\label{equation:lifting:optimal-transport:input}
\probabilityInput{\varepsilon}{k} = \pushforward{\left(\projectionFromTo{\statespace{k}\times\refSpace{k}\times\ldots\times\refSpace{N}}{\statespace{k}}, \inputMap{}{}{\varepsilon/2}{k}\right)}{\epsilonPlan{\gamma}{\varepsilon/2}{k}}
\end{equation}
is an $\varepsilon$-optimal state-input distribution.
If $\varepsilon=0$, then $\probabilityInput{\ast}{k} \coloneqq \probabilityInput{\varepsilon}{k}$ is optimal. 
\item
If $\epsilonPlan{\gamma}{\varepsilon/2}{k}$ in (ii) is induced by a transport map $\transportMap{\varepsilon/2}{k}: \statespace{k}\to\refSpace{k}\times\ldots\times\refSpace{N}$, the $\varepsilon$-optimal control input reads as $\probabilityInput{\varepsilon}{k} = \pushforward{(\identityOn{\statespace{k}}, \comp{(\identityOn{\statespace{k}}, \transportMap{\varepsilon/2}{k})}{\inputMap{}{}{\varepsilon/2}{k}})}{\mu_k}$.
\end{enumerate}
\end{theorem}

% \begin{remark}
% In many practical settings (e.g., in the \gls*{acr:lqr} formulation), the reference at time $k$ does not affect the optimal input at time $k$. In such cases, the optimization in the ground space involves only the residual trajectory $\refVar_{k+1}, \refVar_{k+2}, \ldots, \refVar_N$. A similar effect can be observed in the probability spaces, as the reference probability measure $\refPVar_k$ will be coupled only with $\mu_k$ and thus the computation of the optimal control law $\lambda_k$ amounts to solving a $(N - k + 1)$-marginals optimal transport problem, dropping the marginal $\refPVar_k$.
% Instead, when the stage costs also depend on the next state $\mu_{k+1}$ via an optimal transport discrepancy, the price to pay is an additional marginal. 
% \end{remark}

Before discussing~\cref{theorem:lifting} and its implications, we consider the special case when the stage costs $g_k$ do not depend on the reference; i.e., $\stagecostsmall{k}: \statespace{k}\times\inputspace{k}\to\nonnegativeRealsBar$. For instance, any shortest path problem on a graph can be converted into a finite-horizon optimal control problem (see, e.g., \cite{Bertsekas2017}), where the weights of the edges determine the stage costs $\stagecostsmall{k}$; these depend only on the pair $(x_k, \inputMap{}{}{}{k})$. In these cases, the \gls*{acr:dpa} reads
\begin{equation}\label{equation:dp:originalspace}
\begin{aligned}
    \costtogosmall{N}(x_N, \refVar_N) &\coloneqq \terminalcostsmall(x_N,\refVar_N);\\
    \costtogosmall{k}(x_k, \refVar_N) &\coloneqq \inf_{\inputMap{}{}{}{k}\in \inputspace{k}}\stagecostsmall{k}(x_k,\inputMap{}{}{}{k}) + \costtogosmall{k+1}(\dynamics{k}(x_k,\inputMap{}{}{}{k}), \refVar_N).
\end{aligned}
\end{equation}
Accordingly, the ground space $\frac{\varepsilon}{2}$-optimal input is of the form $\inputMap{}{}{\varepsilon/2}{k}:\statespace{k}\times\refSpace{N}\to\inputspace{k}$ and the cost-to-go $\costtogo{k}$ simplifies to a two-marginals optimal transport discrepancy: 

\begin{corollary}[When two marginals are all you need]\label{corollary:lifting:ot}
Consider the setting of \cref{theorem:lifting}, with $\stagecostsmall{k}:\statespace{k}\times\inputspace{k}\to\nonnegativeRealsBar$. At every stage $k$, the following hold: 
% If there exist maps $\tilde{\stagecostsmall{k}}:\statespace{k}\times\inputspace{k}\to\nonnegativeRealsBar$ such that $\stagecostsmall{k} = \comp{(\projectionFromTo{\statespace{k}\times\inputspace{k}\times\refSpace{k}}{\statespace{k}\times\inputspace{k}})}{\tilde{\stagecostsmall{k}}}$, then:
\begin{enumerate}[leftmargin=*]
\item 
The cost-to-go equals the optimal transport discrepancy
\begin{equation}\label{equation:corollary:lifting:ot}
\costtogo{k}(\mu_k, \refPVar_N)
= 
\kantorovich{\costtogosmall{k}}{\mu_k}{\refPVar_N}{} 
= 
\inf_{\plan{\gamma} \in \setPlans{\mu_k}{\refPVar_N}}
\int_{\statespace{k}\times\refSpace{N}}
\costtogosmall{k}(x_k,\refVar_N)
\d
\plan{\gamma}(x_k,\refVar_N),
\end{equation}
where $\costtogosmall{k}$ is the cost-to-go in the ground space, as in \eqref{equation:dp:originalspace}.
Moreover, the \gls*{acr:dpa} yields the optimal solution $\costtogo{} = \costtogo{0}$.
\item 
For $\varepsilon \geq 0$, suppose $\inputMap{}{}{\varepsilon/2}{k}: \statespace{k}\times\refSpace{N}\to\inputspace{k}$ and $\epsilonPlan{\gamma}{\varepsilon/2}{k} \in \setPlans{\mu_k}{\refPVar_N}$ are $\frac{\varepsilon}{2}$-optimal in \eqref{equation:dp:originalspace} and \eqref{equation:lifting:optimal-transport}, respectively.
Then
\begin{equation}\label{equation:corollary:lifting:ot:input}
\probabilityInput{\varepsilon}{k} = \pushforward{\left(\projectionFromTo{\statespace{k}\times\refSpace{N}}{\statespace{k}}, \inputMap{}{}{\varepsilon/2}{k}\right)}{\epsilonPlan{\gamma}{\varepsilon/2}{k}}
\end{equation}
is an $\varepsilon$-optimal state-input distribution.
If $\varepsilon=0$, then $\probabilityInput{\ast}{k} \coloneqq \probabilityInput{\varepsilon}{k}$ is optimal.
\item If $\epsilonPlan{\gamma}{\varepsilon}{k}$ in (ii) is induced by a transport map $\transportMap{\varepsilon/2}{k}: \statespace{k}\to\refSpace{N}$, the $\varepsilon$-optimal control input reads as $\probabilityInput{\varepsilon}{k} = \pushforward{(\identityOn{\statespace{k}}, \comp{(\identityOn{\statespace{k}}, \transportMap{\varepsilon/2}{k})}{\inputMap{}{}{\varepsilon/2}{k}})}{\mu_k}$.
\end{enumerate}
\end{corollary}
% \begin{remark}
%     More generally, with $\mathcal{K}_k$ being the indices of the stage costs $\stagecostsmall{h}, \inIndexSet{h}{k}{N}$, which depend on the last argument, the cost-to-go $\costtogo{k}$ is a multi-marginal transport problem with $2+|\mathcal{K}_k|$ marginals.
% \end{remark}

We defer the proofs of these results to~\cref{section:proofs}.
%First, we discuss them in words (\cref{section:discussion}) and present some preliminary technical results (\cref{subsec:technical results}).

\subsection*{Discussion}\label{section:discussion}
A few comments on our results are in order. 

\begin{paragraph}{How does one construct optimal state-input distributions?}
We start with more details on~\eqref{equation:lifting:optimal-transport:input} and~\eqref{equation:corollary:lifting:ot:input}.
For simplicity, assume that an optimal input map $\inputMap{}{}{\ast}{k}$ and an optimal transport plan $\epsilonPlan{\gamma}{\ast}{k}$ exist (else, resort to an $\varepsilon$ argument). Then (ii) in \cref{theorem:lifting,corollary:lifting:ot} predicate that an optimal state-input distribution $\probabilityInput{\ast}{k}$ for~\cref{problem:finitehorizon:rigorous} results from the~\gls{acr:dpa} in the ground space (i.e., $\inputMap{}{}{\ast}{k}$) and the solution of an optimal transport problem (i.e., $\epsilonPlan{\gamma}{\ast}{k}$):
\begin{enumerate}[leftmargin=*]
    \item\textit{Optimal particle allocation:}
The transport plan $\epsilonPlan{\gamma}{\ast}{k} \in \setPlans{\mu_k}{\refPVar_k,\ldots,\refPVar_N}$ describes the optimal allocation of the particles throughout the horizon. In discrete instances, $\epsilonPlan{\gamma}{\ast}{k}(x_k, \refVar_k, \ldots, \refVar_N)$ quantifies the share of agents with state $x_k$ that will follow the reference trajectory $\refVar_k,\ldots,\refVar_N$.

% The role of the multi-marginal formulation is further explored in \cref{section:examples}, but consider now the case $\refSpace{h} = \statespace{h}$, $\refPVar_h = \refPVar$, $\forall \inIndexSet{h}{k}{N}$. One can interpret this setting as a steering problem, where $\refPVar$ is the target configuration. A particle $x \in \supp(\mu_k)$ does not change target location throughout the horizon if $\plan{\gamma}(x, y, \dots, y) = \mu_k(x)$ for some $y \in \supp(\refPVar)$. However, such a constraint is not enforced in \cref{theorem:lifting}. That is, there might exists $\indexedVar{y}{k}{N} \in \setProduct{\refSpace{h}}{h}{k}{N}$ and $\inIndexSet{i,j}{k}{N}$ such that $\refVar_i \neq \refVar_j$ but $\plan{\gamma'}(x, \indexedVar{y}{k}{N}) > 0$. That is, a ``re-allocation'' of $x$ from $\refVar_i$ to $\refVar_j$ might be required to obtain optimality. Macroscopically, it does not make a difference which is the particle $x$ that is compared to which $y$. Hence, even if the target configuration $\refPVar$ does not change, the possibility of re-allocation must be considered. Instead, in the settings of \cref{corollary:lifting:ot}, such re-allocation is never necessary, and $\plan{\gamma}(x, y)$ describes the number of particles at $x$ aiming at $y$.
    \item\textit{Optimal input coupling:}
% The second ingredient is $\plan{\gamma} \in \Pp{}{\statespace{k}\times\setProduct{\refSpace{i}}{i}{k}{N}\times\inputspace{k}}$, which is the result of the composition of $\plan{\gamma'}$ with the ($\varepsilon$-)optimal low level control law $u_{xy}$. Roughly speaking, $\plan{\gamma}$ assigns probability mass to $(x, u, \indexedVar{y}{k}{N})$ only if applying $u$ to a particle at the state $x$ is optimal to minimize the trajectory cost concerning $\refVar_k, \refVar_{k+1}, \ldots, \refVar_N$. 
%     \item[]\textit{Interpretation of $\lambda_k$.}
Accordingly, we can interpret $\probabilityInput{\ast}{k}$ as the number of particles at $x_k$ that apply the optimal input $\inputMap{}{}{\ast}{k}(x_k,\refVar_k,\ldots,\refVar_N)$.
Intuitively, $\probabilityInput{\ast}{k}$ assigns probability mass to $(x_k, \inputVariable{}{}{}{k})$ if there is a trajectory $\refVar_k, \refVar_{k+1}, \ldots, \refVar_N$ to which $x_k$ has been allocated by $\epsilonPlan{\gamma}{\ast}{k}$, such that $\inputVariable{}{}{\ast}{k}$ is the optimal input to minimize the cost along that trajectory.
% \cref{theorem:lifting} then shows this is indeed optimal. 
%Moreover, \cref{theorem:lifting} effectively provides a recipe to construct $\lambda_k$ for any $\mu_k$. Call any such object $\lambda_{\mu_k}$ to highlight this dependency. Then, one can construct a \emph{feedback} law $\tilde{\lambda}_k: \Pp{}{\statespace{k}} \to \cup_{\mu_k \in \Pp{}{\statespace{k}}} \probabilityInputSpace{k}(\mu_k)$ defined as $\Pp{}{\statespace{k}}\ni \mu_k \mapsto \lambda_{\mu_k}$.
\end{enumerate}
\end{paragraph}

\begin{paragraph}{Existence of optimal solutions}
In turn, our results provide sufficient conditions for the existence of an optimal solution for \cref{problem:finitehorizon:rigorous}: existence of a solution for both the \gls*{acr:dpa} in the ground space and the associated optimal transport problem.
\end{paragraph}

\begin{paragraph}{Existence of optimal input maps}
An optimal solution to \cref{equation:dp:originalspace:extended} always exists when all spaces are finite or when for any $K \subseteq \statespace{k} \times \refSpace{k}\times\ldots\times\refSpace{N}$ compact and $L > 0$ the sets
$
\{\inputVariable{}{}{}{k} \in \inputspace{k} \st \stagecostsmall{k}(x_k, \inputVariable{}{}{}{k}, \refVar_k) + \costtogosmall{k+1}(\dynamics{k}(x_k, \inputVariable{}{}{}{k}), \refVar_{k+1}, \ldots, \refVar_N) \leq L,\forall (x_k, y) \in K\}
$
are compact, the maps $\stagecostsmall{k},\terminalcostsmall$ are lower semicontinuous, and $\dynamics{k}(x_k, \cdot)$ are continuous for all $x_k \in \statespace{k}$; see \cite[Proposition 4.2.2]{Bertsekas2014} and \cite[Theorem 18.19]{Aliprantis2006}.
In general, however, optimal input may not exist. 
For this reason, we state our results using $\varepsilon$-optimality.
\end{paragraph}

\begin{paragraph}{Existence of optimal transport maps}
If the solution of the optimal transport problem is a transport map, then (iii) in~\cref{theorem:lifting} suggests that the optimal input is deterministic. 
Without aims of completeness, this is the case when the following hold: 
\begin{enumerate}[leftmargin=*]
    \item the marginals are empirical with the same number of particles (in virtue of the Birkhoff theorem~\cite[Theorem 6.0.1]{Ambrosio2008}); or
    %\item The marginals are absolutely continuous, the ground spaces $\statespace{k} \subset \reals^{n_k}$ are open and bounded, and the cost-to-go $\costtogosmall{k}$ is continuous \cite{Brenier1991}; or
    \item the cost-to-go $\costtogosmall{k}$ is continuous and semiconcave, and for each $x_k\in\statespace{k}$ the map $(\refVar_k, \ldots, \refVar_N)\mapsto\frac{\partial{\costtogosmall{k}}}{\partial x_k}(x_k, \refVar_k, \ldots, \refVar_N)$ is injective in its domain of definition intersected with splitting sets~\cite[Definition 2.4]{Kim2014}, and $\mu_k$ is absolutely continuous~\cite{Kim2014,Pass2015} (see~\cite[Theorem 1.2]{Figalli2007} for the case with two marginals). 
\end{enumerate}
%Furthermore, one can always resort to $\varepsilon$-optimal solutions of the Monge formulation.
\end{paragraph}

\begin{paragraph}{Connections to previous work}
The approach in the literature for distribution/fleet steering is fundamentally different from ours: It is a priori stipulated that the steering problem is an optimal transport problem from an initial distribution to a target one, without formulating an optimal control problem in probability spaces. This way, the complexity of~\gls*{acr:dpa} probability spaces is bypassed, at the price, however, of potentially suboptimal solutions: There is no reason for this approach to be optimal for a corresponding control problem in the probability space.
%In the literature for distribution/fleet steering, the steering problem is not formulated as an optimal control problem in probability spaces, but rather as an optimal transport problem from an initial distribution to a target one. Such an approach bypasses the complexity of~\gls*{acr:dpa} in probability spaces, at the price, however, of potentially suboptimal solutions: There is no reason for it to yield optimality for a corresponding control problem in the probability space.
%Our methodology, instead, is fundamentally different: We study~\gls*{acr:dpa} in the probability space. 
With~\cref{theorem:lifting,corollary:lifting:ot}, we show that, provided the transportation cost is judiciously chosen, this approach is optimal and yields the same solution as the~\gls*{acr:dpa} in probability spaces.
For instance, the results in \cite[\S A]{HudobadeBadyn2021} correspond to the optimal strategy when $\stagecostsmall{k}(x_k, \inputVariable{}{}{}{k}, \refVar_k) = \norm{\inputVariable{}{}{}{k}}^2$, and terminal constraint on the final distribution (see \cref{section:problem:cost}). The results in \cite{HudobadeBadyn2021} can thus be extended to more general terminal costs (e.g., $\terminalcostsmall(x_N, \refVar_N) = \norm{x_N - \refVar_N}^2$).
Instead, the results in \cite[\S B]{HudobadeBadyn2021} are suboptimal in the sense of the~\gls*{acr:dpa} in the probability space. By~\cref{theorem:lifting}, when the stage costs are reference-dependent (e.g, $\stagecostsmall{k}(x_k, \inputVariable{}{}{}{k}, \refVar_k) = \norm{\inputVariable{}{}{}{k}}^2 + \norm{x_k - \refVar_k}^2$), the cost-to-go results from a multi-marginal optimal transport problem. As such, the strategy proposed in \cite{HudobadeBadyn2021} does not minimize, at every time-step $k$, the weighted sum of the squared Wasserstein distance from the target configuration and the input effort.
Similarly, the problem formulation in \cite{Krishnan2019} can be recovered with integrator dynamics $\dynamics{k}(x_k,\inputVariable{}{}{}{k})=x_k+\inputVariable{}{}{}{k}$, cost $\stagecostsmall{k}(x_k, \inputVariable{}{}{}{k}, \refVar_k) = \norm{\inputVariable{}{}{}{k}}^2$ and terminal constraint on the final distribution (see \cref{section:problem:cost}).
With a state augmentation (the input used along the trajectory, an independent integrator dynamics) and input constraints as suggested in \cref{section:problem:cost}, \cite[Problem 2]{Balci2021} is a special case of our setting, with linear dynamics (see \cref{section:problem:dynamics}), stage cost $\stagecostsmall{k} \equiv 0$, and terminal cost the squared Wasserstein distance; i.e., $\terminalcostsmall(x_N, \refVar_N) = \norm{x_N - \refVar_N}^2$. Simple calculations reveal that the hard-constrained covariance formulation in \cite[Problem 1]{Balci2021} can be reformulated via a hard terminal constraint on the final probability measure (a Gaussian probability measure with appropriate covariance). In both cases, such specializations are possible because the authors restrict themselves to the Gaussian and linear setting. In general, covariance constraints or penalties require further study; see \cref{section:problem:cost}. Similarly, noisy settings do not immediately benefit from our reformulation; see \cref{section:examples:negative:dynamics}. Analogous considerations hold for \cite{Bakolas2016,Bakolas2017,Bakolas2018b,Bakolas2018}.
% \end{itemize}
\end{paragraph}

\begin{paragraph}{Design of transportation costs}
In many disciplines, the design of transportation costs is challenging; see, e.g., \cite{Scarvelis2022,Terpin2022}. 
For instance, in~\cite{Scarvelis2022}, the underlying Riemannian metric characterizing the trajectory of single-cell RNA is retrieved in a data-driven fashion. 
\cref{theorem:lifting,corollary:lifting:ot} suggest an alternative approach: first, ``learn'' the cost-to-go for single particles and then use it as the transportation cost.
\end{paragraph}

\begin{paragraph}{Measurability issues}
In general, the cost-to-go $\costtogosmall{k}$ fails to be Borel (see, e.g.,~\cite[\S 8.2, Example 1]{Bertsekas1996}). Nonetheless, with our assumptions, it is lower semi-analytic~\cite[Corollary 8.2.1]{Bertsekas1996} and, thus, the integral in~\eqref{equation:lifting:optimal-transport} is well-defined~\cite[\S 7.7]{Bertsekas1996}. Similarly, for any $\varepsilon > 0$, the inputs $\inputMap{}{}{\varepsilon/2}{k}$ may fail to be Borel measurable or $\epsilonPlan{\gamma}{\varepsilon/2}{k}$-measurable but are only universally measurable~\cite[Proposition 7.50]{Bertsekas1996}. Non-measurability may raise concerns for the pushforward operation (only defined for Borel or $\epsilonPlan{\gamma}{\varepsilon/2}{k}$-measurable maps) in~\eqref{equation:corollary:lifting:ot:input}.
However, for any Borel measure $\epsilonPlan{\gamma}{\varepsilon/2}{k}$ there exists a Borel map $\tilde{u}: \statespace{k}\times\refSpace{k}\times\ldots\times\refSpace{N}\to\inputspace{k}$ so that $\inputMap{x_k}{\refVar_k, \ldots, \refVar_N}{\varepsilon/2}{k}$ and  $\tilde{u}(x_k, \refVar_k, \ldots, \refVar_N)$ achieve the same cost in~\eqref{equation:dp:originalspace} $\epsilonPlan{\gamma}{\varepsilon/2}{k}$-a.e.\cite[Lemma 7.27]{Bertsekas1996}. Therefore, we can without loss of generality assume that $\inputMap{}{}{\varepsilon/2}{k}$ is Borel. This way, the pushforward operation in~\eqref{equation:corollary:lifting:ot:input} is well-defined. 
\end{paragraph}

\subsection*{Computational aspects}
\label{section:computational-aspects}
%Separation principles, ubiquitous in control theory, provide numerous benefits in terms of modularity and simplicity. 
We now investigate the computational aspects of our results in the context of multi-agent systems.
We argue that our separation principle
\begin{enumerate}[leftmargin=*]
\item renders computationally feasible otherwise infeasible multi-agent settings and
\item balances offline and online computational requirements, providing additional efficiency and adaptivity to fleet changes.
\end{enumerate}
Consider the setting of~\cref{corollary:lifting:ot} in finite spaces: Let $|\statespace{}|$ be the number of states in the ground space, $N$ the horizon length, and $|\inputspace{}|$ the number of available actions at each state.
We compare the \gls{acr:dpa} in probability and the recipe in~\cref{corollary:lifting:ot} in terms of \emph{offline} and \emph{online} computational effort.
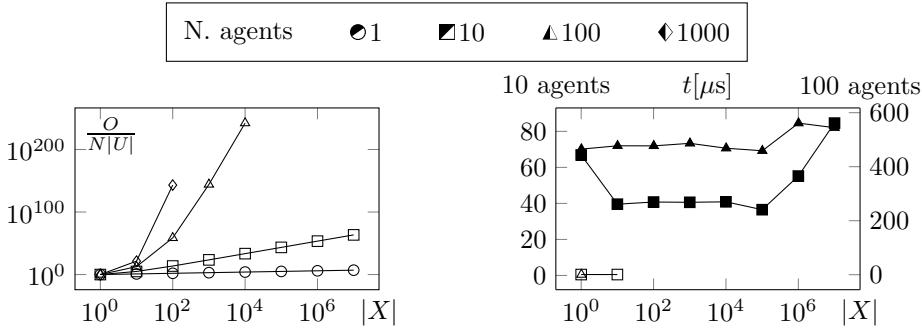
\begin{figure*}[t!]
\centering
\begin{tikzpicture}[every text node part/.style={align=center}]
\matrix [draw, column sep=15] {
% |[multicol=4]| N. agents
% \\
  \node [] {N. agents};
  &
  \node [label={[xshift=-.2cm, yshift=0cm]right:1}] {\statcirc[white]{black}};
  &
  \node [label={[xshift=-.2cm, yshift=0cm]right:10}] {\statrect[white]{black}};
  &
  \node [label={[xshift=-.2cm, yshift=0cm]right:100}] {\stattrian[white]{black}};
  &
  \node [label={[xshift=-.2cm, yshift=0cm]right:1000}] {\statdiam[white]{black}};\\
};
\end{tikzpicture}
\\
\begin{subfigure}[t]{0.48\textwidth}
\centering
\begin{tikzpicture}
\begin{axis}[
    xmode=log,
    ymode=log,
    xlabel={$|\statespace{}|$},
    ylabel={$\frac{O}{N|\inputspace{}|}$},
    ytick={1,1e100,1e200,1e300},
    x label style={at={(axis description cs:1,0)},anchor=north},
    y label style={at={(axis description cs:0,1)},rotate=-90,anchor=north west},
    height=4cm,
    width=.9\textwidth,
]
\addlegendimage{empty legend}
% \addplot[black,mark=*] table [x=n_states, y=1, col sep=comma] {data/ablation_offline.csv};
\addplot[black,mark=o] table [x=n_states, y=1, col sep=comma] {data/ablation_offline.csv};
\addplot[black,mark=square] table [x=n_states, y=10, col sep=comma] {data/ablation_offline.csv};
\addplot[black,mark=triangle] table [x=n_states, y=100, col sep=comma] {data/ablation_offline.csv};
\addplot[black,mark=diamond] table [x=n_states, y=1000, col sep=comma] {data/ablation_offline.csv};
% \addplot[black,mark=triangle*] table [x=n_states, y=10000, col sep=comma] {data/ablation_offline.csv};
% \legend{N. agents, $1$, $10$, $100$, $1000$} % $10000$
\end{axis}
\end{tikzpicture}
\caption{\emph{Offline} requirements. Computational effort of~\gls{acr:dpa} in probability spaces for various problem sizes and number of agents. 
The recipe in \cref{corollary:lifting:ot} always requires the computational effort of the single-agent case.
}
\label{fig:computation:offline}
\end{subfigure}%
\hfill
\begin{subfigure}[t]{0.48\textwidth}
\centering
\begin{tikzpicture}
\begin{axis}[
    xmode=log,
    xlabel={$|\statespace{}|$},
    ylabel style={align=center}, ylabel={$10$ agents},
    scaled ticks=false,
    axis y line*=left,
    x label style={at={(axis description cs:1,0)},anchor=north},
    y label style={at={(axis description cs:0,1)},rotate=-90,anchor=south},
    height=4cm,
    width=.9\textwidth,
    title={$t[\mathrm{\mu s}]$},
    title style={yshift=-1.5ex}
]
\addplot[black,mark=square] table [x=n_states, y expr=\thisrow{10_coupled}*1000000, col sep=comma] {data/ablation_online.csv};
\addplot[black,mark=square*] table [x=n_states, y expr=\thisrow{10_decoupled}*1000000, col sep=comma] {data/ablation_online.csv};
% \addplot[black,mark=diamond*] table [x=n_states, y=1000_decoupled, col sep=comma] {data/ablation_online.csv};
% \legend{N. agents, $1$ coupled, $10$ coupled, $100$ coupled, $1000$ coupled, $1$ decoupled, $10$ decoupled, $100$ decoupled, $1000$ decoupled} % $10000$
\end{axis}
\begin{axis}[
    xmode=log,
    xlabel={$|\statespace{}|$},
    ylabel style={align=center}, ylabel={$100$ agents},
    scaled ticks=false,
    axis y line*=right,
    axis x line=none,
    y label style={at={(axis description cs:1,1)},rotate=-90,anchor=south},
    height=4cm,
    width=.9\textwidth,
]
\addplot[black,mark=triangle] table [x=n_states, y expr=\thisrow{100_coupled}*1000000, col sep=comma] {data/ablation_online.csv};
\addplot[black,mark=triangle*] table [x=n_states, y expr=\thisrow{100_decoupled}*1000000, col sep=comma] {data/ablation_online.csv};
\end{axis}
\end{tikzpicture}
\caption{\emph{Online} requirements. Computational time for~\gls{acr:dpa} in probability spaces (empty marker) and the recipe in \cref{corollary:lifting:ot} (filled marker) in settings with $10$ agents (left $y$-axis) and $100$ agents (right $y$-axis). 
%Left $y$-axis for $10$ agents, right for $100$.
%The numbers are averaged on thousand different samples.
}
\label{fig:computation:online}
\end{subfigure}%
% \ifbool{arxiv}{}{\vspace{-.6cm}}
\caption{
Sensitivity of \emph{offline} and \emph{online} computational effort to the number of agents $M$ and states $|\statespace{}|$ for the \gls*{acr:dpa} in probability spaces (empty marker) and \cref{corollary:lifting:ot} (filled marker), with $|\inputspace{}|, N, O$ being the number of actions, time-steps, and operations, respectively. 
We omit markers when numbers exceed the IEEE 754 floating point representation.
The \gls{acr:dpa} in probability spaces is infeasible already for a small number of agents. The recipe in \cref{corollary:lifting:ot}, instead, remains feasible for large fleet sizes. 
}
\label{fig:computation}
% \ifbool{arxiv}{}{\vspace{-0.7cm}}
\end{figure*}

\begin{paragraph}{Offline computation}
The \gls*{acr:dpa} in the ground space %(for all initial and terminal states) 
has a time and memory complexity of 
$
\mathcal{O}\left(|\statespace{}|N|\inputspace{}|\right)
$, since the computation of the feedback law requires testing all feasible inputs $u\in U$ (and choosing the optimal one) for each state $x \in \statespace{}$ and at every of the $N$ time-steps. If we normalize by the constants in the $\mathcal{O}$-notation and $N|\inputspace{}|$, the number of operations is proportional to the number of states. At the fleet level, this number amounts to the number of configurations of the fleet. If we restrict ourselves to empirical probability measures with $M$ particles, which can be written as $\mu_k = \frac{1}{M}\sum_{i = 1}^M \delta_{x^{(i)}_k}$ for $x^{(i)}_k \in \statespace{}, \inIndexSet{i}{1}{M}$, the number of such configurations is 
$
\binom{{M + |\statespace{}| - 1}}{M} = \frac{(M + |\statespace{}| - 1)!}{M!(|\statespace{}| - 1)!}
$
(cf. \cite[Theorem 2.6.3 and Example 2.7.9, Case 4]{Combinatorics2019}).
%In particular, for $M = 1$ this number reduces to $|\statespace{}|$.
In particular, note the factorial growth in the number of agents. 
%The resulting ablation study in \cref{fig:computation:offline} reveals the driving factor to be the number of agents. 
For instance, the control problems in the benchmark dataset~\cite{RLBenchmark2023} have  $10 \leq |\statespace{}| \leq 1\mathrm{e}7$, $3 \leq |\inputspace{}| \leq 18$, and $10 \leq N \leq 200$. Hardware capable of trillions of operations per second solves the hardest instance in a few milliseconds. As shown in \cref{fig:computation:offline}, the \gls*{acr:dpa} in probability spaces would be unfeasible already for $M = 10$ and, for the easiest instance, with $M = 1000$.
\end{paragraph}

\begin{paragraph}{Online computation}
As long as the dynamic programming table in probability spaces has a reasonable size (which, as discussed above, is not the case already for a few states and agents), the cost amounts to the lookup\footnote{While in principle the retrieval from a table has a constant time, with very high probability for large query domains (as in the multi-agent case) it is logarithmic in the table size~\cite[\S 11]{Cormen2009}.} of the optimal state-input distribution. Instead, the recipe in~\cref{corollary:lifting:ot}  requires solving an optimal transport problem (i.e., an allocation problem), which with efficient solvers is real-time feasible even in large instances (e.g., a few milliseconds for thousands of agents with~\cite{Flamary2021})\footnote{In this case, the retrieval, needed to evaluate assignment costs, has to be done on a much smaller table, as discussed in \emph{Offline computation}.}. %In the context of multi-agent systems, in virtue of Birkhoff theorem~\cite[Theorem 6.0.1]{Ambrosio2008}, it is convenient to consider the current and target fleet configurations to be uniform (with a simple state-augmentation), so that a transport map is guaranteed to exist.
% $
% \min_{\plan{\gamma}_k^{(ij)} \geq 0} 
% \sum_{i = 1}^M
% \sum_{j = 1}^M
% C_k^{(ij)}
% \plan{\gamma}_k^{(ij)}
% $,
% $
% \sum_{i = 1}^M\plan{\gamma}_k^{(ij)} = \frac{1}{M}
% $
% and
% $
% \sum_{j = 1}^M\plan{\gamma}_k^{(ij)} = \frac{1}{M}
% $,
% so that a transport map is guaranteed to exist. The current and target probability measures are considered uniform possibly duplicating entries in the cost matrix $C_k$, which has the $\nth{i,j}$ entry $C_k^{(ij)} = \costtogosmall{k}(x_k^{(i)}, \refVar_N^{(j)})$, with $\refPVar = \frac{1}{M}\sum_{j = 1}^M\delta_{\refVar_N^{(j)}}$ being the target probability measure.
% Already for a few agents and states, the memory requirements consumer hardware cannot load in memory the data relative to the current time-steps (cf. \cref{fig:computation:online}), and resorting to other (slower) techniques is not suited for real-time robotics applications. On the other hand, even for very large fleet sizes, efficient optimal transport solvers such as \cite{Flamary2021} find a solution in less than a millisecond.
While the computational time depends on the specific instance, we randomly generate a cost-to-go for variable problem sizes, and we juxtapose the timings (averaged over thousand samples) in~\cref{fig:computation:online}. As the~\gls*{acr:dpa} in probability spaces only amounts to a look-up, its deployment entails almost no computational effort (empty marker in the plot), but it is only feasible in instances where the dynamic programming table can be computed (i.e., up to $|X|=100$ for 10 agents and $|X|=10$ for 100 agents). The recipe in~\cref{corollary:lifting:ot}, instead, yields real-time-feasible computational times across all instances (filled markers in the plot), with moderate dependence on the problem size.
\end{paragraph}

Altogether, \cref{corollary:lifting:ot} drastically reduces the \emph{offline} computation at the price of slightly higher \emph{online} computation. Namely, it alleviates the factorial growth in the number of agents of the \gls*{acr:dpa} in probability spaces, which makes already small instances intractable, while only requiring the real-time solution of an allocation problem.
Moreover, the recipe in~\cref{corollary:lifting:ot} does not require additional (offline) computation when the fleet is reconfigured (e.g., an agent is added or removed).

\begin{comment}
%Additional benefits of the recipe provided in \cref{corollary:lifting:ot} include: 
%\begin{enumerate}[leftmargin=*]
% \item
% One may as well compute the feedback law for the fleet completely offline, thus solving all the possible allocation problems; this however defeat the purpose, as the number of such problems explode with increasing size of the fleet.
%\item 
%No additional offline computation is needed for reconfiguration of the fleet (e.g., an additional agent is added or an agent is removed).
%\item
%As only the solution of the \gls*{acr:dpa} for a single agent is needed, we believe that our results can be leveraged in collaborative reinforcement learning~\cite{Abbeel2018,MARL2022}.
%\end{enumerate}

All together, the benefits of \cref{corollary:lifting:ot} increase with $M$ and $|\statespace{}|$. When these numbers are small, one may instead resort to the \gls*{acr:dpa} in probability spaces. For instance, in \cref{example:robots,example:robots:input,example:robots:dynamics,example:robots:cost}, $|\statespace{}| = 3$ and $|\inputspace{}| = 2$, and the offline effort is proportional to $6N$ operations with \cref{corollary:lifting:ot}, and $(M + 2)(M + 1)N$ otherwise: One pays (more) online what saves offline.
\end{comment}
\section{Examples and Pitfalls}
\label{section:examples}
In \cref{section:examples:positive}, we present two examples where two marginals are enough, in line with the existing literature \cite{HudobadeBadyn2021,Krishnan2019}. Then, in \cref{section:examples:negative:costs}, we showcase that, in general, the multi-marginal formulation is necessary. Finally,~\cref{section:examples:negative:dynamics} shows that our results do not readily extend to noisy dynamics. 

\subsection{Examples when two marginals are all you need}\label{section:examples:positive}
We start with an example to which \cref{corollary:lifting:ot} applies.

\begin{example}[Integrator particle dynamics, input effort]\label{example:dynamics-integrator:effort-input}
Suppose we aim at steering a probability measure $\mu_0 \in \spaceProbabilityBorelMeasures{\reals^n}$ to a target $\refPVar_N \in \spaceProbabilityBorelMeasures{\reals^n}$ in $N$ steps; i.e., $\statespace{k}=\refSpace{k}=\reals^n$. The input space is $\inputspace{k}=\reals^n$, and the dynamics are $\dynamics{k}(x_k,\inputVariable{}{}{}{k}) = x_k + \inputVariable{}{}{}{k}$. The costs are $\stagecostsmall{k}(x_k,\inputVariable{}{}{}{k})=\norm{\inputVariable{}{}{}{k}}^2$, and $\terminalcostsmall(x_N,\refVar_N) = 0$ if $x_N = \refVar_N$ and $+\infty$ otherwise, so that the stage cost in the probability space is $\kantorovich{\stagecostsmall{k}}{\probabilityInput{}{k}}{\refPVar_k}{} =\expectedValue{\probabilityInput{}{k}}{\norm{\inputVariable{}{}{}{k}}^2}$, and the terminal cost is $\kantorovich{\terminalcostsmall}{\mu_N}{\refPVar_N}{}=0$ if $\mu_N=\refPVar_N$ and $+\infty$ otherwise. 
The optimal control problem in the ground space admits the solution $\inputMap{x_k}{\refVar_N}{}{k} = \frac{\refVar_N - x_k}{N - k}$, with the associated cost-to-go $\costtogosmall{k}(x_k,\refVar_N) = \frac{N-k}{N^2}\norm{\refVar_N - x_k}^2$.
% \begin{displaymath}
% u_{xy}^k = \begin{cases}
% 0\quad&\text{if }x = y,\\
% \frac{y - x}{N - k}&\text{otherwise}
% \end{cases}
% \quad\text{and}\quad
% \costtogosmall{k}(x,y) = \frac{N-k}{N^2}\norm{y - x}^2.
% \end{displaymath}
By~\cref{corollary:lifting:ot}, the cost-to-go in the space of probability measures $\costtogo{k}$ is
\begin{equation*}
\costtogo{k}(\mu_k, \refPVar_N) 
= 
\min_{\plan{\gamma} \in \setPlans{\mu_k}{\refPVar_N}}
\int_{\reals^{n}\times\reals^n}
\frac{N - k}{N^2}
\norm{\refVar_N - x_k}^2
\d \plan{\gamma}(x_k,\refVar_N)
=
\frac{N - k}{N^2}\wassersteinDistance{2}{\mu_k}{\refPVar_N}^2,
\end{equation*}
% Take $k = 0$ to obtain the optimal cost
% $
% \costtogo{}(\mu_0, \refPVar_N) 
% = 
% \kantorovich{\costtogosmall{0}}{\mu_0}{\refPVar_N}{} 
% =
% \frac{1}{N}\wassersteinDistance{2}{\mu_0}{\refPVar_N}^2
% $.
and the optimal input reads as $\probabilityInput{}{k}=\pushforward{(\projectionFromTo{\statespace{k}\times\refSpace{N}}{\statespace{k}}, \inputMap{}{}{}{k})}{\plan{\gamma}_k}$, where $\plan{\gamma}_k$ is the optimal transport plan for $\costtogo{k}(\mu_k, \refPVar_N)$. In the particular case where an optimal transport map $\transportMap{}{k}:\statespace{k}\to\refSpace{N}$ exists, the optimal input simplifies to
$
\probabilityInput{}{k} = \pushforward{(\identityOn{\statespace{k}}, \inputMap{}{}{}{k}(\cdot,\transportMap{}{k}(\cdot)))}{\mu_k}
$.
That is, all particles having state $x_k$ apply the input $\inputMap{}{}{}{k}(x_k,\transportMap{}{k}(x_k))=\frac{\transportMap{}{k}(x_k) - x_k}{N - k}$.
% For instance, if $\mu_k$ and $\refPVar_N$ are Gaussian with mean $m_k$ and $m$ and covariance $\Sigma_k$ and $\Sigma$, respectively, the optimal transport problem admits an optimal transport map, which can be computed in closed-form \cite{Talagrand1996}. This yields the optimal input
% \begin{displaymath}
% u_k(x) = \frac{T_k(x) - x}{N - k}
% ,\quad\text{with}\quad
% \begin{aligned}
% T_k(x) &= Q_k(x - m_k) + m\\
% Q_k &= \Sigma_k^{-1/2}(\Sigma_k^{1/2}\Sigma\Sigma_k^{1/2})^{1/2}\Sigma_k^{-1/2}
% \end{aligned}
% \end{displaymath}
% and the cost-to-go
% \begin{displaymath}
% \costtogo{k}(\mu_k, \refPVar_N) = \frac{N - k}{N^2}\left(\norm{m - m_k}^2 + \trace\left(\Sigma+\Sigma_k - 2(\Sigma^{1/2}\Sigma_k\Sigma^{1/2})^{1/2}\right)\right).
% \end{displaymath}
% In particular, the optimal input is deterministic at every stage; i.e., $\lambda_k = \pushforward{(\identityOn{\statespace{k}}, u_k)}{\mu_k}$.
\end{example}

Sometimes, the optimal input is probabilistic.

\begin{example}[Sometimes it is necessary to split the mass]
\label{example:split-necessary}
Let $N = 1$, and consider $\statespace{k} = \inputspace{k} = \refSpace{k} = \reals$, $\dynamics{0}(x_0, \inputVariable{}{}{}{0}) = \inputVariable{}{}{}{0}$, $\stagecostsmall{0}(x_0, \inputVariable{}{}{}{0}) = 0$, $\terminalcostsmall(x_1, \refVar_1) = \norm{x_1 - \refVar_1}^2$. Let $\mu_0 = \delta_0$ and $\refPVar_1=\frac{1}{2}\delta_{-1}+\frac{1}{2}\delta_{+1}$. For every pair $(x_0, \refVar_1)$, the solution in the ground space is $\inputMap{x_0}{\refVar_1}{}{0} = \refVar_1$, which yields the cost-to-go $\costtogosmall{0}(x_0, \refVar_1) = 0$. That is, any allocation $\plan{\gamma} \in \setPlans{\delta_0}{\refPVar_N}$ is optimal; in particular, the only feasible plan $\plan{\gamma^\ast} \in \setPlans{\delta_0}{\refPVar_N}$ displaces 50\% of mass to $\refVar_1 = -1$ and the other 50\% of the mass to $\refVar_1 = +1$; see \cref{fig:example:split-necessary}. Then the optimal input reads $\probabilityInput{}{0} = \pushforward{(\projectionFromTo{\statespace{0}\times\refSpace{1}}{\statespace{0}}, \inputMap{}{}{}{0})}{\plan{\gamma^\ast}}$: $50\%$ of the particles apply the input $\inputVariable{}{}{}{0} = +1$ and the others $\inputVariable{}{}{}{0} = -1$. 
\end{example}

\subsection{Why all these marginals?}\label{section:examples:negative:costs}

Hereby, we explore the differences between \cref{theorem:lifting} and \cref{corollary:lifting:ot}. Specifically, we clarify why a \emph{multi-marginal} optimal transport formulation arises, even when the target probability measure remains constant throughout the horizon (i.e.,  $\refPVar_0=\ldots=\refPVar_N\eqqcolon\refPVar$). %In the following, we always consider empirical probability measures $\mu_0$ and $\refPVar$ with the same number of particles. Therefore, in virtue of Birkhoff theorem~\cite[Theorem 6.0.1]{Ambrosio2008}, all optimal transport problems coincide with their Monge formulation (i.e., we can restrict ourselves to transport maps instead of transport plans). 

\begin{counterexample}[Two marginals are not enough]
\label{example:multi-marginal}
Consider, as in \cref{example:robots:dynamics}, $\statespace{k} = \refSpace{k} = \{\pm1, 0\}, \inputspace{k} = \{-1, 0\}$, dynamics $\dynamics{k}(x_k, \inputVariable{}{}{}{k}) = x_k\inputVariable{}{}{}{k}$, with horizon $N = 2$, and costs $\stagecostsmall{k}(x_k,\inputVariable{}{}{}{k},\refVar_k) = \norm{x_k-\refVar_k}^2$ and $\terminalcostsmall(x_N,\refVar_N) = \norm{x_N - \refVar_N}^2$, so that the stage and terminal cost in the probability space are the squared (type 2) Wasserstein distance from the fixed reference measure $\refPVar = \frac{1}{2}(\delta_{-1} + \delta_{+1})$.
First, we utilize~\cref{corollary:lifting:ot}, keeping the reference constant throughout the horizon. 
The cost-to-go $\tilde{\costtogosmall{0}}(\pm 1, \pm 1) = 2$ (here and below, this notation means $\tilde{\costtogosmall{0}}( +1,  +1)=\tilde{\costtogosmall{0}}(- 1, -1) = 2$) and $\tilde{\costtogosmall{0}}(\pm 1, \mp 1) = 6$, both obtained applying at the first stage $\inputVariable{}{}{}{0} = 0$ (and subsequently any input). The cost-to-go for the fleet is $\tilde{\costtogo{0}}(\mu_0, \refPVar) = \kantorovich{\tilde{\costtogosmall{0}}}{\mu_0}{\refPVar}{} = 2$, with the particle having state $x_0 = \pm 1$ allocated to $\refVar_2 = \pm 1$.
However, from a fleet perspective, the input $\inputVariable{}{}{}{k} = -1$ leads to $\mu_0 = \mu_1 = \mu_2 = \refPVar$. By changing allocations throughout the horizon, we obtain a total cost $\costtogo{0}(\mu_0, \refPVar) = 0$.
This behavior emerges naturally with~\cref{theorem:lifting}. The cost-to-go in the ground space satisfies $\costtogosmall{0}(x_0 = \pm 1, \refVar_0 = \pm 1, \refVar_1 = \mp 1, \refVar_2 = \pm 1)=0$, with the input $\inputVariable{}{}{}{k}=-1$ at all times.
Then the transport plan $\plan{\gamma} = \pushforward{(\identityOn{\reals}, \identityOn{\reals}, -\identityOn{\reals}, \identityOn{\reals})}{\mu_0} \in \setPlans{\mu_0}{\refPVar,\refPVar,\refPVar}$ yields
\begin{equation*}
    \costtogo{0}(\mu_0, \refPVar, \refPVar, \refPVar) 
    = \kantorovich{\costtogosmall{0}}{\mu_0}{\refPVar, \refPVar, \refPVar}{}
    %\\&
    \leq
    2\frac{1}{2}\costtogosmall{0}(x_0 = \pm 1, \refVar_0 = \pm 1, \refVar_1 = \mp 1, \refVar_2 = \pm 1) 
    %\\&= 2\left(\norm{x_0 - \refVar_0}^2 + \norm{-x_0 - \refVar_1}^2 + \norm{x_0 - \refVar_2}^2\right)
    = 0,
\end{equation*}
necessarily optimal; see \cref{fig:example:multi-marginal}. In particular, $\costtogo{}(\mu_0, \refPVar, \refPVar, \refPVar) < \tilde{\costtogo{}}(\mu_0, \refPVar)=2$. That is, \cref{corollary:lifting:ot} does not apply and the optimal solution results from~\cref{theorem:lifting}.
\end{counterexample}

\begin{figure*}[t!]
\centering
\begin{subfigure}[t]{0.32\textwidth}
\centering
\begin{tikzpicture}[scale=\ifbool{arxiv}{1.0}{.7}]
%\draw[help lines, color=gray!30, dashed] (-1.9,-.4) grid (1.9,1.9);
\node (x0) at (0,1.5) {};
\node (x1m1) at (-1.5,.75) {};
\node (x1p1) at (1.5,.75) {};
\node[below right] at (0,0) {$x_0$};
\node[below] at (1.5,0) {$+1$};
\node[below] at (-1.5,0) {$-1$};
\draw[->,thick] (-2.5,0)--(2.5,0) node[below]{$x$};
\draw[->,thick] (0,-.5)--(0,2.5) node[right]{$\mu(\{x\})$};
\draw[->, ultra thick] (0,0)--(x0);
\draw[->, ultra thick, dashed] (1.5,0)--(1.5,1.5);
\draw[->, ultra thick, dotted] (-1.5,0)--(-1.5,1.5);
\draw[->, ultra thick] (1.5,0)--(1.5,.75);
\draw[->, ultra thick] (-1.5,0)--(-1.5,.75);
\path[->, thick] (x0) edge[bend right] node[above=.35cm,left=-.35cm]{$\dynamics{0}(x_0, -1)$} (x1m1);
\path[->, thick] (x0) edge[bend left] node[above=.35cm,right=-.35cm]{$\dynamics{0}(x_0, +1)$} (x1p1);
\end{tikzpicture}
\caption{A deterministic input yields either the dotted or the dashed configuration, but it cannot split the probability mass and yield the solid configuration; see \cref{example:split-necessary}.}
\label{fig:example:split-necessary}
\end{subfigure}%
\hfill
\begin{subfigure}[t]{0.32\textwidth}
\centering
\begin{tikzpicture}[scale=\ifbool{arxiv}{1.1}{.8}]
\draw[help lines, color=gray!30, dashed] (-.5,-1.4) grid (2.5,1.4);
\draw[->,thick] (-.5,0)--(2.5,0) node[below]{$k$};
\draw[->,thick] (0,-1.5)--(0,1.5) node[right]{$x$};
\node[circle, fill, minimum size=5pt, inner sep=0pt, outer sep=0pt] (x01) at (0,1) {};
\node[circle, fill, minimum size=5pt, inner sep=0pt, outer sep=0pt] (x02) at (0,-1) {};
\node[circle, draw, minimum size=5pt, inner sep=0pt, outer sep=0pt] (x11) at (1,-1) {};
\node[circle, draw, minimum size=5pt, inner sep=0pt, outer sep=0pt] (x12) at (1,1) {};
\node[circle, draw, minimum size=5pt, inner sep=0pt, outer sep=0pt] (x21) at (2,1) {};
\node[circle, draw, minimum size=5pt, inner sep=0pt, outer sep=0pt] (x22) at (2,-1) {};
\draw[->, thick, dashed] (x01)--(x11);
\draw[->, thick, dashed] (x11)--(x21);
\draw[->, thick, dashed] (x02)--(x12);
\draw[->, thick, dashed] (x12)--(x22);
% Sub-optimal
\node[circle, fill, minimum size=5pt, inner sep=0pt, outer sep=0pt] (subx1) at (1,0) {};
\node[circle, fill, minimum size=5pt, inner sep=0pt, outer sep=0pt] (subx2) at (2,0) {};
\draw[->, thick] (x01)--(subx1);
\draw[->, thick] (x02)--(subx1);
\draw[->, thick] (subx1)--(subx2);
\end{tikzpicture}
\caption{Two switching particles (dashed) yield the optimal configuration from a fleet perspective at all times, oppositely to a fixed allocation (solid); see \cref{example:multi-marginal}.}
\label{fig:example:multi-marginal}
\end{subfigure}
\hfill
\begin{subfigure}[t]{0.32\textwidth}
\centering
\begin{tikzpicture}[scale=\ifbool{arxiv}{1.5}{1.25}]
\node (heart) at (0,1.25) {$-1$};
\node (diamond) at (0,0) {$-2$};
\node (spade) at (2,1.25) {$2$};
\node (club) at (2,0) {$1$};
\node (offset) at (0,-.25) {};
\draw[->, thick, dashed] (heart)--(spade);
\draw[->, thick, dashed] (heart) edge[bend right](club);
\draw[->, thick, dashed] (diamond) edge[bend left] (spade);
\draw[->, thick, dashed] (diamond) -- (club);
\draw[->, thick] (club) -- (heart);
\draw[->, thick] (spade) -- (diamond);
\draw[->, thick, dotted] (club) edge[bend left] (diamond);
\draw[->, thick, dotted] (spade) edge[bend right] (heart);

\draw[->, thick, dotted] (2.2, .8) -- node[above] {$0$} (2.75, .8);
\draw[->, thick, dashed] (2.2, .7) -- (2.75, .7);
\draw[->, thick] (2.2, .25) -- node[above] {$2$} (2.75, .25);
\end{tikzpicture}
\caption{The noisy drift (dashed) may favor a re-allocation with lower effort (dotted), oppositely to the fixed allocation that yields a larger cost (solid) every second time; see \cref{example:lifting:localnoise}.}
\label{fig:example:noise}
\end{subfigure}%
\ifbool{arxiv}{}{\vspace{-.6cm}}
\caption{Depiction of \cref{example:split-necessary}, 
\cref{example:multi-marginal}, and \cref{example:lifting:localnoise}.}
\ifbool{arxiv}{}{\vspace{-0.7cm}}
\end{figure*}
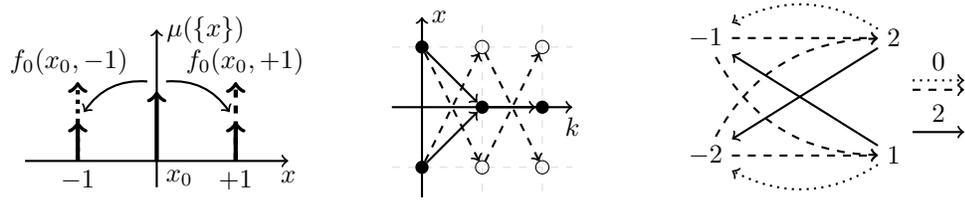

\subsection{The effect of local noise}\label{section:examples:negative:dynamics}
When the particle dynamics are noisy, it is common to minimize the expected particle cost via the \emph{stochastic \gls*{acr:dpa}}:
\begin{equation}\label{equation:dpa:stochastic}
\begin{aligned}
    \costtogosmall{N}(x_N, \refVar_N) &= \terminalcostsmall(x_N, \refVar_N);\\
    \costtogosmall{k}(x_k, \refVar_N) &= \inf_{\inputVariable{}{}{}{k} \in \inputspace{k}}\expectedValue{w_k \sim \xi_k}{\stagecostsmall{k}(x_k, \inputVariable{}{}{}{k}, w_k) + \costtogosmall{k + 1}(\dynamics{k}(x_k, \inputVariable{}{}{}{k}, w_k), \refVar_N)},
\end{aligned}
\end{equation} 
where $\xi_k\in\Pp{}{W_k}$ is the probability measure of the noise, and $W_k$ is the space of possible realizations. 
Since $\costtogosmall{k}$ is of the form required for \cref{corollary:lifting:ot}, it is tempting to extend our results.
Unfortunately, the noisy drift may favor a different allocation of the particles, and the expectation annihilates such an effect.

\begin{counterexample}[\cref{corollary:lifting:ot} does not readily extend]\label{example:lifting:localnoise}
Consider a horizon $N = 2$ and the setting depicted in \cref{fig:example:noise}. Let $\statespace{k}=\refSpace{k} = \inputspace{k} = \{\pm1, \pm2\}$, and consider uniformly distributed noise over $\noisespace{k} = \{1, 2\}$. The particle dynamics is $\dynamics{k}(x_k, \inputVariable{}{}{}{k}, w_k) = w_k$ if $x_k < 0$ and $\dynamics{k}(x_k, \inputVariable{}{}{}{k}, w_k) = \inputVariable{}{}{}{k}$ otherwise.
The stage cost is $\kantorovich{\stagecostsmall{k}}{}{}{}$, where $\stagecostsmall{k}(x_k, -x_k, w_k) = 2$ if $x_k > 0$, and $0$ otherwise. The terminal cost enforces the configuration $\refPVar_N = \frac{1}{2}\delta_{-2} + \frac{1}{2}\delta_{-1}$, namely $\kantorovich{\terminalcostsmall}{}{}{}$ with $\terminalcostsmall(x_N, \refVar_N) = 0$ if $x_N = \refVar_N$ and $+\infty$ otherwise. The recursion~\cref{equation:dpa:stochastic} yields $\costtogosmall{0}(x_0, \refVar_N) = 1$ for $x_0, \refVar_N < 0$ (any input at the first stage and $u_1 = \refVar_N$ at the second stage). Therefore, with the initial configuration and target configuration $\mu_0 = \refPVar_N = \frac{1}{2}\left(\delta_{-2} + \delta_{-1}\right)$, \cref{corollary:lifting:ot} yields $\tilde{\costtogo{0}}(\mu_0, \refPVar_N) = \kantorovich{\costtogosmall{0}}{\mu_0}{\refPVar_N}{} = 1$.
Instead, the \gls*{acr:dpa} in the probability space gives $\mu_1 = \frac{1}{2}\left(\delta_{+1} + \delta_{+2}\right)$ with zero cost (regardless of the input). Then, the evolution is deterministic and the cost-to-go amounts to $\costtogosmall{1}(x_1, \refVar_N) = 0$ unless $\refVar_N = -x_1$. Thus, \cref{corollary:lifting:ot} applies and yields $\costtogo{1}(\mu_1, \refPVar_N)=\kantorovich{\costtogosmall{1}}{\mu_1}{\refPVar_N}{}=\frac{1}{2}\costtogosmall{1}(+1, -2) + \frac{1}{2}\costtogosmall{1}(+2, -1)=0$.
Overall, $\costtogo{0}(\mu_0, \refPVar_N) \leq 0 + \costtogo{1}(\mu_1, \refPVar_N) =0<1=\tilde{\costtogo{0}}(\mu_0, \refPVar_N)$. Thus, the naive application of~\cref{corollary:lifting:ot} is suboptimal.
\end{counterexample}
\section{Proof of~\texorpdfstring{\cref{theorem:lifting,corollary:lifting:ot}}{Theorem 4.1 and Corollary 4.2}}\label{section:proofs}
%\subsection{Technical results}
For the proof~\cref{theorem:lifting,corollary:lifting:ot}, we need a few preliminary lemmata. For ease of notation, let $\statespace{} \coloneqq \statespace{1}\times\ldots\times\statespace{k}$, $\otherspace{} \coloneqq \otherspace{1}\times\ldots\times\otherspace{h}$, and $\anotherspace{} \coloneqq \anotherspace{1}\times\ldots\times\anotherspace{k}$.
To start, we introduce a variation of~\eqref{equation:cost:multimarginaloptimaltransport}, in which only the first $k$ marginals $\mu_i \in \Pp{}{\statespace{i}}$ are fixed. Namely,
\begin{align*}
\freemarginals{c}{\mu_1}{}{\mu_k}
&\coloneqq
\inf_{
\pushforward{\left(\projectionFromTo{\statespace{}\times\otherspace{}}{\statespace{}}\right)}{\plan{\gamma}}
\in \setPlans{\mu_1, \ldots}{\mu_k}
}
\int_{
\statespace{}\times\otherspace{}
}
c
\d\plan{\gamma},
% \\
% \setPlansPartial{\indexedVar{\mu}{1}{k}}{\otherspace{}}
% &\coloneqq
% \left\{
% \plan{\gamma} \in \Pp{}{\statespace{}\times\otherspace{}} 
% \st
% \pushforward{\left(\projectionFromTo{\statespace{}\times\otherspace{}}{\statespace{}}\right)}{\plan{\gamma}}
% =
% \indexedVar{\mu}{1}{k}
% \right\},
\end{align*}
where $c: \statespace{}\times\otherspace{} \to \nonnegativeRealsBar$ is the transportation cost. 
When $c \in \statespace{} \to \nonnegativeRealsBar$ (i.e., there are no free marginals), we conveniently write $\freemarginals{c}{\mu_1}{}{\mu_k} 
=
\kantorovich{c}{\mu_1}{}{\mu_k}$.
Furthermore, given a collection of maps $\{l_k: \statespace{k} \to \anotherspace{k}\}_{k=i}^j$, we denote by 
% \begin{align*}
$
l \coloneqq l_i \times \ldots \times l_{j}
% \\
$
the map $\statespace{i}\times\ldots\times\statespace{j} %&
\to \anotherspace{i}\times\ldots\times\anotherspace{j}$
defined pointwise as
$
(x_i, \ldots, x_j)
% &
\mapsto 
(l_i(x_i), \ldots, l_j(x_j))
$.
% \end{align*}
Given the probability measures $\{\mu_i \in \Pp{}{\statespace{i}}\}_{i = 1}^k$, $\mu \coloneqq (\mu_1,\ldots,\mu_k)$, we conveniently write $\pushforward{l}{\mu} \coloneqq (\pushforward{l_1}{\mu_1}, \ldots, \pushforward{l_k}{\mu_k})$. A measure-valued map $\statespace{} \ni x \mapsto \mu\in\Pp{}{\statespace{}}$ is Borel if and only if, for any Borel set $B \subseteq \statespace{}$, the map $x \mapsto \mu(B)$ is Borel.

In our setting, the cost-to-go will be an optimal transport discrepancy, and the dynamics are a pushforward. To relate the cost-to-go at the $\nth{k}$ stage to the one at the previous time step, we rigorously formalize their interplay. A similar but less general result (i.e., only with two fixed marginals) was derived in the context of uncertainty propagation via optimal transport~\cite{Aolaritei2022}.

\begin{lemma}[Pushforward and optimal transport]\label{lemma:pushforward:stability}
Given a transportation cost $c: \anotherspace{}\times\otherspace{} \to \nonnegativeRealsBar$, $k\in\naturals_{\geq 1}, h \in \naturals$, maps $\{l_{i}: \statespace{i} \to \anotherspace{i}\}_{i = 1}^k$, $l \coloneqq l_1 \times\ldots\times l_k$, and probability measures $\{\mu_i \in \Pp{}{\statespace{i}}\}_{i = 1}^k$, $\mu \coloneqq (\mu_1,\ldots,\mu_k)$, it holds that
\begin{equation*}\label{equation:pushforward:stability:1}
\begin{aligned}
\freemarginals{\comp{(l\times\identityOn{\otherspace{}})}{c}}{\mu}{}{}
&=
\inf_{
\pushforward{\left(\projectionFromTo{\statespace{}\times\otherspace{}}{\statespace{}}\right)}{\plan{\mu}}
\in \setPlans{\mu}{}
}
\int_{\statespace{}\times\otherspace{}} 
\comp{(l\times\identityOn{\otherspace{}})}{c}
\d\plan{\mu}
\\&= 
\inf_{
\pushforward{\left(\projectionFromTo{\anotherspace{}\times\otherspace{}}{\anotherspace{}}\right)}{\plan{\mu'}}
\in \setPlans{\pushforward{l}{\mu}}{}
}
\int_{\anotherspace{}\times\otherspace{}} 
c\d\plan{\mu'}
=
\freemarginals{c}{\pushforward{l}{\mu}}{}{}.
\end{aligned}
\end{equation*}
\begin{comment}
In particular, when $c$ is a map $\anotherspace{}\to\nonnegativeRealsBar$, it holds:
\begin{equation}\label{equation:pushforward:stability:2}
\kantorovich{\comp{l}{c}}{\mu}{}{}
=
\inf_{
\plan{\mu} \in \setPlans{\mu}{}
}
\int_{\statespace{}} 
\comp{l}{c}
\d\plan{\mu}
= 
\inf_{\plan{\mu'} 
\in \setPlans{\pushforward{l}{\mu}}{}
}
\int_{\statespace{}} 
c\d\plan{\mu'}
=
\kantorovich{c}{\pushforward{l}{\mu}}{}{}.
\end{equation}
\end{comment}
\end{lemma}

\begin{proof}
We prove ``$\leq$'' and ``$\geq$'' separately.
%for \eqref{equation:pushforward:stability:1}; \eqref{equation:pushforward:stability:2} is then a special case.
% \begin{enumerate}
%     \item[``$\geq$''] 
We start with ``$\geq$''.
    For any $\plan{\mu} \in \Pp{}{\statespace{}\times\otherspace{}}$ such that $\pushforward{(\projectionFromTo{\statespace{}\times\otherspace{}}{\statespace{}})}{\plan{\mu}} \in \setPlans{\mu}{}$, let
    $
    \plan{\mu'} = \pushforward{(l\times\identityOn{\otherspace{}})}{\plan{\mu}}.
   $
    For $\inIndexSet{i}{1}{k}$ consider $\phi \in \Cb{\anotherspace{i}}$. It holds that
    \begin{align*}
        \int_{\anotherspace{}\times\otherspace{}}
        \phi(z_i)
        \d\plan{\mu'}(z,y)
        &=
        \int_{\anotherspace{}\times\otherspace{}}
        \phi(z_i)
        \d(\pushforward{(l\times\identityOn{\otherspace{}})}{\plan{\mu}})(z_1, \ldots, z_k,y)
        \\&=
        \int_{\statespace{}\times\otherspace{}}
        \phi(l_i(x_i))
        \d\plan{\mu}(x_1, \ldots, x_k, y)
        \\&=
        \int_{\statespace{i}}
        \phi(l_i(x_i))
        \d(\pushforward{(\projectionFromTo{\statespace{}\times\otherspace{}}{\statespace{i}})}{\plan{\mu}})(x_i)
        \\&=
        \int_{\statespace{i}}
        \comp{l_i}{\phi}
        \d\mu_i
        =
        \int_{\anotherspace{i}}
        \phi
        \d(\pushforward{l_i}{\mu_i}).
    \end{align*}
    That is,
    $
    \pushforward{(\projectionFromTo{\anotherspace{}\times\otherspace{}}{\anotherspace{i}})}{\plan{\mu'}}
    =
    \pushforward{l_i}{\mu_i}
    $
    and, thus,
    $
    \pushforward{(\projectionFromTo{\anotherspace{}\times\otherspace{}}{\anotherspace{}})}{\plan{\mu'}} \in \setPlans{\pushforward{l}{\mu}}{}
    $.
    Similarly, for all $\inIndexSet{j}{1}{h}$ we have 
    $
    \pushforward{(\projectionFromTo{\anotherspace{}\times\otherspace{}}{\otherspace{j}})}{\plan{\mu'}}
    =
    \pushforward{(\projectionFromTo{\statespace{}\times\otherspace{}}{\otherspace{j}})}{\plan{\mu}}
    \in \Pp{}{\otherspace{j}}$.
    Therefore, $\plan{\mu'}$ provides the upper bound 
    % is a valid candidate for the infimum in $\freemarginals{c}{\pushforward{l}{\mu}}{}{}$.
    % Moreover,
\begin{equation*}
    \freemarginals{c}{\pushforward{l}{\mu}}{}{}
    % &
    \leq
    % \int_{\anotherspace{}\times\otherspace{}}
    % c
    % \d\plan{\mu'}
    %\\&
    % =
    \int_{\statespace{}\times\otherspace{}}
    \comp{(l\times\identityOn{\otherspace{}})}{c}
    \d\plan{\mu}.
\end{equation*}
    Since $\plan{\mu}$ is arbitrary, we obtain the desired inequality.
    % $
    % \freemarginals{c}{\pushforward{l}{\mu}}{}{}
    % \leq
    % \freemarginals{\comp{(l\times\identityOn{\otherspace{}})}{c}}{\mu}{}{}
    % $.
% \item[``$\leq$''] 

To prove ``$\leq$'',
    fix $\plan{\mu'} \in \Pp{}{\anotherspace{}\times\otherspace{}}$ with $\pushforward{(\projectionFromTo{\anotherspace{}\times\otherspace{}}{\anotherspace{}})}{\plan{\mu'}} \in \setPlans{\pushforward{l}{\mu}}{}$.
    By definition,
    $
    \plan{\mu'} \in \setPlans{\pushforward{l}{\mu}}{\pushforward{(\projectionFromTo{\anotherspace{}\times\otherspace{}}{\otherspace{}})}{\plan{\mu'}}}.
    $
    Then, for all $\inIndexSet{i}{1}{k}$, let 
    $
    \plan{\mu_{i}} = \pushforward{\left(\id_{\statespace{i}}, l_{i}\right)}{\mu_i}\in \Pp{}{\statespace{i}\times\anotherspace{i}}.
    $
    Analogously to the previous step, we have $\plan{\mu_{i}} \in \setPlans{\mu_i}{\pushforward{l_{i}}{\mu_i}}$.  
    We can ``glue'' $\{\plan{\mu_{i}}\}_{i = 1}^k$ and $\plan{\mu'}$ to obtain $\plan{\mu^\ast} \in \Pp{}{\statespace{}\times\anotherspace{}\times\otherspace{}}$ such that
    $
    \pushforward{(\projectionFromTo{\statespace{}\times\anotherspace{}\times\otherspace{}}{\statespace{}\times\anotherspace{}})}{\plan{\mu^\ast}} \in \setPlans{\mu}{\pushforward{l}{\mu}}. 
    $
    Specifically, we apply $k$ times \cite[Gluing lemma]{Villani2007} as follows. First, we glue $\plan{\mu'}$ and $\plan{\mu_1}$, since they share a marginal:
    \begin{equation*}\pushforward{(\projectionFromTo{\anotherspace{}\times\otherspace{}}{\anotherspace{1}})}{\plan{\mu'}} 
    =  
    \pushforward{l_{1}}{\mu_1} 
    = 
    \pushforward{(\projectionFromTo{\statespace{1}\times\anotherspace{1}}{\anotherspace{1}})}{\plan{\mu_{1}}}.
    \end{equation*}
    Call the resulting plan
    $
    \plan{\mu^*_1} \in \setPlans{\mu_1}{\pushforward{l}{\mu}, \pushforward{(\projectionFromTo{\anotherspace{}\times\otherspace{}}{\otherspace{}})}{\plan{\mu'}}}. 
    $
    Next, we define inductively
    \begin{equation*}
    \plan{\mu^*_i} \in \setPlans{\mu_1, \ldots, \mu_i}{\pushforward{l}{\mu}, \pushforward{(\projectionFromTo{\anotherspace{}\times\otherspace{}}{\otherspace{}})}{\plan{\mu'}}}
    \end{equation*}
    as the plan obtained from gluing $\plan{\mu^*_{i - 1}}$ and $\plan{\mu_{i}}$ for $\inIndexSet{i}{2}{k}$. The definition is well-posed in view of \cite[Gluing lemma]{Villani2007}, since
    \begin{equation*}\pushforward{(\projectionFromTo{\statespace{1}\times\ldots\times\statespace{i}\times\anotherspace{}\times\otherspace{}}{\anotherspace{i}})}{\plan{\mu^\ast_{i - 1}}} = \pushforward{l_{i}}{\mu_i} = \pushforward{(\projectionFromTo{\statespace{i}\times\anotherspace{i}}{\anotherspace{i}})}{\plan{\mu_{i}}}.
    \end{equation*}
    Finally, we take $\plan{\mu^\ast} = \plan{\mu^\ast_k}$, so that
    \begin{equation*}
    \plan{\mu} = \pushforward{(\projectionFromTo{\statespace{}\times\otherspace{}\times\anotherspace{}}{\statespace{}\times\otherspace{}})}{\plan{\mu^\ast}} \in \SetPlans{\mu_1,\ldots,\mu_k}{\pushforward{(\projectionFromTo{\anotherspace{}\times\otherspace{}}{\otherspace{}})}{\plan{\mu'}}}.
    \end{equation*}
    Let $\bar{\statespace{}}\coloneqq \statespace{1}\times\ldots\times\statespace{k-1}$, $\bar{\anotherspace{}}\coloneqq \anotherspace{1}\times\ldots\times\anotherspace{k-1}$, and $\bar l \coloneqq l_1\times\ldots\times l_{k-1}$. Then, for the $\nth{k}$ argument of $c$,
    \begin{align*}
    &\int_{\statespace{}\times\otherspace{}} 
    (\comp{(l\times\identityOn{\otherspace{}})}{c})(x,y)
    \d\plan{\mu}(x,y)
    \\&=
    \int_{\statespace{} \times \anotherspace{} \times \otherspace{}} 
    (\comp{(l\times\identityOn{\otherspace{}})}{c})(x, y)
    \d\plan{\mu^\ast}(x, z, y)
    \\\overset{\clubsuit}&{=} 
    \int_{\statespace{k}\times\anotherspace{k}}
    \int_{\bar{\statespace{}}\times\bar{\anotherspace{}}\times\otherspace{}}
    (\comp{(l\times\identityOn{\otherspace{}})}{c})(\bar x, x_k, y)
    \d\plan{\tilde{\mu}^{x_kz_k}}(\bar x,\bar z, y)
    \d\plan{\mu_{k}}(x_k, z_k)
    \\&=
    \int_{\statespace{k}\times \anotherspace{k}}
    \int_{\bar{\statespace{}}\times\bar{\anotherspace{}}\times\otherspace{}}
    c(\bar l(\bar x), l_k(x_k), y)
    \d\plan{\tilde{\mu}^{x_kz_k}}(\bar x,\bar z, y)
    \d\plan{\mu_{k}}(x_k, z_k)
    \\\overset{\spadesuit}&{=} 
    \int_{\statespace{k}\times \anotherspace{k}}
    \int_{\bar{\statespace{}}\times\bar{\anotherspace{}}\times\otherspace{}}
    c(\bar l(\bar x), z_k, y)
    \d\plan{\tilde{\mu}^{x_kz_k}}(\bar x,\bar z, y)
    \d\plan{\mu_{k}}(x_k, z_k)
    \\\overset{\clubsuit}&{=}
    \int_{\statespace{}\times\anotherspace{}\times\otherspace{}} 
    c(\bar l(\bar x), z_k, y)
    \d\plan{\mu^*}(x,z,y),
    \end{align*}
    where in $\clubsuit$ we used the disintegration theorem (see \cite[Theorem 5.3.1]{Ambrosio2008}), which provides us a collection $\{\plan{\tilde{\mu}^{x_kz_k}}\}_{(x_k,z_k) \in \statespace{k}\times\anotherspace{k}}$ to complement $\plan{\mu_k}$. Then, in $\spadesuit$, we used the definition of $\plan{\mu_{k}}$: $z_k = l_k(x_k)$ $\mu_k$-$\ae$.
    Repeating the same steps for the other arguments of $c$, we obtain 
    \begin{align*}
    \freemarginals{\comp{(l\times\identityOn{\otherspace{}})}{c}}{\mu}{}{}
    &\leq
    \int_{\statespace{}\times\otherspace{}}
    c(l(x), y) 
    \d\plan{\mu}(x, y)
    \\&=
    \int_{\statespace{}\times\anotherspace{}\times\otherspace{}}
    c(z, y)
    \d\plan{\mu^\ast}(x,z,y)
    % \\&
    =
    \int_{\anotherspace{}\times\otherspace{}}
    c
    \d\plan{\mu'}.
    \end{align*}
    Since $\plan{\mu'}$ is arbitrary, it follows that
    $
    \freemarginals{\comp{(l\times\identityOn{\otherspace{}})}{c}}{\mu}{}{}
    \leq
    \freemarginals{c}{\pushforward{l}{\mu}}{}{}
    $.
% \end{enumerate}
\end{proof}

The next result expresses the sum of two optimal transport discrepancies, possibly with free marginals, as a single optimal transport discrepancy with the same free marginals. Similar results provide multi-marginal reformulations for Wasserstein barycenters \cite{Martial2011,Lindheim2022}, whose computation has recently received much interest \cite{Altschuler2020,Cuturi2013bary}.

\begin{lemma}[Sum of optimal transport discrepancies]\label{lemma:compositionality:free-marginals}
Given transportation costs $c_1: \statespace{1} \times \anotherspace{} \to \nonnegativeRealsBar$, $c_2: \statespace{}\times\otherspace{} \to \nonnegativeRealsBar$, and probability measures $\{\mu_i \in \Pp{}{\statespace{i}}\}_{i = 1}^k$, $\mu \coloneqq (\mu_1,\ldots,\mu_k)$, $\nu \in \Pp{}{\anotherspace{}}$, it holds that
\begin{equation*}
\kantorovich{c_1}{\mu_1}{\nu}{}
+ 
\freemarginals{c_2}{\mu}{}{}
=
\freemarginals{c}{\mu}{\nu}{},
\end{equation*}
with $c: \statespace{}\times\otherspace{}\times\anotherspace{} \to \nonnegativeRealsBar$ defined as
$
c
(x_1, \ldots, x_k, y, z) 
= 
% \comp{l}{V}(x_1) 
% + 
c_1
(x_1, z) 
+ 
c_2
(x_1, \ldots, x_k, y).
$
\begin{comment}
In particular, when $c_2$ is a map $\statespace{} \to \nonnegativeRealsBar$, it holds:
\begin{equation}\label{equation:proposition:compositionality:free-marginals:2}
\kantorovich{c_1}{\mu_1}{\nu}{}
+ 
\kantorovich{c_2}{\mu}{}{}
=
\kantorovich{c}{\mu}{\nu}{},
\end{equation}
with $c: \statespace{}\times \statespace{k} \to \nonnegativeRealsBar$ defined as 
$
c(x_1, \ldots, x_k, z) 
= 
% \comp{l}{V}(x_1) 
% + 
c_1(x_1, z) 
+ 
c_2(x_1, \ldots, x_k).
$
\end{comment}
\end{lemma}
\begin{proof}
We prove ``$\leq$'' and ``$\geq$'' separately.
%for \eqref{equation:proposition:compositionality:free-marginals:1}; \eqref{equation:proposition:compositionality:free-marginals:2} follows as a special case.
% \begin{itemize}
% \item[``$\leq$''] 
With the short-hand notation $x \coloneqq (x_1, \ldots, x_k)$, 
``$\leq$'' 
follows from minimizing separately over the shared marginal:
\begin{align*}
\freemarginals{c}{\mu}{\nu}{}
&=
\inf_{
\pushforward{(\projectionFromTo{\statespace{}\times\otherspace{}\times\anotherspace{}}{\statespace{}\times\anotherspace{}})}{\plan{\gamma}
} \in
\setPlans{\mu}{\nu}}
\int_{\statespace{}\times\otherspace{}\times\anotherspace{}}
c_1(x_1, z)
+
c_2(x,y)
\d \plan{\gamma}(x, y, z)
\\&\geq
\inf_{
\pushforward{(\projectionFromTo{\statespace{}\times\otherspace{}\times\anotherspace{}}{\statespace{}\times\anotherspace{}})}{\plan{\gamma}
} \in
\setPlans{\mu}{\nu}}
\int_{\statespace{}\times\otherspace{}\times\anotherspace{}}
c_1(x_1, z)
\d \plan{\gamma}(x, y, z)
\\&\quad + 
\inf_{
\pushforward{(\projectionFromTo{\statespace{}\times\otherspace{}\times\anotherspace{}}{\statespace{}\times\anotherspace{}})}{\plan{\gamma}
} \in
\setPlans{\mu}{\nu}}
\int_{\statespace{}\times\otherspace{}\times\anotherspace{}}
c_2(x, y)
\d \plan{\gamma}(x, y, z)
\\&=
\inf_{
\pushforward{(\projectionFromTo{\statespace{}\times\otherspace{}\times\anotherspace{}}{\statespace{}\times\anotherspace{}})}{\plan{\gamma}
} \in
\setPlans{\mu}{\nu}}
\int_{\statespace{1}\times\anotherspace{}}
c_1(x_1, z)
\d (\pushforward{
(\projectionFromTo{\statespace{}\times\otherspace{}\times\anotherspace{}}{\statespace{1}\times\anotherspace{}})
}{\plan{\gamma}})(x_1, z)
\\&\quad + 
\inf_{\pushforward{
(\projectionFromTo{\statespace{}\times\otherspace{}\times\anotherspace{}}{\statespace{}\times\anotherspace{}})
}{\plan{\gamma}} \in
\setPlans{\mu}{\nu}}
\int_{\statespace{} \times\otherspace{}}
c_2(x,y)
\d(\pushforward{(\projectionFromTo{\statespace{}\times\otherspace{}\times\anotherspace{}}{\statespace{}\times\otherspace{}})}
{\plan{\gamma}})(x,y)
\\\overset{\heartsuit}&{=}
\inf_{\plan{\gamma_1} \in \setPlans{\mu_1}{\nu}}
\int_{\statespace{1}\times\anotherspace{}}
c_1
\d \plan{\gamma_1}
+ 
\inf_{\pushforward{(
\projectionFromTo{\statespace{}\times\otherspace{}}{\statespace{}}
)}{\plan{\gamma_2}} \in 
\setPlans{\mu}{}}
\int_{\statespace{}\times\otherspace{}}
c_2
\d \plan{\gamma_2}
\\&=
\kantorovich{c_1}{\mu_1}{\nu}{}
+ 
\freemarginals{c_2}{\mu}{}{},
\end{align*}
where in $\heartsuit$ (i) we noticed that the first infimum is only over $\pushforward{(\projectionFromTo{\statespace{}\times\otherspace{}\times\anotherspace{}}{\statespace{1}\times\anotherspace{}})}{\plan{\gamma}} = \plan{\gamma'} \in \setPlans{\mu_1}{\nu}$, and (ii) in the second infimum we used \cref{lemma:pushforward:stability} with the pushforward map being $\projectionFromTo{\statespace{}\times\otherspace{}\times\anotherspace{}}{\statespace{}\times\otherspace{}}$.

% \item[``$\geq$'']
We now prove ``$\geq$''.
For all $\varepsilon > 0$, consider $\varepsilon$-optimal $\epsilonPlan{\gamma_1}{\varepsilon}{} \in \setPlans{\mu_1}{\nu}$ and $\epsilonPlan{\gamma_2}{\varepsilon}{} \in \Pp{}{\statespace{}\times\otherspace{}}$ so that $\pushforward{(\projectionFromTo{\statespace{}\times\otherspace{}}{\statespace{}})}{\epsilonPlan{\gamma_2}{\varepsilon}{}} \in \setPlans{\mu}{}$; i.e., 
\begin{equation*}
\int_{\statespace{1}\times\anotherspace{}}c_1\d\epsilonPlan{\gamma_1}{\varepsilon}{} \leq \kantorovich{c_1}{\mu_1}{\nu}{} + \varepsilon
\quad\text{and}\quad
\int_{\statespace{}\times\otherspace{}}c_2\d\epsilonPlan{\gamma_2}{\varepsilon}{} \leq \freemarginals{c_2}{\mu}{}{} + \varepsilon.
\end{equation*}
Since 
$
\pushforward{(\projectionFromTo{\statespace{1}\times\anotherspace{}}{\statespace{1}})}{\epsilonPlan{\gamma_1}{\varepsilon}{}} = \mu_1 = \pushforward{(\projectionFromTo{\statespace{}\times\otherspace{}}{\statespace{1}})}{\epsilonPlan{\gamma_2}{\varepsilon}{}}
$,
we can glue them \cite[Gluing lemma]{Villani2007} to obtain
$
\epsilonPlan{\gamma}{\varepsilon}{} \in \setPlans{\mu,\nu}{\pushforward{(\projectionFromTo{\statespace{}\times\otherspace{}}{\otherspace{}})}{\epsilonPlan{\gamma_2}{\varepsilon}{}}}.
$
% In particular, 
% $
% \pushforward{(\projectionFromTo{\statespace{}\times\otherspace{}\times\anotherspace{}}{\statespace{}})}{\plan{\gamma}} \in \setPlans{\indexedVar{\mu}{1}{k}}{}$,
% $
% \pushforward{(\projectionFromTo{\statespace{}\times\otherspace{}\times\anotherspace{}}{\statespace{1}\times\anotherspace{}})}{\plan{\gamma}} = \plan{\gamma_1}
% $, and
% $
% \pushforward{(\projectionFromTo{\statespace{}\times\otherspace{}\times\anotherspace{}}{\statespace{}\times\otherspace{}})}{\plan{\gamma}} = \plan{\gamma_2}.
% $
Then it holds that
\begin{displaymath}
% \freemarginals{c}{\mu}{\nu}{}
% &\leq 
\int_{\statespace{}\times \otherspace{}\times\anotherspace{}}
c
\d \epsilonPlan{\gamma}{\varepsilon}{}
% \\&=
=
\int_{\statespace{1}\times\anotherspace{}}
c_1
\d 
\underbrace{
(\pushforward{(\projectionFromTo{\statespace{}\times\otherspace{}\times\anotherspace{}}{\statespace{1}\times\anotherspace{}})}{\epsilonPlan{\gamma}{\varepsilon}{}})
}_{\epsilonPlan{\gamma_1}{\varepsilon}{}}
+
\int_{\statespace{}\times\otherspace{}}
c_2
\d 
\underbrace{
(\pushforward{(\projectionFromTo{\statespace{}\times\otherspace{}\times\anotherspace{}}{\statespace{}\times\otherspace{}})}{\epsilonPlan{\gamma}{\varepsilon}{}})
}_{\epsilonPlan{\gamma_2}{\varepsilon}{}}
\end{displaymath}
and, thus, $\freemarginals{c}{\mu}{\nu}{} \leq 
\kantorovich{c_1}{\mu_1}{\nu}{} + \freemarginals{c_2}{\mu}{}{} 
+ 
2\varepsilon$. Let $\varepsilon \to 0$ to conclude.
% \end{itemize}
\end{proof}

In particular, when $\kantorovich{c_1}{}{}{}$ is an expected value, the composition simplifies.
%and $\freemarginals{c_2}{}{}{}$ is an optimal transport problem between two marginals, the composition can be expressed as the latter:
\begin{lemma}[Compositionality of optimal transport]\label{lemma:compositionality:optimal-transport}
Given a cost $v: \statespace{} \to \nonnegativeRealsBar$, a transportation cost $c: \otherspace{} \times \anotherspace{} \to \nonnegativeRealsBar$, a map $l: \statespace{} \to \otherspace{}$, and probability measures $\mu \in \Pp{}{\statespace{}}, \nu \in \Pp{}{\anotherspace{}}$, it holds that
\begin{equation*}
\expectedValue{\mu}{v}(\mu)
+
\kantorovich{c}{\pushforward{l}{\mu}}{\nu}{}
=
\kantorovich{v + \comp{(l\times\identityOn{\otherspace{}})}{c}}{\mu}{\nu}{}.
\end{equation*}
\end{lemma}
\begin{proof}
The statement is a special case of~\cref{lemma:pushforward:stability}.
\begin{comment}
With~\cref{lemma:pushforward:stability}:
\begin{align*}
\int_{\statespace{}}v\d\mu + \inf_{\plan{\gamma'} \in \setPlans{\pushforward{l}{\mu}}{\nu}} \int_{\anotherspace{}\times \otherspace{}} c \d\plan{\gamma'} 
&=
\int_{\statespace{}}v\d\mu 
+ 
\inf_{\plan{\gamma} \in \setPlans{\mu}{\nu}} 
\int_{\statespace{}\times \otherspace{}} 
\comp{(l\times\identityOn{\otherspace{}})}{c} 
\d\plan{\gamma} 
\\
&= \inf_{\plan{\gamma} \in \setPlans{\mu}{\nu}} 
\int_{\statespace{}\times \otherspace{}} 
v + \comp{(l \times \identityOn{\otherspace{}})}{c}
\d\plan{\gamma}.
\end{align*}
\end{comment}
\end{proof}

Finally, we give a useful disintegration property of the cost term $\freemarginals{c}{}{}{}$.

\begin{lemma}[Disintegration of the optimizer]\label{lemma:disintegration}
Given a transportation cost $c: \statespace{}\times\otherspace{}\times\anotherspace{} \to \nonnegativeRealsBar$ and probability measures $\mu \in \Pp{}{\statespace{}}$, $\nu \in \Pp{}{\otherspace{}}$, it holds that
\begin{multline*}
\inf_{\pushforward{(\projectionFromTo{\statespace{}\times\otherspace{}\times\anotherspace{}}{\statespace{}\times\otherspace{}})}{\plan{\gamma}} \in \setPlans{\mu}{\nu}}
\int_{\statespace{}\times\otherspace{}\times\anotherspace{}}
c(x,y,z)
\d\plan{\gamma}(x,y,z)
\\=
\inf_{\plan{\gamma'} \in \setPlans{\mu}{\nu}} 
\inf_{
\{\xi^{xy}\} \in \Lambda(\anotherspace{})
}
\int_{\statespace{}\times\otherspace{}}
\int_{\anotherspace{}}
c(x, y, z)
\d\xi^{xy}(z)
\d\plan{\gamma'}(x,y),
\end{multline*}
where $
\Lambda(\anotherspace{}) 
\coloneqq
\{
\{\xi^{xy}\}_{(x,y) \in \statespace{} \times \otherspace{}}
\subseteq \Pp{}{\anotherspace{}} 
\st
\statespace{}\times\otherspace{} \ni (x, y) \mapsto \xi^{xy} \in \Pp{}{\anotherspace{}} \text{ Borel}
\}.
$
\end{lemma}
\begin{proof}
We prove ``$\geq$'' and ``$\leq$'' separately.
% \begin{itemize}
% \item[``$\geq$''] 
To prove ``$\geq$'', consider any $\plan{\gamma} \in \Pp{}{\statespace{}\times\otherspace{}\times\anotherspace{}}$ such that $\pushforward{(\projectionFromTo{\statespace{}\times\otherspace{}\times\anotherspace{}}{\statespace{}\times\otherspace{}})}{\plan{\gamma}} \in \setPlans{\mu}{\nu}$. By~\cite[Theorem 5.3.1]{Ambrosio2008}, there exists $\{\plan{\gamma}^{xy}\} \in \Lambda(\anotherspace{})$ such that
\begin{displaymath}
\begin{aligned}
\int_{\statespace{}\times\otherspace{}\times\anotherspace{}}
c
\d\plan{\gamma}
&=
\int_{\statespace{}\times\otherspace{}}
\int_{\anotherspace{}}
c(x, y, z)
\d\plan{\gamma}^{xy}(z)
\d(\pushforward{(\projectionFromTo{\statespace{}\times\otherspace{}\times\anotherspace{}}{\statespace{}\times\otherspace{}})}{\plan{\gamma}})(x, y)
\\
&\geq
\inf_{\plan{\gamma'} \in \setPlans{\mu}{\nu}} 
\inf_{\{\xi^{xy}\} \in \Lambda(\anotherspace{})}
\int_{\statespace{}\times\otherspace{}}
\int_{\anotherspace{}}
c(x, y, z)
\d\xi^{xy}(z)
\d\plan{\gamma'}(x, y).
\end{aligned}
\end{displaymath}
Then take the infimum over $\plan{\gamma}$.

% \item[``$\leq$'']
To prove ``$\leq$'',
we follow \cite[\S 5.3]{Ambrosio2008} to construct the reverse of the disintegration. Given any $\plan{\gamma'} \in \setPlans{\mu}{\nu}$ and any $\{\xi^{xy}\} \in \Lambda(\anotherspace{})$, then we can construct a Borel probability measure $\plan{\gamma} \in \Pp{}{\statespace{}\times\otherspace{}\times\anotherspace{}}$ defined for every Borel set $B \subseteq \statespace{}\times\otherspace{}\times\anotherspace{}$ as
\begin{equation*}
\plan{\gamma}(B) 
= 
\int_{\statespace{}\times\otherspace{}} 
\int_{\anotherspace{}} 
1_B(x, y, z)
\d \xi^{xy} (z)
\d \plan{\gamma'}(x, y)
,
\end{equation*}
so that for every Borel measurable map $\phi: \statespace{}\times\otherspace{}\times\anotherspace{} \to \nonnegativeRealsBar$ we have (cf. \cite[Definition 11.12 and Definition 11.1]{Aliprantis2006})
\begin{displaymath}
\int_{\statespace{}\times\otherspace{}\times\anotherspace{}}
\phi(x, y, z)\d\plan{\gamma}(x, y, z)
=
\int_{\statespace{}\times\otherspace{}}\int_{\anotherspace{}}
\phi(x, y, z)
\d \xi^{xy} (z)
\d \plan{\gamma'}(x, y).
\end{displaymath}
Then, for $\phi \in \Cb{\statespace{}\times\otherspace{}}$, we have
\begin{align*}
\int_{\statespace{}\times\otherspace{}\times\anotherspace{}} 
\phi(x, y)
\d \plan{\gamma}(x, y, z)
&=
\int_{\statespace{}\times\otherspace{}} 
\int_{\anotherspace{}} 
\phi(x, y)
\d \xi^{xy} (z)
\d \plan{\gamma'}(x, y)
\\&=
\int_{\statespace{}\times\otherspace{}} 
\phi(x, y)
\d \plan{\gamma'}(x, y).
\end{align*}
Thus, $\pushforward{(\projectionFromTo{\statespace{}\times\otherspace{}\times\anotherspace{}}{\statespace{}\times\otherspace{}})}{\plan{\gamma}} = \plan{\gamma'} \in \setPlans{\mu}{\nu}$ and
\begin{multline*}
\int_{\statespace{}\times\otherspace{}}
\int_{\anotherspace{}}
c(x, y, z)
\d\xi^{xy}(z)
\d\plan{\gamma'}(x, y)
\\=
\int_{\statespace{}\times\otherspace{}\times\anotherspace{}}
c
\d\plan{\gamma}
\geq
\inf_{\pushforward{(\projectionFromTo{\statespace{}\times\otherspace{}\times\anotherspace{}}{\statespace{}\times\otherspace{}})}{\plan{\gamma}} \in \setPlans{\mu}{\nu}}
\int_{\statespace{}\times\otherspace{}\times\anotherspace{}}
c
\d\plan{\gamma}.
\end{multline*}
The claim follows taking the infimum over $\plan{\gamma'}$ and $\{\xi^{xy}\}$.
% \end{itemize}
\end{proof}

We are now ready to prove~\cref{theorem:lifting,corollary:lifting:ot}.

\begin{proof}[Proof of~\cref{theorem:lifting}]
We prove the statements separately. To ease the notation, we recall $\refSpace{} \coloneqq \refSpace{k} \times \refSpace{k+1} \times \ldots \refSpace{N}$, and we introduce 
\begin{equation}
\label{equation:shorthand:dpa-step}
c_k \coloneqq \stagecostsmall{k}
+ 
\comp{(\dynamics{k}\times\identityOn{\refSpace{k+1}\times\ldots\times\refSpace{N}})}{\costtogosmall{k+1}}: \statespace{k}\times\inputspace{k}\times\refSpace{}\to\nonnegativeRealsBar.
\end{equation}
\begin{enumerate}[leftmargin=*]
\item 
We proceed by induction. The base case is $\costtogo{N} = \kantorovich{\terminalcostsmall}{}{}{}$ and $\costtogosmall{N} = \terminalcostsmall$. For $k < N$, suppose $\costtogo{k+1}=\kantorovich{\costtogosmall{k+1}}{}{}{}$.
Then the backward recursion gives
\begin{align*}
\costtogo{k}(\mu_k, \refPVar)
&=
\inf_{
\pushforward{(\projectionFromTo{\statespace{k}\times\inputspace{k}}{\statespace{k}})}{\probabilityInput{}{k}}=\mu_k
}
\kantorovich{\stagecostsmall{k}}{\probabilityInput{}{k}}{\refPVar_k}{}
+ 
\costtogo{k+1}(\pushforward{\dynamics{k}}{\probabilityInput{}{k}}, \refPVar_{k+1}, \ldots, \refPVar_N)
\\
&=
\inf_{
\pushforward{(\projectionFromTo{\statespace{k}\times\inputspace{k}}{\statespace{k}})}{\probabilityInput{}{k}}=\mu_k
}
\inf_{\plan{\gamma_1} \in \setPlans{\probabilityInput{}{k}}{\refPVar_k}}
\int_{\statespace{k}\times\inputspace{k}\times \refSpace{k}}
\stagecostsmall{k}
\d \plan{\gamma_1}
\\&\qquad+
\inf_{\plan{\gamma_2} \in \setPlans{\pushforward{\dynamics{k}}{\probabilityInput{}{k}}}{\refPVar_{k+1}, \ldots, \refPVar_N}}
\int_{\statespace{k+1}\times\refSpace{k+1}\times\ldots\times\refSpace{N}} 
\costtogosmall{k+1}
\d \plan{\gamma_2}.
\end{align*}
\cref{lemma:pushforward:stability,lemma:compositionality:free-marginals}, together with the definition of $c_k$ (see \eqref{equation:shorthand:dpa-step}), yield
\begin{align*}
\costtogo{k}(\mu_k, \refPVar)
&=
\inf_{
\substack{
\pushforward{(\projectionFromTo{\statespace{k}\times\inputspace{k}}{\statespace{k}})}{\probabilityInput{}{k}}=\mu_k
\\
\plan{\gamma'} \in \setPlans{\probabilityInput{}{k}}{\refPVar}
}}
\int_{\statespace{k}\times\inputspace{k}\times\refSpace{}}
c_k
\d
\plan{\gamma'}
\\
&=
\inf_{\pushforward{(\projectionFromTo{\statespace{k}\times\inputspace{k}\times\refSpace{}}{\statespace{k}\times\refSpace{}})}{\plan{\gamma'}} \in \setPlans{\mu_k}{\refPVar}}
\int_{\statespace{k}\times\inputspace{k}\times\refSpace{}}
c_k
\d
\plan{\gamma'}.
\end{align*}
\cref{lemma:disintegration} enables us to disintegrate $\costtogo{k}(\mu_k, \refPVar)$ as 
\begin{multline*}
\inf_{
\substack{
\plan{\gamma'}\in\setPlans{\mu_k}{\refPVar}
\\
\{\xi^{x_k\refVar}\} \in \Lambda(\inputspace{k})
}}
\int_{\statespace{k}\times\refSpace{}}
\int_{\inputspace{k}}
c_k(x_k, \inputVariable{}{}{}{}, \refVar)
\d\xi^{x_k\refVar}(\inputVariable{}{}{}{})\d\plan{\gamma'}(x_k, \refVar)
\\
\overset{\clubsuit}{=}
\inf_{\plan{\gamma'}\in\setPlans{\mu_k}{\refPVar}}
\int_{\statespace{k}\times\refSpace{}}
\inf_{
\xi 
\in \Pp{}{\inputspace{k}}
}
\int_{\inputspace{k}}
c_k(x_k, \inputVariable{}{}{}{}, \refVar)
\d\xi(\inputVariable{}{}{}{})\d\plan{\gamma'}(x_k, \refVar).
\end{multline*}
The equality in $\clubsuit$ requires proving ``$\geq$'' and ``$\leq$'' separately.
Let 
\begin{equation*}
\psi(x_k, \xi, \refVar) \coloneqq \int_{\inputspace{k}}
c_k(x_k, \inputVariable{}{}{}{}, \refVar)
\d\xi(\inputVariable{}{}{}{}).    
\end{equation*}
Then, $\Lambda(\inputspace{k}) \subseteq \{\xi^{x_k\refVar} \in \Pp{}{\inputspace{k}}\}$, and
% \begin{displaymath}
% \inf_{
% \{\xi^{x_k\refVar}\} \in \Lambda(\inputspace{k})
% }
% \int_{\statespace{k}\times\refSpace{}}
% \psi(x_k, \xi^{x_k\refVar}, \refVar)
% \d\plan{\gamma'}(x_k, \refVar)
% \geq
% \inf_{
% \{\xi^{x_k\refVar} \in \Pp{}{\inputspace{k}}\}
% }
% \int_{\statespace{k}\times\refSpace{}}
% \psi(x_k, \xi^{x_k\refVar}, \refVar)
% \d\plan{\gamma'}(x_k, \refVar)
% .
% \end{displaymath}
$\psi(x_k, \xi, \refVar) \geq \inf_{\xi \in \Pp{}{\inputspace{k}}} \psi(x_k, \xi, \refVar)$ reveal ``$\geq$''.
To prove ``$\leq$'',
let $\Omega \coloneqq \supp(\mu_k)\times\supp(\refPVar) \subseteq \statespace{k}\times\refSpace{}$ and $\plan{\gamma'}\in\setPlans{\mu_k}{\refPVar}$. By definition, we can restrict the integration domain to the support of $\plan{\gamma'}$, for which it holds that $\supp(\plan{\gamma'}) \subseteq \supp(\mu_k) \times \supp(\refPVar)$. We thus consider $\Omega$ in place of $\statespace{k}\times\refSpace{}$ as the integration domain.
For all $\varepsilon>0$, consider the collection $\{\inputMap{x_k}{\refVar}{\varepsilon/2}{k}\}_{(x_k, \refVar) \in \Omega} \subseteq \inputspace{k}$.
Without loss of generality, we assume that $\inputMap{}{}{\varepsilon/2}{k}$ is Borel; see the discussion in~\cref{section:lifting}. As a consequence of the next lemma (\cref{lemma:borel_lifting}), also the measure-valued map $h: \Omega \to \Pp{}{\inputspace{k}}, h(x_k, \refVar) \coloneqq \delta_{\inputMap{x_k}{\refVar}{\varepsilon/2}{k}}$ is Borel.
Then, $\{\delta_{\inputMap{x_k}{\refVar}{\varepsilon/2}{k}}\}_{(x_k,\refVar) \in \Omega} \in \Lambda(\inputspace{k})$, with $\Lambda(\inputspace{k})$ as in~\cref{lemma:disintegration}, and thus
\begin{align*}
&\int_{\Omega}
\inf_{
\xi 
\in \Pp{}{\inputspace{k}}
}
\int_{\inputspace{k}}
c_k(x_k, \inputVariable{}{}{}{}, \refVar)
\d\xi(\inputVariable{}{}{}{})\d\plan{\gamma'}(x_k, \refVar)
\\&\geq
\int_{\Omega}
\inf_{\inputVariable{}{}{}{} \in \inputspace{k}}
c_k(x_k, \inputVariable{}{}{}{}, \refVar)
\d\plan{\gamma'}(x_k, \refVar) 
\\&\geq
\int_{\Omega}
c_k(x_k, \inputMap{x_k}{\refVar}{\varepsilon/2}{k}, \refVar)
\d\plan{\gamma'}(x_k, \refVar) 
- \frac{\varepsilon}{2}
\\&\geq
\int_{\Omega}
\int_{\inputspace{k}}
c_k(x_k, \inputVariable{}{}{}{}, \refVar)
\d\delta_{\inputMap{x_k}{\refVar}{\varepsilon/2}{k}}(\inputVariable{}{}{}{})\d\plan{\gamma'}(x_k, \refVar)
- \frac{\varepsilon}{2}
\\&\geq
\inf_{
\{\xi^{x_k\refVar}\} \in \Lambda(\inputspace{k})
}
\int_{\Omega}
\int_{\inputspace{k}}
c_k(x_k, \inputVariable{}{}{}{}, \refVar)
\d\xi^{x_k\refVar}(\inputVariable{}{}{}{})
\d\plan{\gamma'}(x_k, \refVar) - \frac{\varepsilon}{2}.
\end{align*}
Take the infimum over $\plan{\gamma'}$ on both sides, and let $\varepsilon \to 0$ to prove ``$\leq$''.

Next, it holds that
\begin{equation*}
\inf_{\xi \in \Pp{}{\inputspace{k}}} 
\int_{\inputspace{k}}
c_k(x_k, \inputVariable{}{}{}{}, \refVar)
\d\xi(\inputVariable{}{}{}{})
\geq
\inf_{\inputVariable{}{}{}{} \in \inputspace{k}}
c_k(x_k, \inputVariable{}{}{}{}, \refVar)
=
\costtogosmall{k}(x,\refVar).
\end{equation*}
For ``$\leq$'', let $\{\inputVariable{}{}{}{n}\}_{n \in \naturals} \subseteq \inputspace{k}$ yield
$
\costtogosmall{k}(x_k,\refVar) 
% = 
% \inf_{\inputVariable{}{}{}{} \in \inputspace{k}} c_k(x_k, \inputVariable{}{}{}{}, \refVar)
= 
\lim_{n \to \infty} c_k(x_k, \inputVariable{}{}{}{n}, \refVar)
$, and consider $\{\delta_{\inputVariable{}{}{}{n}}\}_{n \in \naturals} \subseteq \Pp{}{\inputspace{k}}$.
For all $n \in \naturals$, we have 
\begin{displaymath}
\inf_{\xi \in \Pp{}{\inputspace{k}}} 
\int_{\inputspace{k}}
c_k(x_k, \inputVariable{}{}{}{}, \refVar)
\d\xi(\inputVariable{}{}{}{})
\leq
\int_{\inputspace{k}}
c_k(x_k, \inputVariable{}{}{}{}, \refVar)
\d\delta_{\inputVariable{}{}{}{n}}(\inputVariable{}{}{}{})
=
c_k(x_k, \inputVariable{}{}{}{n}, \refVar).
\end{displaymath}
The limit $n \to \infty$ reveals ``$\leq$''
% \begin{displaymath}
% \inf_{\xi \in \Pp{}{\inputspace{k}}} 
% \int_{\inputspace{k}}
% c_k(x_k, \inputVariable{}{}{}{}, \refVar)
% \d\xi(\inputVariable{}{}{}{})
% \leq
% \costtogosmall{k}(x,\refVar)
% \end{displaymath}
and thus the equality. Thus, for every $x_k \in \statespace{k}, \refVar \in \refSpace{}$, we have 
\begin{multline*}
\inf_{
\xi 
\in \Pp{}{\inputspace{k}}
}
\int_{\inputspace{k}}
c_k(x_k, \inputVariable{}{}{}{}, \refVar)
\d\xi(\inputVariable{}{}{}{})
=
\costtogosmall{k}(x_k, \refVar)
\\\text{and so}\quad
\costtogo{k}(\mu_k, \refPVar)
=
\inf_{\plan{\gamma'}\in\setPlans{\mu_k}{\refPVar}}
\int_{\statespace{k}\times\refSpace{}}
\costtogosmall{k}
\d\plan{\gamma'}.
\end{multline*}

This proves~\eqref{equation:lifting:optimal-transport}. Finally, analogously to the traditional \gls*{acr:dpa}~\cite{Bertsekas2014,Bertsekas2017}, the additivity of the cost structure yields $\costtogo{} = \costtogo{0}$.
\item
Let $\varepsilon\geq 0$, and define $\inputMap{}{}{\varepsilon/2}{k}$, and $\epsilonPlan{\gamma}{\varepsilon/2}{k}$ as in the theorem statement. Consider the (possibly sub-optimal) plan
\begin{equation}
\label{equation:auxiliary-plan:definition}
    \epsilonPlan{\tilde{\gamma}}{\varepsilon}{k} \coloneqq \pushforward{\left(\projectionFromTo{\statespace{k}\times\refSpace{}}{\statespace{k}}, \inputMap{}{}{\varepsilon/2}{k}, \projectionFromTo{\statespace{k}\times\refSpace{}}{\refSpace{}}\right)}{\epsilonPlan{\gamma}{\varepsilon/2}{k}}.
\end{equation}
By definition, $\pushforward{\projectionFromTo{\statespace{k}\times\inputspace{k}\times\refSpace{}}{\statespace{k}\times\inputspace{k}}}{\epsilonPlan{\tilde{\gamma}}{\varepsilon}{k}} = \probabilityInput{\varepsilon}{k}$ and $\pushforward{\projectionFromTo{\statespace{k}\times\inputspace{k}\times\refSpace{}}{\refSpace{}}}{\epsilonPlan{\tilde{\gamma}}{\varepsilon}{k}} = \refPVar$. Therefore, $\epsilonPlan{\tilde{\gamma}}{\varepsilon}{k}$ is a valid choice for the infimum, and it holds that
\begin{align*}
\kantorovich{\stagecostsmall{k}}{\probabilityInput{\varepsilon}{k}}{\refPVar_k}{}
&+ 
\costtogo{k+1}(\pushforward{\dynamics{k}}{\probabilityInput{\varepsilon}{k}}, \refPVar_{k+1}, \ldots, \refPVar_N)
\\
\overset{\heartsuit}&{=}
\inf_{\plan{\gamma'} \in \setPlans{\probabilityInput{\varepsilon}{k}}{\refPVar}}
    \int_{\statespace{k}\times\inputspace{k}\times\refSpace{}}
    c_k(x_k,\inputVariable{}{}{}{k},y)
    \d \plan{\gamma'}(x_k,\inputVariable{}{}{}{k},\refVar)
\\
&\leq
\int_{\statespace{k}\times\inputspace{k}\times\refSpace{}}
c_k(x_k, \inputVariable{}{}{}{k}, \refVar)
\d \epsilonPlan{\tilde{\gamma}}{\varepsilon}{k}(x_k, \inputVariable{}{}{}{k}, \refVar)
\\
\overset{\text{\eqref{equation:auxiliary-plan:definition}}}&{=}
\int_{\statespace{k}\times\refSpace{}}
c_k(x_k, \inputMap{x_k}{y}{\varepsilon/2}{k}, \refVar)
\d \epsilonPlan{\gamma}{\varepsilon/2}{k}(x_k, \refVar)
\\
\overset{\text{\eqref{equation:input:epsilon-optimality}}}&{\leq}
\frac{\varepsilon}{2}
+
\int_{\statespace{k}\times\refSpace{}}
\costtogosmall{k}(x_k, \refVar)
\d \epsilonPlan{\gamma}{\varepsilon/2}{k}(x_k, \refVar)
\\
\overset{\text{\eqref{equation:transport-plan:epsilon-optimal}}}&{\leq}
\varepsilon
+
% \inf_{\plan{\gamma'}\in\setPlans{\mu_k}{\refPVar}}
% \int_{\statespace{k}\times\refSpace{}}
% \costtogosmall{k}(x_k, \refVar)
% \d \plan{\gamma'}(x_k, \refVar),
\costtogo{k}(\mu_k, \refPVar),
\end{align*}
where, in $\heartsuit$, we used the definition of $c_k$ (see \eqref{equation:shorthand:dpa-step}), \cref{lemma:pushforward:stability}, and \cref{lemma:compositionality:free-marginals}. 
Overall, $\probabilityInput{\varepsilon}{k}$ is an $\varepsilon$-optimal control input at $\mu_k$.
When $\varepsilon = 0$, the infima are attained and we obtain the optimal state-input distribution $\probabilityInput{\ast}{k}$.
\item
The statement follows from (ii), plugging in the given maps $\inputMap{}{}{\varepsilon/2}{k}$ and $\transportMap{\varepsilon/2}{k}$.
\ifbool{arxiv}{\qedhere}{}
\end{enumerate}
\end{proof}

\begin{lemma}
\label{lemma:borel_lifting}
For any Borel map $\phi: \statespace{} \to \inputspace{}$, the measure-valued map $h: \statespace{} \to \Pp{}{\inputspace{}}$ defined for every $x \in \statespace{}$ as $h(x) \coloneqq \delta_{\phi(x)}$ is Borel.
\end{lemma}
\begin{proof}
To show this, we can equivalently show that, for every $B \subseteq \inputspace{}$ Borel, the pre-image of the intervals $[a, +\infty]$, for $a \in \reals$, of $x \mapsto h(x)(B)$ is Borel. Define $h_B: \inputspace{} \to \nonnegativeReals$ for every $u \in \inputspace{}$ as $\inputVariable{}{}{}{} \mapsto h_B(\inputVariable{}{}{}{}) \coloneqq \delta_{\inputVariable{}{}{}{}}(B)$. Then
\begin{equation*}
h_B^ {-1} ([a, +\infty))
=
\begin{cases}
\emptyset\quad&\text{if }a  > 1\\
B\quad&\text{if }a \in (0, 1]\\
\inputspace{}&\text{otherwise}.
\end{cases}
\end{equation*}
% \begin{displaymath}
% h_B^ {-1} ([a, +\infty)) 
% =
% \begin{cases}
%     \emptyset \quad & \text{if } a > 1
%     \\
%     B & \text{if }a \in (0, 1]
%     \\
%     \inputspace{k} & \text{otherwise.}
% \end{cases}
% \end{displaymath}
In all cases, $h_B^ {-1} ([a, +\infty))$ is Borel set, and thus the map $h_B$ is Borel. Since the composition of Borel maps is a Borel map, $\comp{\phi}{h_B}$ is Borel. Therefore, the measure-valued map $h$ is Borel.
\end{proof}

\begin{proof}
[Proof of \cref{corollary:lifting:ot}]
The proof is analogous to \cref{theorem:lifting}. To express the cost-to-go $\costtogo{k}$ as a two-marginals optimal transport discrepancy, it suffices to replace~\cref{lemma:compositionality:free-marginals} with~\cref{lemma:compositionality:optimal-transport}. The simplified optimal control input $\probabilityInput{\varepsilon}{k}$ follows.
\end{proof}
\section{Conclusions}
\label{section:conclusion}

We showed that many discrete-time finite-horizon optimal control problems in probability spaces are multi-marginal optimal transport problems, whose transportation cost stems from an optimal control problem in the space on which the probability measures are defined.
This implies a separation principle: The optimal control strategy for a fleet of identical agents results from the optimal control strategy of each agent (how does one go from $x$ to $y$?) and an optimal transport problem (who goes from $x$ to $y$?).
We complemented our theoretical results with various examples.
Among others, our results back up many existing approaches in the literature which a priori formalize the distribution/fleet steering problems as an optimal transport problem and not as an optimal control problem in the probability space.
Our analysis is based on novel stability results for the multi-marginal optimal transport problem, whose study is of independent interest.

Future work will explore extensions to noisy dynamics and different cost functionals, the limit cases of the infinite horizon and continuous-time dynamics, and the practical impact of our theoretical results.
% \appendix
% \input{appendices/A_Preliminaries}

\bibliographystyle{siamplain}
\bibliography{references}

\begin{thebibliography}{10}

\bibitem{Martial2011}
{\sc M.~Agueh and G.~Carlier}, {\em {Barycenters in the Wasserstein Space}},
  SIAM Journal on Mathematical Analysis, 43 (2011), pp.~904--924.

\bibitem{Albi2015}
{\sc G.~Albi, L.~Pareschi, and M.~Zanella}, {\em {On the optimal control of
  opinion dynamics on evolving networks}}, in IFIP Conference on System
  Modeling and Optimization, Springer, 2015, pp.~58--67.

\bibitem{Aliprantis2006}
{\sc C.~D. Aliprantis and K.~C. Border}, {\em {Infinite Dimensional Analysis: a
  Hitchhiker's Guide}}, Springer, Berlin; London, 2006.

\bibitem{Altschuler2020}
{\sc J.~M. Altschuler and E.~Boix-Adsera}, {\em {Wasserstein barycenters can be
  computed in polynomial time in fixed dimension}}, Journal of Machine Learning
  Research, 22 (2021).

\bibitem{Ambrosio2008}
{\sc L.~Ambrosio, N.~Gigli, and G.~Savar{\'{e}}}, {\em {Gradient Flows: In
  Metric Spaces and in the Space of Probability Measures}}, Birkh{\"{a}}user
  Basel, 1~ed., 2008.

\bibitem{Aolaritei2022}
{\sc L.~Aolaritei, N.~Lanzetti, H.~Chen, and F.~D{\"{o}}rfler}, {\em
  {Distributional uncertainty propagation via optimal transport}}, arXiv
  preprint arXiv:2205.00343,  (2022).

\bibitem{Arque2022}
{\sc F.~Arqu{\'{e}}, C.~A. Uribe, and C.~Ocampo-Martinez}, {\em {Approximate
  Wasserstein attraction flows for dynamic mass transport over networks}},
  Automatica, 143 (2022), p.~110432.

\bibitem{Bakolas2016}
{\sc E.~Bakolas}, {\em {Optimal covariance control for discrete-time stochastic
  linear systems subject to constraints}}, in 55th Conference on Decision and
  Control, 2016, pp.~1153--1158.

\bibitem{Bakolas2017}
{\sc E.~Bakolas}, {\em {Covariance control for discrete-time stochastic linear
  systems with incomplete state information}}, in 2017 American Control
  Conference, 2017.

\bibitem{Bakolas2018b}
{\sc E.~Bakolas}, {\em {Constrained minimum variance control for discrete-time
  stochastic linear systems}}, Systems {\&} Control Letters, 113 (2018),
  pp.~109--116.

\bibitem{Bakolas2018}
{\sc E.~Bakolas}, {\em {Finite-horizon covariance control for discrete-time
  stochastic linear systems subject to input constraints}}, Automatica, 91
  (2018).

\bibitem{Balci2021}
{\sc I.~M. Balci and E.~Bakolas}, {\em {Covariance control of discrete-time
  Gaussian linear systems using affine disturbance feedback control policies}},
  in 60th Conference on Decision and Control, 2021, pp.~2324--2329.

\bibitem{Benamou2000}
{\sc J.-D. Benamou and Y.~Brenier}, {\em {A computational fluid mechanics
  solution to the Monge-Kantorovich mass transfer problem}}, Numerische
  Mathematik, 84 (2000), pp.~375--393.

\bibitem{Bertsekas2014}
{\sc D.~Bertsekas}, {\em {Abstract dynamic programming}}, Athena Scientific,
  2022.

\bibitem{Bertsekas1996}
{\sc D.~Bertsekas and S.~E. Shreve}, {\em {Stochastic Optimal Control The
  Discrete-time Case}}, vol.~5, Athena Scientific, 1996.

\bibitem{Bertsekas2017}
{\sc D.~P. Bertsekas}, {\em {Dynamic Programming and Optimal Control}}, vol.~I,
  Athena Scientific, 4~ed., 2017.

\bibitem{Bonnet2019b}
{\sc B.~Bonnet}, {\em {A Pontryagin Maximum Principle in Wasserstein spaces for
  constrained optimal control problems}}, ESAIM - Control, Optimisation and
  Calculus of Variations, 25 (2019).

\bibitem{Bonnet2021}
{\sc B.~Bonnet and H.~Frankowska}, {\em {Necessary optimality conditions for
  optimal control problems in Wasserstein spaces}}, Applied Mathematics and
  Optimization,  (2021).

\bibitem{Bonnet2019c}
{\sc B.~Bonnet and F.~Rossi}, {\em {The Pontryagin maximum principle in the
  Wasserstein space}}, Calculus of Variations and Partial Differential
  Equations, 58 (2019).

\bibitem{Bonnet2023}
{\sc B.~Bonnet-Weill and H.~Frankowska}, {\em {On the Viability and Invariance
  of Proper Sets under Continuity Inclusions in Wasserstein Spaces}}, 2023.

\bibitem{Cavagnari2020}
{\sc G.~Cavagnari and A.~Marigonda}, {\em {Attainability property for a
  probabilistic target in Wasserstein spaces}}, Discrete and Continuous
  Dynamical Systems - Series A, 41 (2020).

\bibitem{Cavagnari2023}
{\sc G.~Cavagnari, G.~Savar{\'{e}}, and G.~E. Sodini}, {\em {Dissipative
  probability vector fields and generation of evolution semigroups in
  Wasserstein spaces}}, Probability Theory and Related Fields, 185 (2023),
  pp.~1087--1182.

\bibitem{Chen2017Network}
{\sc Y.~Chen, T.~Georgiou, M.~Pavon, and A.~Tannenbaum}, {\em {Robust Transport
  over Networks}}, IEEE Transactions on Automatic Control, 62 (2017).

\bibitem{Chen2016}
{\sc Y.~Chen, T.~T. Georgiou, and M.~Pavon}, {\em {On the relation between
  optimal transport and Schr{\"{o}}dinger bridges: a stochastic control
  viewpoint}}, Journal of Optimization Theory and Applications, 169 (2016).

\bibitem{Chen2016a}
{\sc Y.~Chen, T.~T. Georgiou, and M.~Pavon}, {\em {Optimal steering of a linear
  stochastic system to a final probability distribution - part I}}, IEEE
  Transactions on Automatic Control, 61 (2016).

\bibitem{Chen2016b}
{\sc Y.~Chen, T.~T. Georgiou, and M.~Pavon}, {\em {Optimal steering of a linear
  stochastic system to a final probability distribution - part II}}, IEEE
  Transactions on Automatic Control, 61 (2016).

\bibitem{Chen2018c}
{\sc Y.~Chen, T.~T. Georgiou, and M.~Pavon}, {\em {Optimal steering of a linear
  stochastic system to a final probability distribution - part III}}, IEEE
  Transactions on Automatic Control, 63 (2018).

\bibitem{Chen2021}
{\sc Y.~Chen, T.~T. Georgiou, and M.~Pavon}, {\em {Optimal transport in systems
  and control}}, Annual Review of Control, Robotics, and Autonomous Systems, 4
  (2021).

\bibitem{Chen2018Network}
{\sc Y.~Chen, T.~T. Georgiou, M.~Pavon, and A.~Tannenbaum}, {\em {Efficient
  robust routing for single commodity network flows}}, IEEE Transactions on
  Automatic Control, 63 (2018).

\bibitem{Chen2020Network}
{\sc Y.~Chen, T.~T. Georgiou, M.~Pavon, and A.~Tannenbaum}, {\em {Relaxed
  schr{\"{o}}dinger bridges and robust network routing}}, IEEE Transactions on
  Control of Network Systems, 7 (2020).

\bibitem{Cormen2009}
{\sc T.~H. Cormen, C.~E. Leiserson, R.~L. Rivest, and C.~Stein}, {\em
  {Introduction to Algorithms, Third Edition}}, The MIT Press, 2009.

\bibitem{Cuturi2013}
{\sc M.~Cuturi}, {\em {Sinkhorn distances: Lightspeed computation of optimal
  transport}}, in Advances in Neural Information Processing Systems, vol.~26,
  2013.

\bibitem{Cuturi2013bary}
{\sc M.~Cuturi and A.~Doucet}, {\em {Fast computation of Wasserstein
  barycenters}}, in International Conference on Machine Learning, arXiv, 2014,
  pp.~685--693.

\bibitem{Figalli2007}
{\sc A.~Figalli}, {\em {Existence, uniqueness, and regularity of optimal
  transport maps}}, SIAM Journal on Mathematical Analysis, 39 (2007),
  pp.~126--137.

\bibitem{Flamary2021}
{\sc R.~Flamary, N.~Courty, A.~Gramfort, M.~Z. Alaya, A.~Boisbunon, S.~Chambon,
  L.~Chapel, A.~Corenflos, K.~Fatras, N.~Fournier, L.~Gautheron, N.~T.~H.
  Gayraud, H.~Janati, A.~Rakotomamonjy, I.~Redko, A.~Rolet, A.~Schutz,
  V.~Seguy, D.~J. Sutherland, R.~Tavenard, A.~Tong, and T.~Vayer}, {\em {POT:
  Python Optimal Transport}}, Journal of Machine Learning Research, 22 (2021),
  pp.~1--8.

\bibitem{Genevay2016}
{\sc A.~Genevay, M.~Cuturi, G.~Peyr{\'{e}}, and F.~Bach}, {\em {Stochastic
  optimization for large-scale optimal transport}}, in Advances in Neural
  Information Processing Systems, 2016.

\bibitem{Hindawi2010}
{\sc A.~Hindawi, J.-B. Pomet, and L.~Rifford}, {\em {Mass transportation with
  LQ cost functions}}, Acta Applicandae Mathematicae, 113 (2010), pp.~215--229.

\bibitem{Huang2020}
{\sc E.~Y. Huang, D.~Paccagnan, W.~Mei, and F.~Bullo}, {\em {Assign and
  appraise: Achieving optimal performance in collaborative teams}}, IEEE
  Transactions on Automatic Control,  (2022).

\bibitem{HudobadeBadyn2021}
{\sc M.~Hudoba~de Badyn, E.~Miehling, D.~Janak, B.~A{\c{c}}kme{\c{s}}e,
  M.~Mesbahi, T.~Ba{\c{s}}ar, J.~Lygeros, and R.~S. Smith}, {\em {Discrete-time
  linear-quadratic regulation via optimal transport}}, in 60th Conference on
  Decision and Control, 2021, pp.~3060--3065.

\bibitem{Kim2014}
{\sc Y.~H. Kim and B.~Pass}, {\em {A general condition for Monge solutions in
  the multi-marginal optimal transport problem}}, SIAM Journal on Mathematical
  Analysis, 46 (2014), pp.~1538--1550.

\bibitem{Krishnan2019}
{\sc V.~Krishnan and S.~Mart{\'{i}}nez}, {\em {Distributed online optimization
  for multi-agent optimal transport}}, arXiv preprint arXiv:1804.01572,
  (2019).

\bibitem{RLBenchmark2023}
{\sc C.~Laidlaw, S.~J. Russell, and A.~Dragan}, {\em {Bridging RL theory and
  practice with the effective horizon}}, Advances in Neural Information
  Processing Systems, 36 (2024).

\bibitem{Lanzetti2022}
{\sc N.~Lanzetti, S.~Bolognani, and F.~D{\"{o}}rfler}, {\em {First-order
  conditions for optimization in the Wasserstein space}}, arXiv preprint
  arXiv:2209.12197,  (2022).

\bibitem{Pass2015}
{\sc B.~Pass}, {\em {Multi-marginal optimal transport: theory and
  applications}}, ESAIM: Mathematical Modelling and Numerical
  Analysis-Mod{\'{e}}lisation Math{\'{e}}matique et Analyse Num{\'{e}}rique, 49
  (2015), pp.~1771--1790.

\bibitem{Combinatorics2019}
{\sc {Pavle Mladenovi{\'{c}}}}, {\em {Combinatorics: A Problem-Based
  Approach}}, Springer Cham, 1~ed., 2019.

\bibitem{Peyre2019}
{\sc G.~Peyr{\'{e}} and M.~Cuturi}, {\em {Computational optimal transport}},
  Foundations and Trends in Machine Learning, 11 (2019).

\bibitem{Santambrogio2015}
{\sc F.~Santambrogio}, {\em {Optimal Transport for Applied Mathematicians}},
  vol.~55, Springer, 2015.

\bibitem{Scarvelis2022}
{\sc C.~Scarvelis and J.~Solomon}, {\em {Riemannian metric learning via optimal
  transport}}, arXiv preprint arXiv:2205.09244,  (2022).

\bibitem{Schmitzer2019}
{\sc B.~Schmitzer}, {\em {Stabilized sparse scaling algorithms for entropy
  regularized transport problems}}, SIAM Journal on Scientific Computing, 41
  (2019).

\bibitem{Bahar2022}
{\sc B.~Ta{\c{s}}kesen, S.~Shafieezadeh-Abadeh, and D.~Kuhn}, {\em
  {Semi-discrete optimal transport: hardness, regularization and numerical
  solution}}, Mathematical Programming,  (2022).

\bibitem{Taskesen2022}
{\sc B.~Ta{\c{s}}kesen, S.~Shafieezadeh-Abadeh, D.~Kuhn, and K.~Natarajan},
  {\em {Discrete Optimal Transport with Independent Marginals is {\#}P-Hard}},
  arXiv preprint arXiv:2203.01161,  (2022).

\bibitem{Terpin2021}
{\sc A.~Terpin, S.~Fricker, M.~Perez, M.~Hudoba~de Badyn, and
  F.~D{\"{o}}rfler}, {\em {Distributed feedback optimisation for robotic
  coordination}}, in 2022 American Control Conference, 2022, pp.~3710--3715.

\bibitem{Terpin2022}
{\sc A.~Terpin, N.~Lanzetti, B.~Yardim, F.~D{\"{o}}rfler, and G.~Ramponi}, {\em
  {Trust region policy optimization with optimal transport discrepancies:
  Duality and algorithm for continuous actions}}, in Advances in Neural
  Information Processing Systems, 2022.

\bibitem{Villani2007}
{\sc C.~Villani}, {\em {Optimal Transport: Old and New}}, Springer, Berlin,
  Heidelberg, 1~ed., 2007.

\bibitem{Lindheim2022}
{\sc J.~von Lindheim}, {\em {Approximative algorithms for multi-marginal
  optimal transport and free-support Wasserstein barycenters}}, arXiv preprint
  arXiv:2202.00954,  (2022).

\bibitem{Zardini2021}
{\sc G.~Zardini, N.~Lanzetti, M.~Pavone, and E.~Frazzoli}, {\em {Analysis and
  control of autonomous mobility-on-demand systems}}, Annual Review of Control,
  Robotics, and Autonomous Systems, 5 (2021).

\end{thebibliography}
\end{document}

\maketitle
\section{Optimal transport examples}
\begin{example}[Wasserstein distance between Gaussian measures~{\cite{Talagrand1996}}]
\label{example:optimaltransport:gaussian}
If $p = 2$, $\mu = \gaussian{m}{\Sigma}$ and $\mu^\ast = \gaussian{m_\ast}{\Sigma_\ast}$, then:
\begin{enumerate}
    \item The optimal transport map is $T^\ast: X\to X$ is defined $\mu\ae{}$ as 
    \begin{displaymath}
    T^\ast(x) = \Sigma^{-1/2}(\Sigma^{1/2}\Sigma_\ast\Sigma^{1/2})^{1/2}\Sigma^{-1/2}(x - m) + m_\ast;\text{ and}
    \end{displaymath}
    \item $\wassersteinDistance{2}{\mu}{\mu^\ast}^2 = \norm{m - m_\ast}^2 + \trace\left[\Sigma + \Sigma_\ast - 2\left(\Sigma^{1/2}\Sigma_\ast\Sigma^{1/2}\right)^{1/2}\right]$.
\end{enumerate}
Notice that if $X = \reals$, then the Wasserstein distance becomes:
\begin{displaymath}
\wassersteinDistance{2}{\mu}{\mu^\ast}^2 = (m - m_\ast)^2 + (\Sigma - \Sigma_\ast)^2.
\end{displaymath}
That is, it is the sum of the squared difference of the mean and variance of the two Gaussian distributions.
\end{example}
\begin{example}[Empirical measures]\label{example:optimaltransport:empiricalmeasures}
Let $\mu \in \spaceProbabilityBorelMeasures{X}$ and $\nu \in \spaceProbabilityBorelMeasures{Y}$ be two finite empirical measures:
\begin{displaymath}
\begin{aligned}
\mu &= \sum_{i = 1}^N \alpha_i \delta_{x_i},\quad x_i \in X, \quad \alpha_i \geq 0, \quad \sum_{i = 1}^N \alpha_i = 1,\\
\mu^\ast &= \sum_{i = 1}^M \beta_i \delta_{y_i},\quad y_i \in Y, \quad \beta_i \geq 0, \quad \sum_{i = 1}^M \beta_i = 1.
\end{aligned}
\end{displaymath}
Then, the cost function can be written as a vector $c \in \nonnegativeRealsBar^{NM}$: the $\nth{h}$ entry is $c_{h} = c(x_i, y_j)$, where $h = N(i-1)+j$, $i \in \{1,\ldots,N\}$ and $j \in \{1,\ldots,M\}$.

The optimal transport problem can thus be formulated as a linear program:
\begin{displaymath}
\begin{aligned}
\min\,&c^T v\\
\mathrm{s.t. }\,
&v \in [0,1]^{NM}\\
&(1_M^T \otimes I_N)v = \alpha\\
&(I_M \otimes 1_{N}^T)v = \beta,
\end{aligned}
\end{displaymath}
where $1_K$ is the vector of ones of length $K$, $I_K$ is the $K$ dimensional identity matrix, and $\otimes$ is the Kronecker product.
In particular, it admits an optimal transport plan:
\begin{displaymath}
\plan{\mu}(x_i,y_j) = v_{N(i-1)+j}.
\end{displaymath}
If $N = M$, it can be shown that there exists an optimal transport map. Namely, there is a solution $v \in \{0,1\}^{NM}$.
\end{example}

\begin{example}
Consider the following problem setting:

\emph{
Assume we want to regulate a process multiple times, possibly infinitely many. Each realization of the process has an initial state that is distributed as a Gaussian distribution: $x_0 \sim \mu_0 = \gaussian{m_0}{\Sigma_0}$. The process evolves with the linear dynamics $x_{k+1} = f_k(x_k,u_k(x_k)) = A_k x_k + B_k u_k(x_k) + \omega$, $\omega \sim \gaussian{m_{k\omeg}}{\Sigma_{k\omega}}$, where $u_k$ is an affine control law: 
\begin{displaymath}
u_k \in \probabilityInputSpace{k} = \{\statespace{k}\ni x\mapsto K_kx + \bar{u}_k \in \inputspace{k} \st K_k \in \reals^{p\times n}, \bar{u}_k \in \reals^p\} \subset \C{0}{\statespace{k}}{\inputspace{k}}.
\end{displaymath} 
Not only we want each realization to be regulated optimally according to some quadratic cost, but we aim also at ensuring that all the realizations behave similarly to the nominal one. Namely, we want to reduce the variance of the different realizations at each stage.
}

In our framework, the state is the probability distribution of the possible realizations of the process, which at every stage $k$ is always a Gaussian distribution with mean and the covariance:
\begin{displaymath}
\begin{aligned}
m_{k+1} &= (A_{k} + B_{k} K_{k})m_{k} + B_{k}\bar{u}_{k}\\
\Sigma_{k+1} &= (A_{k}+B_{k}K_{k})\Sigma_{k}(A_{k}+B_{k}K_{k})^T.
\end{aligned}
\end{displaymath}

Then, we consider the cost terms:
\begin{displaymath}
\begin{aligned}
\stagecost{k}(\mu_{k}, (K_{k}, \bar{u}_{k})) 
&= \underbrace{\wassersteinDistance{2{Q_{1k}}}{\mu_{k}}{\delta_0}^2 + \variance{{\mu_{k}}}{q_{2k}^Tx}}_{\text{State penalty}} +  \underbrace{\wassersteinDistance{2{R_{k}}}{\pushforward{{K_{k}\cdot + \bar{u}_{k}}}{\mu_{k}}}{\delta_0}^2}_{\text{Input effort}}\\
&= \trace\left[Q_{1k} M_{k} + Q_{2{k}} \Sigma_{k} + K_{k}^TR_{k}K_{k}M_{k} + R_{k}\bar{U}_{k} + 2R_{k}\bar{u}_{k}m_{k}^TK_{k}^T\right]\\
\text{and}\quad\terminalcost(\mu_N) 
&= \wassersteinDistance{2{P_{1N}}}{\mu_N}{\delta_0}^2 + \variance{\mu_N}{p_{2N}^Tx}\\
&= \trace[P_{1N}M_N + P_{2N}\Sigma_N],
\end{aligned}
\end{displaymath}
where we replace the inner products with the $\trace$, $M_{k} = m_{k} m_{k}^T$, $\bar{U}_{k} = \bar{u}_{k}\bar{u}_{k}^T$ and $Q_{2k} = Q_{1k} + q_{2k}q_{2k}^T$.
\end{example}
\section{Continuity of the pushforward mapping: counterexamples}
\label{section:pushforward:counterexamples}
We now study the continuity properties of $H$.
\begin{proposition}\label{proposition:pushforward:continuous:partial}
Assume that $h(\cdot, c)$ is continuous for all $c \in C$ and $h(x, \cdot)$ is continuous for all $x \in A_1$. Then, we have that:
\begin{enumerate}
    \item $H(\mu,b_n) \narrowconvergence H(\mu,b)$; and
    \item $H(\mu_n, b) \narrowconvergence H(\mu,b)$.
\end{enumerate}
\end{proposition}
\begin{proof}
\begin{enumerate}
    \item 
    For all $\phi \in \Cb{A_1}$
\begin{equation*}
\begin{aligned}
    \lim_{n\to\infty} \int_{A_1}\phi(x)\d H(\mu,b_n)(x) &= \lim_{n\to\infty} \int_{A_1}\phi(h(x, b_n(x)))\d\mu(x)\\
    &\overset{\heartsuit}{=}\int_{A_1}\lim_{n\to\infty}\phi(h(x, b_n(x)))\d\mu(x)\\
    &\overset{\diamondsuit}{=}\int_{A_1}\phi(h(x, b(x)))\d\mu(x)\\
    &=\int_{A_1}\phi(x)\d H(\mu,b)(x),
\end{aligned}
\end{equation*}
where in:
\begin{itemize}
    \item[$\heartsuit$] we used $\phi \in \Cb{A_1}$, $h(\cdot, b_n(\cdot)) \eqqcolon r_n$ and $r := h(\cdot, b(\cdot))$ are continuous and thus, $\comp{r_n}{\phi} \in \Cb{A_1}$, $(\comp{r_n}{\phi})(x) \to (\comp{r}{\phi})(x)$. Then, we used the dominated convergence result~\cite[Theorem 11.21]{Aliprantis2006}; and
    \item[$\diamondsuit$] we observed $\phi \in \Cb{A_1}$. Hence, 
    \begin{displaymath}
    b_n\to b \Rightarrow h(x, b_n(x)) \to h(x, b(x)) \Rightarrow \phi(h(x,b_n(x))) \to \phi(h(x,b(x))).
    \end{displaymath}
\end{itemize}
    \item
Similarly, since $\phi(h(\cdot, b(\cdot)))\in \Cb{A_1}$ for all $b \in B$, by narrow convergence of $\mu_n$ we have
\begin{equation*}
\begin{aligned}[b]
    \lim_{n\to\infty} \int_{A_1}\phi(x)\d H(\mu_n,b)(x) &= \lim_{n\to\infty} \int_{A_1}\phi(h(x, b(x)))\d\mu_n(x)\\
    &=\int_{A_1}\phi(h(x, b(x)))\d\mu(x)\\
    &=\int_{A_1}\phi(x)\d H(\mu,b)(x).
\end{aligned}
\end{equation*}
\end{enumerate}
\end{proof}

\Cref{proposition:pushforward:continuous:partial} proves that the pushforward mapping $H$ with a \emph{fixed} argument is continuous in the other. However, in general it is not continuous in both:
\begin{example}\label{example:pushforward:notcontinuous}
Consider $h: \reals \times \reals \to \reals$ defined as:
\begin{displaymath}
h(x, c) = 
\begin{cases}
\frac{xc}{x^2 + c^2}\quad&\text{if } (x, c) \neq (0,0),\\
0&\text{otherwise}.
\end{cases}
\end{displaymath}

Clearly, $h(\cdot, c)$ and $h(x, \cdot)$ are continuous for all $x, c \in \reals$. However, $h$ is not continuous with respect to the product topology:
\begin{displaymath}
\lim_{n \to \infty} h\left(\frac{1}{n}, \frac{1}{n}\right) = \frac{1}{2} \neq 0 = h(0,0).
\end{displaymath}

We now show that there exists two converging sequences
\begin{displaymath}
\begin{aligned}
&(\mu_n)_{n \in \naturals} \subset \spaceProbabilityBorelMeasures{\reals}, \mu_n \narrowconvergence \mu \in \spaceProbabilityBorelMeasures{\reals}\\
&(b_n)_{n \in \naturals} \subset B, b_n(\cdot) = c_n \in \reals, b_n \to b(\cdot) = c \in \reals
\end{aligned}
\end{displaymath}
such that $\pushforward{h(\cdot, c_n)}{\mu_n} \not\narrowconvergence \pushforward{h(\cdot, c)}{\mu}$.
Consider for instance $c_n = 1/n \to 0$ and $\mu_n = \delta_{1/n} \narrowconvergence \delta_0$. We have that:

\begin{displaymath}
\pushforward{h(\cdot, c_n)}{\mu_n} = \pushforward{\left(\frac{x/n}{x^2 + (1/n)^2}\right)}{\delta_{1/n}} = \delta_{1/2} \neq \delta_0 = \pushforward{0}{\delta_0} = \pushforward{h(\cdot, 0)}{\delta_0}.
\end{displaymath}
\end{example}

\section{Simplified proof of~\cref{theorem:pushforward:continuity}}
\label{proof:composition:continuity:alternative}
\begin{proof}[Proof of~\cref{theorem:pushforward:continuity}]
Fix a compact $K \subseteq A_1$.
Since $C$ is metric, we can consider the product metric space endowed with the distance $d_{A_1\times C}$; $\restrictDomain{h}{K\times C}$ is uniformly continuous, because continuous on a compact. Since $b_n(x) \to b(x)$, we have that:
\begin{equation*}
\begin{aligned}
\forall \varepsilon > 0\, \exists \delta>0
\text{ s.t. } &\delta > d_{A_1\times C}((x,b(x)),(x, b_n(x))) = d_{C}(b(x),b_n(x)) \\
&\Rightarrow d_{A_2}(h(x,b(x)),h(x,b_n(x))) < \varepsilon.
\end{aligned}
\end{equation*}
From the convergence of $b_n \to b \in B$ we conclude the existence of $N$ such that 
\begin{displaymath}
\forall n > N,\,d_{A_2}(h(x,b(x)),h(x,b_n(x))) < \varepsilon.
\end{displaymath}
Defining $r_n(x)\coloneqq h(x,b_n(x))$ and $r(x)\coloneqq h(x,b(x))$ we have a uniformly convergent sequence of functions converging to a continuous function on $K$. Since $K$ is arbitrary, we have local uniform convergence. 
Hence, we can apply~\cite[Lemma 5.2.1]{Ambrosio2008} to establish the desired result.
\end{proof}
\section{Proof of \cref{theorem:dp:existence}}
\label{proof:dp}
First, we need the following lemma, which proves the existence of a minimizer at every stage:
\begin{lemma}[Existence of the minimizer in~\cref{definition:dynamicprogrammingequation}]\label{lemma:dp:existence}
Assume $\costtogo{k+1}$ satisfies~\cref{hypothesis:terminalcost}. Then, the dynamic programming algorithm in~\cref{definition:dynamicprogrammingequation} admits a minimizer.
\end{lemma}
The proof is standard and it is along the lines of \cite{Haneveld1980}, we write it for completeness:
\begin{proof}
\label{proof:dp:existence:stage}
\begin{enumerate}
    \item 
Fix $(w_k,\nu_k) \in \functorprobabilitynoise{\noisespace{k}}$. Then, using~\cref{hypothesis:statespace,hypothesis:inputspace,hypothesis:dynamics}, since $\costtogo{k+1}$ is $\lowersemicontinuous$, we conclude that $\comp{\dynamicsbig{k,w_k,\nu_k}}{\costtogo{k+1}}$ is also $\lowersemicontinuous$, by~\cref{corollary:pushforward:lsc}.
    \item
Because of~\cref{hypothesis:stagecost}-i, $\stagecost{k,w_k}$ is $\lowersemicontinuous$. Then, $\stagecost{k, w_k} + \comp{\dynamicsbig{k,w_k,\nu_k}}{\costtogo{k+1}}$ is $\lowersemicontinuous$~\cite[Proposition B.1]{Puterman2014}. Hence, for any fixed $u_k \in \functorprobabilityinput{\inputspace{k}}$, $\comp{\dynamicsbig{k,\mu_k,u_k,\nu_k}}{\costtogo{k+1}}$, $\stagecost{k,\mu_k,u_k}$ and their sum $\stagecost{k,\mu_k,u_k} + \comp{\dynamicsbig{k,\mu_k,u_k,\nu_k}}{\costtogo{k+1}}$ are Borel-measurable.%~\cite[Theorem 4.27]{Aliprantis2006}.
\item Let $(u_{k,n})_{n \in \naturals}\subset \functorprobabilityinput{\inputspace{k}}$ be a converging sequence to $u_k$; i.e., $u_{k,n} \to u_{k}$. It holds that:
    \begin{displaymath}
    \begin{aligned}
        \heartsuit &\coloneqq  \costfunctional{\noisespace{k,g}}{\stagecost{k,\mu_k}(u_k, w_k)+(\comp{\dynamicsbig{k,\mu,\nu_k}}{\costtogo{k+1}})(u_k, w_k)}{\xi_k}{w_k}\\
        \overset{\text{(ii)}}&{\leq}
        \costfunctional{\noisespace{k,g}}{\liminf_{n \to \infty}\stagecost{k,\mu_k}(u_{k,n},w_k)+(\comp{\dynamicsbig{k,\mu_k,\nu_k}}{\costtogo{k+1}})(u_{k,n},w_k)}{\xi_k}{w_k}\\
        \overset{\diamondsuit}&{\leq}
        \liminf_{n \to \infty}\costfunctional{\noisespace{k,g}}{\stagecost{k,\mu_k}(u_{k,n},w_k)+(\comp{\dynamicsbig{k,\mu_k,\nu_k}}{\costtogo{k+1}})(u_{k,n},w_k)}{\xi_k}{w_k},
    \end{aligned}
    \end{displaymath}
    where in $\diamondsuit$ we used Fatou's Lemma~\cite[Lemma 11.20]{Aliprantis2006}, in view of (c). We thus proved that $\heartsuit$ is $\lowersemicontinuous$ with respect to $u$. In view of~\cref{remark:coercitive:restriction}, we can restrict $\functorprobabilityinput{\inputspace{k}}$ to some compact subset $\widetilde{\functorprobabilityinput{\inputspace{k}}}$ keeping the same optimal value (otherwise any $u_k$ yields an infinite cost). Then, we can conclude that there exists a minimizer $u_k^\ast$ invoking~\cite[Theorem B.2]{Puterman2014}. 
    \end{enumerate}
\end{proof}

We are now ready to prove the extension of the results in \cite{Bertsekas2014,Haneveld1980} to our setting:
\begin{proof}[Proof of \cref{theorem:dp:existence}]
We prove the result in two steps:
\begin{enumerate}[label=\arabic*.]
    \item We prove that the recursion in~\cref{definition:dynamicprogrammingequation} is well posed; in particular, $\costtogo{k}$ recursively satisfies~\cref{hypothesis:terminalcost}; and
    \item we prove that solving the recursion in~\cref{definition:dynamicprogrammingequation} solves~\cref{problem:finitehorizon} as well.
\end{enumerate}
Clearly, the two together prove the statement.
\begin{enumerate}[label=\arabic*.]
    \item 
We proceed inductively to show that $\costtogo{k}$ satisfies~\cref{hypothesis:terminalcost}. The case $k = N$ is automatically satisfied. 

Let $(\mu_{k,n})_{n \in \naturals}\subset \functorprobabilitystate{\statespace{k}}$ be a (narrowly) converging sequence to $\mu_k$; i.e., $\mu_{k,n} \narrowconvergence \mu_k$. 
In view of \cref{lemma:dp:existence}, there exists a sequence $(u_{k,n}^\ast)_{n\in \naturals} \subset \probabilityInputSpace{k}$ of minimizers for each $\mu_{k,n}$, respectively, and $u_k^\ast \in \probabilityInputSpace{k}$ minimizer for $\mu_k$.

Define $L \coloneqq \liminf_{n \to \infty} \costtogo{k}(\mu_{k,n}) \in \nonnegativeRealsBar$ and $h_k: \functorprobabilitystate{\statespace{k}}\times \probabilityInputSpace{k} \to \nonnegativeRealsBar$ as
\begin{displaymath}
h_k(\bar{\mu},\bar{u}) = \costfunctional{\noisespace{k,g}}{\stagecost{k,\nu_k}(\bar{\mu},\bar{u},w_k)+(\comp{\dynamicsbig{k,\nu_k}}{\costtogo{k+1}})(\bar{\mu},\bar{u},w_k)}{\xi_k}{w_k}.
\end{displaymath}

For any converging sequence $(\gamma_n, a_n) \subset \functorprobabilitystate{\statespace{k}}\times\probabilityInputSpace{k}$, $\gamma_n \narrowconvergence \gamma$ and $a_n \to a$, it holds:
\begin{multline*}
\costfunctional{\noisespace{k,g}}{\stagecost{k}(\gamma,a,w_k)+(\comp{\dynamicsbig{k,\nu_k}}{\costtogo{k+1}})(\gamma,a,w_k)}{\xi_k}{w_k}\\
\leq
\costfunctional{\noisespace{k,g}}{\liminf_{n \to \infty}
\stagecost{k,\nu_k}(\gamma_n,a_n,w_k)+(\comp{\dynamicsbig{k,\nu_k}}{\costtogo{k+1}})(\gamma_n,a_n,w_k)}{\xi_k}{w_k}\\
\overset{\diamondsuit}{\leq}
\liminf_{n \to \infty}\costfunctional{\noisespace{k,g}}{\stagecost{k,\nu_k}(\gamma_n,a_n,w_k)+(\comp{\dynamicsbig{k,\nu_k}}{\costtogo{k+1}})(\gamma_n,a_n,w_k)}{\xi_k}{w_k},
\end{multline*}
where in $\diamondsuit$ we used~\cite[Lemma 11.20 (Fatou's Lemma)]{Aliprantis2006}. That is, $h_k$ is $\lowersemicontinuous$: $h_k(\gamma,a)\leq\liminf_{n \to \infty}h_k(\gamma_n,a_n)$.

If $L = \infty$, then trivially $\liminf_{n \to \infty} \costtogo{k}(\mu_{k,n}) \geq \costtogo{k}(\mu_k)$. Hence, let $L < \infty$.

Since $h_k$ is $\lowersemicontinuous$, for all $m \in \naturals$ there exists $n_m \in \naturals$ such that $n_{m+1} > n_m$ and
\begin{displaymath}
h_k(\mu_{k,n_m}, u_{k,n_m}^\ast) \geq h(\mu_k, u_{k,n_m}^\ast) - \frac{1}{m}.
\end{displaymath}

Therefore, $\liminf_{m \to \infty}h_k(\mu_k, u_{k,n_m}^\ast) \leq L$ and thus, from \cref{hypothesis:stagecost}, there exists $\Lambda \subseteq \naturals$, $K \subseteq \probabilityInputSpace{k}$ compact, such that $h_k(\mu_k, u) \leq L$ for all $u \in K$ and for all $m \in \Lambda$ $u_{n_m}^\ast \in K$. In particular, there exists a converging subsequence of $(u_n^\ast)_{n \in \naturals}$, indexed by $\Lambda' \subseteq \naturals$. Let $\lim_{n \in \Lambda'} u_n^\ast = \bar{u}$. We have
\begin{displaymath}
\begin{aligned}
\liminf_{n \to \infty}\costtogo{k}(\mu_{k,n}) 
&\geq 
\liminf_{n \in \Lambda'} h(\mu_{k,n}, u_{k,n}^\ast)\\
&= 
h(\mu_k, \bar{u})
\geq 
\min_{u \in \probabilityInputSpace{k}} h_k(\mu_k, u)
= 
\costtogo{k}(\mu_k).
\end{aligned}
\end{displaymath}
That is, $\costtogo{k}$ is $\lowersemicontinuous$.
    \item
Follows from a simple inductive argument along the lines of \cite[\S 4.2 Finite horizon problems]{Bertsekas2014} and \cite{Haneveld1980}: 
We introduce the collection of auxiliary maps $\heartsuit_k$, defined as $\heartsuit_N = \terminalcost$ and
\begin{displaymath}
\heartsuit_{k}(\mu_k) \coloneqq \inf_{\indexedVar{u}{k}{N-1} \in \indexedVar{\functorprobabilityinput{\inputspace{}}}{k}{N-1}} \costfunctional{\indexedVar{\noisespace{}}{k,g}{N-1}}{\terminalcost(\mu_N) + \sum_{h = k}^{N-1}\stagecost{h}(\mu_h, u_h, w_h)}{\indexedVar{\xi}{k}{N-1}}{\indexedVar{w}{k}{N-1}},
\end{displaymath}
subject to the dynamic and space constraints. Notice that $\heartsuit_0 = \costtogo{}$. With the induction hypothesis $\heartsuit_{k+1} = \costtogo{k+1}$, we have that:
\begin{displaymath}
\begin{aligned}
    \heartsuit_k(\mu_k) 
    &= \inf_{\indexedVar{u}{k}{N-1} \in \indexedVar{\functorprobabilityinput{\inputspace{}}}{k}{N-1}} \costfunctional{\noisespace{k,g}}{
            \costfunctional{\indexedVar{\noisespace{}}{k+1,g}{N-1}}{
                \terminalcost(\mu_N) \\
    &\qquad\qquad\qquad+ \sum_{h = k}^{N-1}\stagecost{h}(\mu_h, u_h, w_h)
            }{\indexedVar{\xi}{k+1}{N-1}}{\indexedVar{w}{k+1}{N-1}}}{\xi_{k}}{w_{k}}\\
    &= \inf_{\indexedVar{u}{k}{N-1} \in \indexedVar{\functorprobabilityinput{\inputspace{}}}{k}{N-1}} \costfunctional{\noisespace{k,g}}{
            \stagecost{k}(\mu_k, u_k, w_k) \\&\qquad+ \costfunctional{\indexedVar{\noisespace{}}{k+1,g}{N-1}}{
                \terminalcost(\mu_N) \\
    &\qquad\qquad\qquad+ \sum_{h = k+1}^{N-1}\stagecost{h}(\mu_h, u_h, w_h)
            }{\indexedVar{\xi}{k+1}{N-1}}{\indexedVar{w}{k+1}{N-1}}}{\xi_{k}}{w_{k}}\\
    &=\inf_{u_{k} \in \functorprobabilityinput{\inputspace{k}}} \costfunctional{\noisespace{k,g}}{
            \stagecost{k}(\mu_k, u_k, w_k) \\
    &\qquad+ \inf_{\indexedVar{u}{k+1}{N-1} \in \indexedVar{\functorprobabilityinput{\inputspace{}}}{k+1}{N-1}} \costfunctional{\indexedVar{\noisespace{}}{k+1,g}{N-1}}{
                \terminalcost(\mu_N) \\
    &\qquad\qquad+ \sum_{h = k+1}^{N-1}\stagecost{h}(\mu_h, u_h, w_h)
            }{\indexedVar{\xi}{k+1}{N-1}}{\indexedVar{w}{k+1}{N-1}}}{\xi_{k}}{w_{k}}\\
    &=\inf_{u_{k} \in \inputspace{k}} \costfunctional{\noisespace{k,g}}{
            \stagecost{k}(\mu_k, u_k, w_k) + \heartsuit_{k+1}(\mu_{k+1})}{\xi_{k}}{w_{k}}\\
    &=\inf_{u_{k} \in \inputspace{k}} \costfunctional{\noisespace{k,g}}{
            \stagecost{k}(\mu_k, u_k, w_k) + \costtogo{k+1}(\mu_{k+1})}{\xi_{k}}{w_{k}}\\
    &=\costtogo{k}(\mu_k).
\end{aligned}
\end{displaymath}
Hence, proceeding backwards, we find $\costtogo{} = \heartsuit_0 = \costtogo{0}$. Therefore, the solutions of~\cref{definition:dynamicprogrammingequation} coincides with the ones of~\cref{problem:finitehorizon}, and we conclude the proof. 
\end{enumerate}
\end{proof}

\section{Derivations for examples in \cref{section:dp:examples}}

\begin{proof}[Detailed steps for \cref{example:linearsystems:lqr}]
\label{proof:derivation:lqr}
We proceed by induction and we assume that the cost to go has the expression:
\[
\costtogo{k+1}(\mu_{k+1}) = \wassersteinDistance{2{P_{k+1}}}{\mu_{k+1}}{\delta_0}^2.
\]

Then, at the $\nth{k}$ stage, the cost to go has the expression:
\[
\begin{aligned}
\costtogo{k}(\mu_{k}) 
&= \min_{K_{k}, \bar{u}_{k}}\wassersteinDistance{2{Q_{k}}}{\mu_{k}}{\delta_0}^2 
+ \wassersteinDistance{2{R_{k}}}{\pushforward{(K_{k}\cdot + \bar{u}_{k})}{\mu_{k}}}{\delta_0}^2
+ \wassersteinDistance{2{P_{k+1}}}{\mu_{k+1}}{\delta_0}^2\\
&= \min_{K_{k}, \bar{u}_{k}}\norm{m_{k}}_{Q_{k}}^2 + \norm{K_{k}m_{k} + \bar{u}_{k}}_{R_{k}}^2 + \norm{(A_{k}+B_{k}K_{k})m_{k} + B_{k}\bar{u}_{k}}_{P_{k+1}}^2 \\
& + \trace\left[Q_{k}\Sigma_{k} + K_{k}^TR_{k}K_{k}\Sigma_{k} + P_{k+1}(A_{k}+B_{k}K_{k})\Sigma(A_{k}+B_{k}K_{k})^T\right].
\end{aligned}
\]

Differentiating with respect to $K_{k}$ and $\bar{u}_{k}$ we obtain (see \cite{Pedersen2008})
\[
\begin{aligned}
0 \overset{!}{=} \pdv{\costtogo{k}(\mu_{k})}{\bar{u}_{k}} 
&= 2R_{k}(K_{k}m_{k} + \bar{u}_{k}) + 2B_{k}^TP_{k+1}((A_{k}+B_{k}K_{k})m_{k} + B_{k}\bar{u}_{k})\\
0 \overset{!}{=} \pdv{\costtogo{k}(\mu_{k})}{K_{k}} 
&= 2R_{k}K_{k}\Sigma_{k} + 2B_{k}^TP_{k+1}B_{k}K_{k}\Sigma_{k}\\
&+ 2B_{k}^TP_{k+1}A_{k}\Sigma_{k} + \cancel{\pdv{\costtogo{k}(\mu_{k})}{\bar{u}_{k}}}m_{k}^T
\end{aligned}
\]
and thus:
\[
\begin{aligned}
\bar{u}_{k}^\ast &= 0,\\
K_{k}^\ast &= -(R_{k} + B_{k}^TP_{k+1}B_{k})B_{k}^TP_{k+1}A_{k},
\end{aligned}
\]
recovering the traditional \gls*{acr:lqr} feedback law. Substituting the optimal values in the cost to go, we complete the induction:
\[
\begin{aligned}
\costtogo{k}(\mu_{k}) 
&= m_{k}^T(Q_{k} + {K_{k}^\ast}^TR_{k}K_{k}^\ast + (A_{k}+B_{k}K_{k}^\ast)^TP_{k+1}(A_{k}+B_{k}K_{k}^\ast))m_{k} \\
&+ \trace\left[(Q_{k} + {K_{k}^\ast}^TR_{k}K_{k}^\ast + (A_{k}+B_{k}K_{k}^\ast)^TP_{k+1}(A_{k}+B_{k}K_{k}^\ast))\Sigma_{k}\right]\\
&= m_{k}^TP_{k}m_{k} + \trace(P_{k}\Sigma_{k}) = \wassersteinDistance{2P_{k}}{\mu}{\delta_0}^2,
\end{aligned}
\]
where $P_{k}$ is the solution of the traditional \gls*{acr:dare}:
\[
P_{k}=Q_{k} + A_{k}^TP_{k}A_{k}-(A_{k}^TP_{k+1}B_{k})(R_{k}+B_{k}^TP_{k+1}B_{k})^{-1}(B_{k}^TP_{k+1}A_{k}).
\]
The cost to go has thus the form 
\[
\costtogo{k}(\mu_{k}) = \wassersteinDistance{2{P_{k}}}{\mu_{k}}{\delta_0}^2.
\]
\end{proof}

\begin{proof}[Detailed steps for \cref{example:linearsystems:varianceawarelqr}]
\label{proof:derivation:lqr:varianceaware}
We proceed by induction with the ansatz:
\[
\costtogo{k+1}(\mu_{k+1}) = \trace[P_{1k+1}M_{k+1} + P_{2k+1}\Sigma_{k+1}].
\]
Then, the cost to go at the $\nth{k}$ stage is:
\[
\begin{aligned}
\costtogo{k}(\mu_{k}) 
= \trace\biggl[&Q_{1k} M_{k} + Q_{2k} \Sigma_{k} + K_{k}^TR_{k}K_{k}M_{k} + R_{k}\bar{U}_{k} + R_{k}\bar{u}_{k}m_{k}^TK_{k}^T\\
&+ P_{1k+1}(A_{k}+B_{k}K_{k})M_{k}(A_{k}+B_{k}K_{k})^T + P_{1k+1}B_{k}\bar{U}_{k}B_{k}^T \\
&+ 2 P_{1k+1}(A_{k}+B_{k}K_{k})m_{k}\bar{u}_{k}B_{k}^T \\
&+ P_{2k+1}(A_{k}+B_{k}K_{k})\Sigma_{k}(A_{k}+B_{k}K_{k})^T\biggr].
\end{aligned}
\]

Differentiating with respect to $\bar{u}_{k}$ \cite{Pedersen2008}:
\[
0 \overset{!}{=} \pdv{\costtogo{k}(\mu_k)}{\bar{u}_k} = R_{k}K_{k}\bar{u}_{k} + R_{k}K_{k}m_{k} + B_{k}^TP_{1k+1}B_{k}\bar{u}_{k} + B_{k}^TP_{1k+1}(A_{k}+B_{k}K_{k})m_{k},
\]
which gives 
\[
\begin{aligned}
\bar{u}_{k}^\ast 
&= -(R_{k}+B_{k}^TP_{1k+1}B_{k})^{-1}((R_{k} + B_{k}^TP_{1k+1}B_{k})K_{k}m_{k} + B_{k}^TP_{1k+1}A_{k}m_{k}) \\
&= -K_{k}m_{k} + K_{1k}^\ast,
\end{aligned}
\]

where $K_{1k}^\ast$ is the classic \gls*{acr:lqr} with matrices $A_{k}, B_{k}, R_{k}, P_{1k+1}$. Namely, we obtain a control law of the form:
\[
u^\ast = K_{k}(x - m_{k}) + K_{1k}^\ast m_{k}.
\]

Differentiating now with respect to $K_{k}$:
\[
\begin{aligned}
0 \overset{!}{=} \pdv{\costtogo{k}(\mu_k)}{K_k} 
&= R_{k}K_{k}^\ast(M_{k}+\Sigma_{k}) + R_{k}\bar{u}_{k}m_{k}^T \\
&+ B_{k}^TP_{1k+1}B_{k}K_{k}^\ast M_{k} + B_{k}^TP_{1k+1}A_{k}M_{k} + B_{k}^TP_{1k+1}B_{k}\bar{u}_{k}m_{k}^T \\
&+ B_{k}^TP_{2k+1}B_{k}K_{k}^\ast\Sigma_{k} + B_{k}^TP_{2k+1}A_{k}\Sigma_{k},
\end{aligned}
\]
and substituting $\bar{u}_{k}^\ast$ we obtain:
\[
0 = (R_{k} + B_{k}^TP_{2k+1}B_{k})K_{k}^\ast\Sigma_{k} + B_{k}^TP_{2k+1}A_{k}\Sigma_{k},
\]
which gives $K_{k}^\ast = K_{2k}^\ast$, \gls*{acr:lqr} solution for $A_{k},B_{k},R_{k},P_{2k+1}$.

With these expression, the input effort becomes:

\begin{displaymath}
\begin{aligned}
\int_{\statespace{k}} \norm{K_{2k}^\ast(x-m_{k}) + K_{1k}^\ast m_{k}}^2_{R_{k}} &\d\mu_{k}(x)
\\
= \trace\biggl[&
{K_{2k}^\ast}^TR_{k}K_{2k}^\ast\Sigma_{k} + {K_{1k}^\ast}^TR_{k}K_{1k}^\ast M_{k}\\
&+ \cancel{2\int_{\statespace{k}}(x-m_{k})^T{K_{2k}^\ast}^TR_{k}K_{1k}^\ast m_{k}\d\mu_{k}(x)}
\biggr].
\end{aligned}
\end{displaymath}
The expressions in $\Sigma_{k}$ and $M_{k}$ are indeed decoupled, and so are the next mean and variance: 
\[
\begin{aligned}
m_{k+1} &= A_{k}m_{k} + B_{k}K_{2k}^\ast m_{k} - B_{k}K_{2k}^\ast m_{k} - B_{k}K_{1k}^\ast m_{k} = (A_{k}+B_{k}K_{1k}^\ast)m_{k},\\
\Sigma_{k+1} &= (A_{k}+B_{k}K_{2k}^\ast)\Sigma_{k}(A_{k}+B_{k}K_{2k}^\ast)^T.
\end{aligned}
\]
Finally, collecting the terms in $M$ and $\Sigma$ we obtain the desired recursion:
\[\belowdisplayskip=-12pt
\begin{aligned}
P_{1k} &= Q_{1k} + A_{k}^TP_{1k+1}A_{k} - A_{k}^TP_{1k+1}B_{k}(R_{k}+B_{k}^TP_{1k+1}B_{k})^{-1}B_{k}^TP_{1k+1}A_{k},\\
P_{2k} &= Q_{2k} + A_{k}^TP_{2k+1}A_{k} - A_{k}^TP_{2k+1}B_{k}(R_{k}+B_{k}^TP_{2k+1}B_{k})^{-1}B_{k}^TP_{2k+1}A_{k}.
\end{aligned}
\]
\end{proof}
\section{Derivation of \cref{example:linearsystems:finitehorizon:lqr:generic}}
\begin{proof}
\label{proof:lifting:linearsystems:inputeffort:detailed}
We now recove the analytical expression for $\costtogosmall{0}(x,y)$, which is a classical \gls*{acr:lqr} with $Q_k = 0$ at every stage. Dropping the stage indeces for clarity, and referring to the quantities at the next stage with $(\cdot)^+$, the base for the induction is:
\begin{displaymath}
\costtogosmall{}^+(x,y) = x^TP_x^+x + y^TP_y^+y + 2y^TP_{xy}^+x,
\end{displaymath}
which is satisfied by the last stage with $P_{Nx} = P_{Nx}^T = P_{Ny} = -P_{Nxy} = P_N$.
Then, proceeding by induction, we have
\begin{displaymath}
\begin{aligned}
\costtogosmall{}(x,y) 
= \min_{u \in \inputspace{}}\,
& \norm{u}_{R}^2 + \norm{A x + B u}_{P_x^+}^2 + \norm{y}_{P_y^+}^2 + 2(Ax+Bu)^TP_{xy}^+y.
\end{aligned}
\end{displaymath}
Hence,
\begin{displaymath}
\begin{aligned}
u(x,y) 
&= -(R + B^TP_{x}^+B)^{-1} B^T\left[P_x^+Ax + P_{xy}^+y\right]\\
&= -K\left[P_x^+Ax + P_{xy}^+y\right]\\
&= -K_x x -K_y y.
\end{aligned}
\end{displaymath}
Substituting the optimal control law in the cost to go we obtain the cost terms
\begin{displaymath}
\begin{aligned}
P_x 
&= A^TP_x^+A - A^TP_xBK_x = P_x^T,\\
P_y
&= K_y^T(R + B^TP_x^+B)K_y + P_y^+ - 2(BK_y)^TP_{xy}^+,\\
&= P_y^+ - (BK_y)^TP_{xy}^+ = P_y^T,\\
P_{xy}
&= {K_x}^TRK_y + (A-BK_x)^T(-P_x^+BK_y+P_{xy}^+)\\
\overset{\heartsuit}&{=}(A-BK_x)^TP_{xy}^+,
\end{aligned}
\end{displaymath}
where in $\heartsuit$ we used
\begin{displaymath}
-A^TP_x^+BK_y + {K_x}^T(R + B^TP_x^+B)K_y = -A^TP_x^+BK_y + A^TP_x^+BK_y = 0.
\end{displaymath}
We conclude the recursion noticing that
\begin{displaymath}
\costtogosmall{}(x,y) = x^TP_xx + y^TP_yy + 2y^TP_{xy}x.
\end{displaymath}
\end{proof}

\section{Solution lifting: pitfalls and hopes}\label{section:lifting:noise}
\subsection{Global cost terms cannot be solved via lifting}
Suboptimal discussions and exampless....

\begin{theorem}
\label{theorem:lifting:optimaltransport:cherrypick}
Define $\widetilde{\costtogo{}}: \functorprobabilitystate{\statespace{}}\to\nonnegativeRealsBar$ and $\widetilde{\costtogo{}}^+:\functorprobabilitystate{\statespace{}^+}\to\nonnegativeRealsBar$ as:
\begin{displaymath}
\begin{aligned}
\widetilde{\costtogo{}}(\mu) 
&= \min_{\plan{\gamma} \in \setPlans{\mu}{\mu^\ast}} \int_{\statespace{}\times\statespace{N}} \costtogosmall{}(x,y)\d\plan{\gamma}(x,y)
\quad\text{and}\\ 
\widetilde{\costtogo{}}^+(\mu) 
&= \min_{\plan{\gamma} \in \setPlans{\mu}{\mu^\ast}} \int_{\statespace{}\times\statespace{N}} \costtogosmall{}^+(x,y)\d\plan{\gamma}(x,y).
\end{aligned}
\end{displaymath}
If $\costtogo{}^+ = \widetilde{\costtogo{}}^+$, then also $\costtogo{} = \widetilde{\costtogo{}}$. Moreover, there exists an optimal transport map $T^\ast$ such that
\begin{displaymath}
\costtogo{}(\mu) = \int_{\statespace{}}\costtogosmall{}(x,T^\ast(x))\d\mu(x)
\end{displaymath}
and an optimal control input $u^* \in \probabilityInputSpace{}$ is defined as $u(x) = u_{xT^\ast(x)}$ $\mu\ae{}$, where $u_{xy}$ yields $\costtogosmall{}(x,y)$.
\end{theorem}
\begin{proof}
Let $T,T^+$ be the optimal transport maps for $\stagecost{}$ and $\costtogo{}^+$, respectively:
\begin{displaymath}
\begin{aligned}
\stagecost{}(\mu,u) 
&= \int_{\statespace{}} \int_{\statespace{}^+}\stagecostsmall{}(x,T(x),u(x),w_l)\d\nu(w_l)\d\mu(x)
\quad\text{and}\\
\costtogo{}^+(\mu^+) 
&= \int_{\statespace{}^+} \costtogosmall{}^+(y,T^+(y))\d\mu(y).
\end{aligned}
\end{displaymath}
It holds that:
\begin{displaymath}
\begin{aligned}
\costtogo{}(\mu) 
&= \min_{u \in \probabilityInputSpace{}}
\stagecost{}(\mu, u) 
+ \alpha\costtogo{}^+(\dynamicsbig{\nu}(\mu, u))\\
&= \min_{u \in \probabilityInputSpace{}}
\int_{\statespace{}} \int_{\statespace{}^+}\stagecostsmall{}(x,T(x),u(x),w_l)\d\nu(w_l)\d\mu(x)\\
&\qquad\qquad+ \alpha 
\int_{\statespace{}^+} \costtogosmall{}^+(y,T^+(y))\d\dynamicsbig{\nu}(\mu,u)(y)\\
\overset{\heartsuit}&{=} 
\min_{u \in \probabilityInputSpace{}}
\int_{\statespace{}} \int_{\statespace{}^+}\stagecostsmall{}(x,T(x),u(x),w_l)\d\nu(w_l)\d\mu(x)\\
&\qquad\qquad + \alpha 
\int_{\statespace{}^+}\int_{\statespace{}^+} \costtogosmall{}^+(z+w_l,T^+(z+w_l))\d\pushforward{\dynamicsnonlinear{}(\cdot,u(\cdot),w_l)}{\mu}(z)\d\nu(w_l)\\
\overset{\clubsuit}&{=} 
\min_{u \in \probabilityInputSpace{}}
\int_{\statespace{}} \int_{\statespace{}^+}\stagecostsmall{}(x,T(x),u(x),w_l)\d\nu(w_l)\d\mu(x)\\
&\qquad\qquad + \alpha 
\int_{\statespace{}^+}\int_{\statespace{}} \costtogosmall{}^+(\dynamics{}(x,u(x),w_l),T^+(\dynamics{}(x,u(x),w_l)))\d\mu(x)\d\nu(w_l)\\
\overset{\diamondsuit}&{=} 
\min_{u \in \probabilityInputSpace{}}
\int_{\statespace{}} \int_{\statespace{}^+}
\stagecostsmall{}(x,T(x),u(x),w_l) \\
&\qquad\qquad+ \alpha \costtogosmall{}^+(\dynamics{}(x,u(x),w_l),T^+(\dynamics{}(x,u(x),w_l)))\d\nu(w_l)\d\mu(x)\\
\overset{\spadesuit}&{=} 
\min_{u \in \probabilityInputSpace{}}
\int_{\statespace{}} \int_{\statespace{}^+}
\stagecostsmall{}(x,T(x),u(x),w_l) \\
&\qquad\qquad+ \alpha \costtogosmall{}^+(\dynamics{}(x,u(x),w_l),T(x))\d\nu(w_l)\d\mu(x)\\
\overset{\triangle}&{=} 
\int_{\statespace{}} \int_{\statespace{}^+}
\stagecostsmall{}(x,T(x),u^*(x),w_l) + \alpha \costtogosmall{}^+(\dynamics{}(x,u^*(x),w_l),T(x))\d\nu(w_l)\d\mu(x)\\
&= \widetilde{\costtogo{}}(\mu),
\end{aligned}
\end{displaymath}
where in:
\begin{itemize}
    \item[$\heartsuit$] we expanded the convolution in the definition of $\dynamicsbig{}$;
    \item[$\clubsuit$] we used~\cref{theorem:pushforward:optimaltransport};
    \item[$\diamondsuit$] we used Tonelli's theorem ~\cite[Theorem 11.28]{Aliprantis2006};
    \item[$\spadesuit$] we used the Bellman's principle of optimality for the dynamic programming algorithm on the original space; and
    \item[$\triangle$] we used the same step as in $\diamondsuit$ in the proof of \cref{theorem:lifting:optimaltransport}.
\end{itemize}
Finally, using the definition of $\stagecost{}$, it holds $T = T^\ast$.
\end{proof}

\begin{example}
\label{example:mathias:noisy}
First, similarly to \cref{example:linearsystems:finitehorizon:lqr:generic}, we recover the analytical expression for $\costtogosmall{}(x,y)$, which is a classical noisy \gls*{acr:lqr}. Dropping the stage indeces for clarity, and referring to the quantities at the next stage with $(\cdot)^+$, the base for the induction is:
\begin{displaymath}
\costtogosmall{}^+(x,y) = x^TP_x^+x + y^TP_y^+y + 2y^TP_{xy}^+x + x^Tc_x^+ + y^Tc_y^+ + c_w^+,
\end{displaymath}
which is satisfied by the last stage with $P_{Nx} = P_{Nx}^T = P_{Ny} = -P_{Nxy}$ and $c_x = c_y = c_w = 0$.
Then, proceeding by induction, we have
\begin{displaymath}
\begin{aligned}
\costtogosmall{}(x,y) 
= \min_{u \in \inputspace{}}\,
\expectedValue{\nu}{}\biggl[
&\norm{x-y}_{Q}^2 + \norm{u}_{R}^2 + \norm{A x + B u + w}_{P_x^+}^2 + \norm{y}_{P_y^+}^2 \\
&+ 2(Ax+Bu+w)^TP_{xy}^+y +(Ax+Bu+w)^Tc_x^+ + y^Tc_y^+ + c_w^+
\biggr]
\end{aligned}
\end{displaymath}
Manipulating the expression, we obtain:
\begin{displaymath}
\begin{aligned}
\costtogosmall{}(x,y) 
= \min_{u \in \inputspace{}}\,
& \norm{x-y}_{Q}^2 + \norm{u}_{R}^2 + \norm{A x + B u}_{P_x^+}^2 + \norm{y}_{P_y^+}^2 + 2(Ax+Bu)^TP_{xy}^+y \\
&+ y^Tc_y^+ + c_w^+ + (Ax+Bu)^Tc_x^+\\
&+ \expectedValue{\nu}{2(A x + B u)^TP_x^+w + \norm{w}_{P_x^+}^2 + 2w^TP_{xy}^+y + w^Tc_x^+}\\
= \min_{u \in \inputspace{}}\,
&\norm{x-y}_{Q}^2 + \norm{u}_{R}^2 + \norm{A x + B u}_{P_x^+}^2 + \norm{y}_{P_y^+}^2 + 2(Ax+Bu)^TP_{xy}^+y \\
&+ y^Tc_y^+ + c_w^+ + (Ax+Bu+m_\nu)^Tc_x^+ \\
&+ 2(A x + B u)^TP_x^+m_\nu + \trace(P_x^+\Sigma_\nu) + 2m_\nu^TP_{xy}^+y,
\end{aligned}
\end{displaymath}
where $m_\nu$ and $\Sigma_\nu$ are the mean and variance of the probability measure $\nu$ (e.g., $\nu = \gaussian{m_\nu}{\Sigma_\nu}$). Hence,
\begin{displaymath}
\begin{aligned}
u(x,y) 
&= -(R + B^TP_{x}^+B)^{-1} B^T\left[P_x^+(Ax+m_\nu) + P_{xy}^+y + c_x^+\right]\\
&= -K\left[P_x^+Ax + P_{xy}^+y + P_x^+m_\nu + c_x^+\right]\\
&= -K_x x -K_y y - k_w.
\end{aligned}
\end{displaymath}
Substituting the optimal control law in the cost to go we obtain the noise independent cost terms
\begin{displaymath}
\begin{aligned}
P_x 
&= Q + A^TP_x^+A - A^TP_xBK_x = P_x^T,\\
P_y
&= Q + K_y^T(R + B^TP_x^+B)K_y + P_y^+ - 2(BK_y)^TP_{xy}^+,\\
&= Q + P_y^+ - (BK_y)^TP_{xy}^+ = P_y^T,\\
P_{xy}
&= -Q +{K_x}^TRK_y + (A-BK_x)^T(-P_x^+BK_y+P_{xy}^+)\\
\overset{\heartsuit}&{=}-Q + (A-BK_x)^TP_{xy}^+,
\end{aligned}
\end{displaymath}
where in $\heartsuit$ we used
\begin{displaymath}
-A^TP_x^+BK_y + {K_x}^T(R + B^TP_x^+B)K_y = -A^TP_x^+BK_y + A^TP_x^+BK_y = 0,
\end{displaymath}
and the noise dependent ones:
\begin{displaymath}
\begin{aligned}
c_x 
&= (A-BK_x)^T(c_x^++2P_x^+m_\nu),\\
c_y
&= c_y^+-(BK_y)^T(c_x^+ + 2P_x^+m_\nu) + 2P_{xy}^{+T}m_\nu,\\
c_w
&= c_x^+ + (m_\nu-Bk_\nu)^Tc_x^+ - 2Bk_\nu P_x^+m_\nu + \trace(P_x^+\Sigma_\nu).
\end{aligned}
\end{displaymath}
We conclude the recursion noticing that
\begin{displaymath}
\costtogosmall{}(x,y) = x^TP_xx + y^TP_yy + 2y^TP_{xy}x + x^Tc_x + y^Tc_y + c_w.
\end{displaymath}

\begin{remark}
In absence of noise $(\nu = \delta_0)$ and with hard terminal constraint $(P_N^x = P_N^y = -P_N^{xy} \to \infty)$ we recover \cite{HudobadeBadyn2021}.
\end{remark}
\begin{remark}
Implementing the controller one does not need to keep track of $c_y$ and $c_w$, and the optimal transport map can be equivalently computing minimizing $\tilde{\costtogosmall{}}(x,y) = x^TP_{xy}y$, since the other terms can be marginalized.
\end{remark}

An application of \cref{example:linearsystems:finitehorizon:lqr:generic} is shown in \cref{fig:example:lqr:generic:eth:withnoise}, where the setting is analogous to \cref{example:linearsystems:finitehorizon:lqr:generic}. The noise is distributed as $\nu_k = \gaussian{0}{0.1I_2}$ and the $Q_k$ matrices are $Q_k = Q = 0.01I_2$.

At a macroscopic level, the final distribution computing the optimal transport solution once at the beginning or at every stage would be approximatively the same, due to the locality of the noise; however, the trajectory of the swarm (not of the pre-dispatched single particles) would be worse. For instance, in the provided example, the cumulative costs of all the particles per each stage with feedback improves of about $2\%$ compared to the results without. Of course, the improvement is problem settings dependent and is provided only for explanation purposes; see \cref{example:cornercase:worseningwithoutfeedback} for a corner-case example with infinite worsening of the cost without feedback. Nevertheless, the main take-home message is that the concept of optimality in the probability space is richer than the one on the original space, and that the optimality of each particle trajectory does not imply optimality of the trajectory of the whole swarm.

For reference, in \cref{fig:example:lqr:generic:eth:withoutnoise} we juxtaposed the final configuration when the simulation is run with $\nu = \delta_0$.
\end{example}

\subsection{Noisy dynamics cannot be lifted}
In~\cref{chapter:dp:finitehorizon,chapter:dp:infinitehorizon} we consider a particle dynamics that accounts for both global, possibly non-linear, noise and local, additive, disturbance. In~\cref{chapter:dp:lifting}, instead, we restrict the results to a deterministic pushforward. In this section, we see why the results cannot be extended with the tools and formulation used in this work. Additionally, the counterexamples will offer us a better understanding of the optimal control problem in probability spaces, and point towards possible directions for future work.

The first example shows why, in general, global noise impedes the lifting of the solution from the ground space:
\begin{example}\label{example:lifting:globalnoise}
Let $\statespace{} = \statespace{}^+ = \{0,1\}$, $\inputspace{}=\{0,1\}$ and $\noisespace{} = \{0,1\}$. The dynamics and the stage cost are:
\begin{displaymath}
\dynamics{}(x,u,w) = \begin{cases}
    u\quad&\text{if } w = 0\\
    1-u&\text{otherwise}
\end{cases}\quad\text{and}\quad\stagecostsmall{}(x,u,w) = \begin{cases}
    0\quad&\text{if } x = u, w = 0\\
    1&\text{otherwise}.
\end{cases}
\end{displaymath}
The noise probability measure $\xi$ is defined as $\xi(0) = \xi(1) = 1/2$. An optimal solution is given by $u_0 = 0, u_1 = 1$. When lifting this solution, this leads to a cost of $1$, which is suboptimal compared to the feedback law $u(x) = 0$, which yields $1/2$. 

Although it is true that also $u_0 = u_1 = 1$ or $u_0 = u_1 = 0$ are optimal solutions that when lifted still yield the optimal value, this example shows that in general the lifting does not preserve the optimality, when the noise is not independent between the particles.
\end{example}

The next example shows that even local uncorrelated noise cannot be treated optimally with the tools of this section. Instead, as in the global noise case, one has to employ the results of \cref{chapter:dp:finitehorizon,chapter:dp:infinitehorizon}.

\begin{example}\label{example:lifting:localnoise}
Consider a horizon $N = 2$ and, for all $i \in \naturals_2$, $j \in \naturals_1$, the spaces $\statespace{i} = \inputspace{j} = \noisespace{j} = \naturals_3$ endowed with the bynary operation
\begin{displaymath}
\begin{aligned}
+: \naturals_3\times\naturals_3 &\to \naturals_3\\
(x,y) &\mapsto \modulo{x+y}{3}.
\end{aligned}
\end{displaymath} 

Each particle evolves according to the noisy integrator dynamics $\dynamics{}(x,u_x,w_l) = u_x + w_l$.

The noise measures are
\begin{displaymath}
\nu_0 = \frac{1}{2}\left(\delta_0+\delta_1\right)
\qquad\text{and}\qquad
\nu_1 = \delta_0.
\end{displaymath}

Define the cost terms:
\begin{displaymath}
\begin{aligned}
\stagecostsmall{0}(x,u_x,w_l) &= \begin{cases}
0\quad&\text{if }u_x = x\\
\infty&\text{otherwise},
\end{cases}\qquad\terminalcostsmall(x,y) = \begin{cases}
0\quad&\text{if }y = x\\
\infty&\text{otherwise}
\end{cases}\qquad\text{and}\\
\stagecostsmall{1}(x,u_x,w_l) &= \begin{cases}
\alpha\quad&\text{if } (x \in \{0,1\}, u_x = 0) \text{ or } (x \in \{2,3\}, u_x = 3)\\
\beta\quad&\text{if } (x = 0, u_x = 3) \text{ or } (x = 3, u_x = 0)\\
0\quad&\text{if } (x = 1, u_x = 3) \text{ or } (x = 2, u_x = 0)\\
\infty&\text{otherwise}.
\end{cases}
\end{aligned}
\end{displaymath}
Consider the initial probability measure $\mu_0$ and the target probability measure $\mu^\ast$ defined as
\begin{displaymath}
\mu_0 = \frac{1}{2}\left(\delta_0 + \delta_2\right)
\qquad\text{and}\qquad
\mu^\ast = \frac{1}{2}\left(\delta_0 + \delta_2\right),
\end{displaymath}

and the noise measures:
\begin{displaymath}
\nu_0 = \frac{1}{2}\left(\delta_0 + \delta_1\right)
\qquad\text{and}\qquad
\nu_1 = \delta_0.
\end{displaymath}

The (relevant) particles evolution with the related costs is pictured in \cref{fig:example:local_noise:scheme}; the dashed lines represents transitions that cannot be chosen via the input but may occur due to the noise.

Solving the dynamic programming recursion for a single particle from $x_0 \in \{0,2\}$ to $y \in \{0,2\}$ we obtain:
\begin{displaymath}
\costtogosmall{1}(x,0) = \begin{cases}
\alpha\quad&\text{if } x \in \{0,1\}\\
0\quad&\text{if } x = 2\\
\beta\quad&\text{if } x = 3
\end{cases},\qquad
\costtogosmall{1}(x,3) = \begin{cases}
\beta\quad&\text{if } x = 0\\
0\quad&\text{if } x = 1\\
\alpha\quad&\text{if } x \in \{2,3\}
\end{cases}
\end{displaymath}
and
\begin{displaymath}
\begin{aligned}
\costtogosmall{}(0, y) = \costtogosmall{0}(0, y) = \begin{cases}
\alpha\quad&\text{if } y = 0\\
\frac{\beta}{2}&\text{if } y = 3\\
\infty&\text{otherwise}
\end{cases},\qquad
\costtogosmall{}(2, y) = \costtogosmall{0}(2, y) = \begin{cases}
\frac{\beta}{2}\quad&\text{if } y = 0\\
\alpha&\text{if } y = 3\\
\infty&\text{otherwise}.
\end{cases}
\end{aligned}
\end{displaymath}

Let $\beta > 2\alpha$. Then, lifting the solution we obtain
\begin{displaymath}
\widetilde{\costtogo{}}(\mu_0) = \min_{\plan{\gamma} \in \setPlans{\mu_0}{\mu^\ast}}\int_{\statespace{0}\times\statespace{2}}\costtogosmall{}(x,y)\d\plan{\gamma} = \alpha.
\end{displaymath}
Instead, solving directly the recursion in the probability space, we have \begin{displaymath} 
\mu_1 = \convolution{\mu_0}{\nu_0} = \frac{1}{4}\sum_{k = 0}^{3} \delta_k
\end{displaymath}
and the cost to go is
\begin{displaymath}
\begin{aligned}
\costtogo{}(\mu_0) 
&= \costtogo{0}(\mu_0) = 0 + \costtogo{1}(\mu_1) \\
&= \frac{1}{4}\left(\costtogosmall{1}(0,0) + \costtogosmall{1}(2,0) + \costtogosmall{1}(1,3) + \costtogosmall{1}(3,3)\right) \\
&= \frac{\alpha}{2} < \widetilde{\costtogo{}}(\mu_0).
\end{aligned}
\end{displaymath}

\end{example}

\begin{remark}
\cref{example:lifting:localnoise} is important not only because shows that the claim in \cref{theorem:lifting:optimaltransport} cannot be extended to encompass local noise in the current formulation, but because highlights the relevance of \cref{theorem:lifting:optimaltransport} itself. Indeed, its statement is rather intuitive in hindsight, and would be what one would do performing multi-agent control. However, as remarked by \cref{example:lifting:localnoise}, this would be in general suboptimal. Hence, \cref{theorem:lifting:optimaltransport} empowers us with an optimality guarantee, and justify an approach used in practice that, to the best of our knowledge, was never analyzed in terms of optimality in the space of probability measures.
\end{remark}

Finally, we show where the math stops us from extending~\cref{theorem:lifting:optimaltransport}: this provides additional insights on~\cref{example:lifting:localnoise}, about what goes wrong and what could be a possible approach or formulation that might, instead, work.

Introducing local noise, the stage cost in \eqref{equation:lifting:stagecost} becomes
\begin{displaymath}
\stagecost{}(\mu,u) = \int_{\statespace{}}\int_{\statespace{}^+}\stagecostsmall{}(x, u(x), w_l)\d\nu(w_l)\d\mu(x) = \expectedValue{\mu}{\int_{\statespace{}^+}\stagecostsmall{}(x,u(x), w_l)\d\nu(w_l)},
\end{displaymath}
and the dynamic programming formulation is:
\begin{displaymath}
\costtogosmall{}(x, y) = \min_{u_{xy} \in \inputspace{}}\int_{\statespace{}^+}\stagecostsmall{}(x,u_{xy},w_l) + \alpha\costtogosmall{}^+(\dynamics{}(x,u_{xy},w_l), y)\d\nu(w_l),
\end{displaymath}
with $\alpha \in [0,1]$.

Then, with the same settings of \cref{theorem:lifting:optimaltransport}, it holds that:
\begin{displaymath}
\begin{aligned}
\costtogo{}(\mu) 
&= \min_{u \in \functorprobabilityinput{\inputspace{}}}
\stagecost{}(\mu, u) 
+ \alpha\costtogo{}^+(\dynamicsbig{\nu}(\mu, u))\\
&= 
\stagecost{}(\mu, u^*) 
+ \alpha \min_{\plan{\gamma_\nu^+} \in \setPlans{\dynamicsbig{\nu}(\mu,u^*)}{\mu^\ast}}\int_{\statespace{}^+\times\statespace{\ast}} \costtogosmall{}^+(y,z)\d\plan{\gamma_\nu^+}(y,z)\\
&= 
\stagecost{}(\mu, u^*) 
+ \alpha \min_{\plan{\gamma_\nu} \in \setPlans{\mu\times\nu}{\mu^\ast}} \int_{(\statespace{}\times\statespace{}^+)\times\statespace{\ast}}\costtogosmall{}^+(\dynamics{}(x,u^*(x),w_l),z)\d\plan{\gamma_\nu}((x,w_l),z)\\
&= 
\int_{\statespace{}}\int_{\statespace{}^+}\stagecostsmall{}(x,u^*(x),w_l)\d\nu(w_l)\d\mu(x)\\
&\qquad\qquad\qquad+ \alpha \min_{\plan{\gamma_\nu} \in \setPlans{\mu\times\nu}{\mu^\ast}} \int_{(\statespace{}\times\statespace{}^+)\times\statespace{\ast}}\costtogosmall{}^+(\dynamics{}(x,u^*(x),w_l),z)\d\plan{\gamma_\nu}((x,w_l),z)\\
&= 
\int_{\statespace{}\times\statespace{}^+}\stagecostsmall{}(x,u^*(x),w_l)\d\nu\times\mu(x,w_l)\\
&\qquad\qquad\qquad+ \alpha \min_{\plan{\gamma_\nu} \in \setPlans{\mu\times\nu}{\mu^\ast}} \int_{(\statespace{}\times\statespace{}^+)\times\statespace{\ast}}\costtogosmall{}^+(\dynamics{}(x,u^*(x),w_l),z)\d\plan{\gamma_\nu}((x,w_l),z)\\
&= 
\min_{\plan{\gamma_\nu} \in \setPlans{\mu\times\nu}{\mu^\ast}} \int_{(\statespace{}\times\statespace{}^+)\times\statespace{\ast}}
\stagecostsmall{}(x,u^*(x),w_l)\\
&\qquad\qquad+ \alpha \costtogosmall{}^+(\dynamics{}(x,u^*(x),w_l),z)\d\plan{\gamma_\nu}((x,w_l),z).
\end{aligned}
\end{displaymath}

That is, we can show that the cost to go $\costtogo{}$ is always an optimal transport cost also in this case. However, we cannot substitute the cost expression with $\costtogosmall{}$ because the noise measure $\nu$ is coupled in the optimal transport plan.

Let $\gamma_\nu^\ast \in \setPlans{\mu\times\nu}{\mu^\ast}$ and $\gamma^\ast = \pushforward{(\id\times T)}{\mu} \in \setPlans{\mu}{\mu^\ast}$ be optimal transport plans:
\begin{displaymath}
\begin{aligned}
\costtogo{}(\mu) &= \int_{(\statespace{}\times\statespace{}^+)\times\statespace{\ast}}
\stagecostsmall{}(x,u^*(x),w_l)
+ \alpha \costtogosmall{}^+(\dynamics{}(x,u^*(x),w_l),z)\d\plan{\gamma_\nu^\ast}((x,w_l),z),\text{ and}\\
\widetilde{\costtogo{}}(\mu) &= \min_{\plan{\gamma} \in \setPlans{\mu}{\mu^\ast}}\int_{\statespace{}\times\statespace{\ast}} \costtogosmall{}(x,z)\d\plan{\gamma} = \int_{\statespace{}\times\statespace{\ast}} \costtogosmall{}(x,z)\d\plan{\gamma^\ast}.
\end{aligned}
\end{displaymath}
Then, using ~\cref{theorem:disintegration} we can equivalently write:
\begin{equation}\label{equation:lifting:disintegration:noisedependentoncoupling}
\begin{aligned}
\costtogo{}(\mu) 
= 
\int_{\statespace{}\times\statespace{\ast}}
\int_{\statespace{}^+}
\stagecostsmall{}(x,u^*(x),w_l) + \alpha\costtogosmall{}^+(\dynamics{}(x,u^*(x),w_l),z)\d\plan{\gamma_\nu^{\ast xz}}(w_l)\d\plan{\gamma^\ast}(x,z),
\end{aligned}
\end{equation}
and
\begin{equation}\label{equation:lifting:disintegration:noisedependenttarget}
\begin{aligned}
\costtogo{}(\mu) 
= 
\int_{\statespace{}^+}
\int_{\statespace{}\times\statespace{\ast}}
\stagecostsmall{}(x,u^*(x),w_l) + \alpha\costtogosmall{}^+(\dynamics{}(x,u^*(x),w_l),z)\d\plan{\gamma_\nu^{\ast w_l}}(x,z)\d\nu(w_l).
\end{aligned}
\end{equation}

If $\plan{\gamma_\nu^{\ast xz}} = \nu$ in \eqref{equation:lifting:disintegration:noisedependentoncoupling}, then we would obtain $\costtogo{} \equiv \widetilde{\costtogo{}}$. However, this is in general not the case, as shown in \cref{example:lifting:localnoise}. Characterizing sufficient conditions under which the equality holds would allow to extend the scope of \cref{theorem:lifting:optimaltransport}.

Assume that $\nu\ae{}$ $\plan{\gamma_\nu^{\ast w_l}}$ is induced by a transport map $T_{w_l}: \statespace{}\to\statespace{\ast}$. Define $T:\statespace{}\times\statespace{}^+\to\statespace{\ast}$ as $T(x,w_l) = T_{w_l}(x)$. Then, \eqref{equation:lifting:disintegration:noisedependentoncoupling} becomes:
\begin{displaymath}
\begin{aligned}
\costtogo{}(\mu)
&=
\int_{\statespace{}^+}
\int_{\statespace{}}
\stagecostsmall{}(x,u^*(x),w_l) + \alpha\costtogosmall{}^+(\dynamics{}(x,u^*(x),w_l),T(x,w_l))\d\mu(x)\d\nu(w_l)\\
&=
\int_{\statespace{}}
\int_{\statespace{}^+}
\stagecostsmall{}(x,u^*(x),w_l) + \alpha\costtogosmall{}^+(\dynamics{}(x,u^*(x),w_l),T(x,w_l))\d\nu(w_l)\d\mu(x),
\end{aligned}
\end{displaymath}
where in the last step we used Tonelli's theorem \cite[Theorem 11.28]{Aliprantis2006}. Namely, we obtain a possible starting point for a computational approach to solve the optimal control problem in probability spaces:
\begin{displaymath}
\begin{aligned}
\costtogo{}(\mu) 
&= 
\min_{T}\expectedValue{\mu,\nu}{\costtogosmall{}(x,T(x,w_k))}\\
&=
\min_{T}\expectedValue{\mu}{\min_{u_x}\expectedValue{\nu}{\stagecostsmall{}(x,u_x,w_l) + \alpha\costtogo{}^+(\dynamics{}(x,u_x,w_l),T(x,w_l))}}.
\end{aligned}
\end{displaymath}

Let $\mu, \nu \in \spaceProbabilityBorelMeasures{\statespace{}}$. Further analysis on the expression of $\costtogo{}(\mu)$ yields to the conclusion that the failing step is that, in general
\begin{displaymath}
\begin{aligned}
\heartsuit &= \min_{\plan{\gamma} \in \setPlans{\convolution{\mu}{\nu}}{\mu^\ast}}
\int_{\statespace{}^+\times\statespace{\ast}}\costtogosmall{}(y,z)\d\plan{\gamma}(y,z) 
\\
&\neq 
\min_{\plan{\gamma} \in \setPlans{\mu}{\mu^\ast}}
\int_{\statespace{}\times\statespace{\ast}}
\int_{\statespace{}}
\costtogosmall{}(x+w_l,z)
\d\nu(w_l)\d\plan{\gamma}(x,z) = \diamondsuit.
\end{aligned}
\end{displaymath}

For the sake of intuition, consider now a discrete setting. Roughly speaking, the number of particles in $\convolution{\mu}{\nu}$ are more than in $\mu$. This increases the degrees of freedom of the transport plan in $\heartsuit$ compared to $\diamondsuit$, and it is not compensated, in general, by the ``smoothing'' on the cost structure. Indeed, we can prove that $\diamondsuit \geq \heartsuit$:
\begin{displaymath}
\begin{aligned}
\diamondsuit 
&= 
\int_{\statespace{}\times\statespace{\ast}}
\int_{\statespace{}}
\costtogosmall{}(x+w_l,z)
\d\nu(w_l)\d\plan{\gamma^\ast}(x,z)\\
&=
\int_{\statespace{}\times\statespace{\ast}}
\costtogosmall{}(y,z)
\d\pushforward{(+, \id)}{\pushforward{(\id,\pi_{\statespace{}}\times\pi_{\statespace{\ast}}}{\nu\times\plan{\gamma^\ast}}}(y,z).
\end{aligned}
\end{displaymath}

Since $\convolution{\mu}{\nu} = \pushforward{+}{\mu\times\nu}$, it follows that $\pushforward{(+, \id)}{\pushforward{(\id,\pi_{\statespace{}}\times\pi_{\statespace{\ast}}}{\nu\times\plan{\gamma^\ast}}}(y,z) \in \setPlans{\convolution{\mu}{\nu}}{\mu^\ast}$ and thus, $\diamondsuit \geq \heartsuit$.

When we perform the lifting, the above consideration means that different particles might have their mass split to different destinations at the next stage: this cannot be embedded in a traditional dynamic programming recursion and thus, the lifted solution becomes suboptimal. It can be noticed that this is indeed what happens in \cref{example:lifting:localnoise}.

These observations hint that a possible alternative formulation, amenable of lifting in a more general setting, would have a dynamic programming recursion of the form:
\begin{displaymath}
\costtogo{}(\mu) = \min_{u \in \functorprobabilityinput{\inputspace{}}} \expectedValue{\nu}{\stagecost{}(\mu,u;W) + \alpha\costtogo{}^+(\dynamicsbig{W}(\mu,u))},
\end{displaymath}
where the $\stagecost{}$ and the $\functorprobability{\dynamics{}}$ are parametrized on the collection of noise values $W$; i.e., $W$ is a realization of the collection of independent and identically distributed random variables, one for each particle, representing the local noise. Then, we take the expected value over these realizations.
\subsection{Stochastic shortest path problems}
\label{section:lifting:spp:stochastic}

In this section we study the problem of steering a probability measure in a finite space. We start formulating the (stochastic) shortest path problem on a (weighted) graph, to later show the two are equivalent. We provide a practical framework to deal with a wide class of problems, that we later instantiate a relevant study case: see \cref{example:gridworld}.

Consider a graph $\graph(\vertexSet,\edgeSet)$, where ($\vertexSet,+_{\vertexSet}$) is a finite commutative group (e.g., see \cref{example:noise:groupstructure}) and $\edgeSet$ is the, possibly time-varying, weighted edge space 
\begin{displaymath}
\edgeSet\coloneqq \left\{(i,j,w,c_{i,j,w}^k) \in \vertexSet\times\vertexSet\times\vertexSet\times\nonnegativeRealsBar \st i,j,w \in \vertexSet, k \in \naturals_N\right\}, N \in \naturals.
\end{displaymath}
The term $c_{i,j,w}^k$ denotes the cost to go from the vertex $i$ to the vertex $j$ at step $k$ when the noise $w$ is sampled. By convention, one can set $c_{i,j,w} = \infty$ if the transition is not feasible.

Given a path $Q = (i_0, \ldots, i_q)$, where $i_{k} \in \vertexSet$ for all $k \in \naturals_q$, and the noise realizations $W = (w_0, \ldots, w_{q-1})$, where $w_{k} \in \vertexSet$ for all $k \in \naturals_{q-1}$, we define its associated cost as
\begin{displaymath}
\costtogo{Q,W}^{y_{i_0}} = \sum_{h = 0}^{q-1} c_{i_{h},a_{h},w_{h}}^h + \terminalcostsmall^{y_{i_0}}(i_q),
\end{displaymath}
where $\terminalcostsmall$ is a penalty term parametrized by $y_{i_0} \in \vertexSet$ and $a_{k} \coloneqq i_{h+1} -_\vertexSet w_{h}$ are the \emph{decisions} at step $h \in \naturals_q$. Namely, what would have been the next step in the path in the absence of noise. We call $i_0$ \emph{source}, and $y_{i_0}$ \emph{destination} of the path $Q$.

The objective of this section is to study the problem of finding the best decisions to go from a set of sources to a set of destinations in a finite number of steps $N$:
\begin{problem}[Multi-source multi-destination shortest path problem]\label{problem:shortestpath}
Fix $N, H \in \naturals$. Let $\mathcal{S} = \{s_1, \ldots, s_{H}\} \subseteq \vertexSet$ be a set of sources and $\mathcal{D} = \{t_1, \ldots, t_{H}\} \subseteq \vertexSet$ be a set of destinations. The multi-source multi-destination shortest path problem then is to find the minimizers of the total expected cost:
\begin{displaymath}
\begin{aligned}
\costtogo{\mathcal{S}} 
= \min \,& \expectedValue{}{\sum_{l = 1}^H \costtogo{Q_l,W_l}^{T(s_l)}}\\
\mathrm{s.t.}\,
& T \in \mathcal{S}^\mathcal{D}\\
& i_{l,k+1} = a_{l,k} +_\vertexSet w_{l,k},\,i_{l,0} = s_{l,0}\\
& Q_l = (i_{l,k})_{k \in \naturals_N}\\
& w_{l,k} \sim \nu_k\\
& W_l = (w_{l,k})_{k \in \naturals_{N-1}},
\end{aligned}
\end{displaymath}
where $\nu_k \in \spaceProbabilityBorelMeasures{\vertexSet}$ are fixed and known
\end{problem}

\begin{remark}
$\tilde{N} < N$ steps might be enough for some source-destination pairs. To model the termination case, one can introduce cost-free self-loops. Then, a proper choice of $\terminalcostsmall$ suffices. However, notice that fixing a horizon $N$ for all pairs is a more reasonable choice when noise is involved: it is desirable to find a path that keeps each particle close to the target destination until the end of the horizon.
\end{remark}

The following example show how to instantiate the abstract formulation of additive noise with respect to the group structure on the state space, which in this section coincides with the set of vertices:

\begin{example}\label{example:noise:groupstructure}
For a fix $H \in \naturals$ take $N_1, \ldots, N_H \in \naturals$. Let $\statespace{} = \naturals_{N_1}\times\ldots\times\naturals_{N_H}$. Define $+_{\statespace{}}: \statespace{}\times\statespace{}\to\statespace{}$ as
\begin{displaymath}
(a_1, \ldots, a_H) +_{\statespace{}} (b_1, \ldots, b_H) = (\modulo{a_1 + b_1}{N_1}, \ldots, \modulo{a_H + b_H}{N_H}).
\end{displaymath}
Then, $(\statespace{}, +_{\statespace{}})$ is a commutative group. With a noise measure $\nu \in \spaceProbabilityBorelMeasures{\noisespace{}}$ that assigns probability with an inverse relation to the absolute values of the entries of a touple $w \in \noisespace{} = \statespace{}$ we obtain an intuitive formulation: a periodic multi-dimensional grid world where the agent drifts due to the noise on the neighbouring cells. For instance, $\nu$ defined as
\begin{displaymath}
\nu((w_1, \ldots, w_H)) = \begin{cases}
0\quad&\text{if }w_i > 1 \text{ for some } i \in \{1, \ldots, H\}\\
\frac{1}{2^H}&\text{otherwise}
\end{cases}
\end{displaymath}
moves the agents only on neighbouring cells (also on diagonals) and 
\begin{displaymath}
\nu((w_1, \ldots, w_H)) = \begin{cases}
0\quad&\text{if } \sum_{i = 1}^H w_i > 1\\
\frac{1}{H}&\text{otherwise}
\end{cases}
\end{displaymath}
only to cells sharing a side with the current location, in both cases with uniform probability.
\end{example}

\cref{example:noise:groupstructure} is an easy, yet interesting, noise formulation: it is arguably the most intuitive for all the shortest path problems on implicit graphs. 

\cref{problem:shortestpath} can be formulated as \cref{problem:finitehorizon}. Let $\statespace{k} = \inputspace{k} = \vertexSet$, for all $K \in \naturals_N$. The dynamics is $\dynamics{k}(x,u,w) = u +_\vertexSet w$, where $w \sim \nu \in \spaceProbabilityBorelMeasures{\vertexSet}$. The probability measure $\mu_k$ and the target measure $\mu^\ast$ are then empirical measures describing $(i_{l,k})_{l \in \{1,2,\ldots,H\}}$ and $\mathcal{D}$, respectively:
\begin{displaymath}
\mu_k = \frac{1}{H}\sum_{l = 1}^H \delta_{i_{l,k}}\quad\text{ and }\quad\mu^\ast = \frac{1}{|\mathcal{D}|}\sum_{t \in \mathcal{D}}\delta_t.
\end{displaymath}
\cref{problem:shortestpath} can be written as a finite-horizon optimal control problem with a cosmetic transformation: the stage cost becomes $\stagecostsmall{k}(x,u_{xy},w) = c^k_{x,u_{x},w}$, and the terminal cost is $\terminalcostsmall(x,y) = \terminalcostsmall^y(x)$; given the continuity of the terminal cost and of the stage cost, and the compactness of the action space and noise space (finite), we can solve it via dynamic programming. Let $\costtogosmall{k}(x,y)$ be the cost to go at stage $k$ from node $x$ to node $y$, and $u_k(x,y)$ be the corresponding optimal control law.

The map $T_N$ in \cref{problem:shortestpath} is relevant only at the last stage; it can thus be embedded in the terminal cost as:
\begin{displaymath}
\terminalcost(\mu_N) = \min_{\plan{\gamma_N} \in \setPlans{\mu_N}{\mu^\ast}} \int_{\statespace{}\times\statespace{}} \terminalcostsmall^y(x)\d\plan{\gamma_N}(x,y) = \int_{\statespace{}} \terminalcostsmall^{T_N(x)}(x)\d\mu_N(x).
\end{displaymath}

The stage cost at the $\nth{k}$ stage is
\begin{displaymath}
\stagecost{k}(\mu_k, u) = \expectedValue{\mu}{\int_{\noisespace{k}} \stagecostsmall{k}(x,u_k(x),w)\d\nu_k(x)}.
\end{displaymath}

The following toy example instantiates the previous results in a comprehensive study case:

\begin{example}[Forest ride]\label{example:gridworld}
We want to drive a swarm of quadcopters through a forest so that they reach a final target configuration; see \cref{fig:example:gridworld:description}. The flags represent the target locations, which are not assigned a priori to a specific drone, and the green blobs are \emph{obstacle} to avoid (e.g., trees). Let $H$ be the height of the grid and $W$ be its width.

The state space is the set of all the cells $\statespace{k} = \statespace{} = \naturals_{W-1}\times\naturals_{H-1}$in the grid world, and is endowed with the group structure described in \cref{example:noise:groupstructure}.

The simulation lasts $N$ steps. At every step the agents can either go \texttt{LEFT}, \texttt{RIGHT}, \texttt{UP}, \texttt{DOWN} or they can \texttt{HOVER} on the current cell. These five actions thus represent the input space $\inputspace{k} = \inputspace{}$ at every stage $k \in \naturals_{N-1}$. In general, the dynamics is noisy. In particular, the noise is additive and the noise space corresponds to the state space, since the latter is time invariant.

Conveniently, for an element $e = (e_h,e_v) \in \statespace{}$ (i.e., state or noise realization) we write
\begin{displaymath}
\pi_h(e) \coloneqq e_h\qquad\text{ and }\qquad\pi_v(e) = e_v,
\end{displaymath}
and we define the function $l: \statespace{} \to \{\text{\texttt{FREE}}, \text{\texttt{TREE}}, \text{\texttt{FLAG}}\}$ as the \emph{label function} that reveals the status of a cell $e \in \statespace{}$.

We distinguish two formulations based on the noise measure, which can be defined point-wise (it is a discrete probability measure):
\begin{enumerate}
    \item in the deterministic formulation, we have $\nu = \delta_{(0,0)}$.
    \item in the noisy formulation, the noise measure is
\begin{displaymath}
\nu(w) = \begin{cases}
1/3\quad&\text{if }\pi_h(w) = 0 \text{ and } \pi_v(w) \in \{H-1,0,1\}\\
0&\text{otherwise},
\end{cases}
\end{displaymath}
    Namely, the noise possibly shifts \texttt{UP} or \texttt{DOWN} the quadcopter.
\end{enumerate}

The cost an agent incurs depend on the current state, the input and the noise: if the resulting (next) position is a \texttt{FREE} cell (or a \texttt{FLAG}), then it incurs $+1$ if it applied an input $u \neq \text{\texttt{HOVER}}$ and $0$ it no input was applied ($u = \text{\texttt{HOVER}}$); otherwise a much larger cost (e.g., $+100$): the quadcopter has to perform an expensive manouver to avoid crashing.

The grid world is periodic: when a quadcopter flies outside the grid it enters from the opposite side. To prevent a trivial solution when the agent crosses the red line on the left it incurs an infinite penalty. Practically, this corresponds to restrict the set of available actions for the subset of states in the first column of the grid. Notice that with other penalties one can make the grid world non periodic -- the agents will stay away from the borders.

Conveniently, for an input $u \in \inputspace{}$ we write:
\begin{displaymath}
\pi_h(u) \coloneqq \begin{cases}
1\quad&\text{if } u = \text{\texttt{RIGHT}}\\
W-1&\text{if } u = \text{\texttt{LEFT}}\\
0&\text{otherwise}
\end{cases}\qquad\text{ and }\qquad
\pi_v(u) \coloneqq \begin{cases}
1\quad&\text{if } u = \text{\texttt{UP}}\\
W-1&\text{if } u = \text{\texttt{DOWN}}\\
0&\text{otherwise}.
\end{cases}
\end{displaymath}

Hence, the dynamics at every stage $k$ becomes
\begin{displaymath}
\dynamics{k}(x, u, w) = \dynamics{}(x, u, w) = x + (\pi_h(u),\pi_v(u)) + w,
\end{displaymath}

and the cost function is
\begin{displaymath}
\stagecostsmall{k}(x, u, w) = \stagecostsmall{}(x, u, w) = \begin{cases}
+100\quad&\text{if }l(\dynamics{}(x,u,w)) = \text{ \texttt{TREE}}\\
+\infty&\text{if }\pi_h(x) = 0\text{ and }l(\dynamics{}(x,u,w)) = \text{ \texttt{FLAG}}\\
0&\text{if }l(\dynamics{}(x,u,w)) = \text{ \texttt{FREE} and } u = \text{ \texttt{HOVER}}\\
+1&\text{otherwise}.
\end{cases}
\end{displaymath}

With the notation of the shortest path problems, we set $\vertexSet = \statespace{}$ and we have that the weights of the edges are:
\begin{displaymath}
c_{i,a,w}^h = c_{i,a,w} = \begin{cases}
\infty\quad&\text{if }(\pi_h\times\pi_v)^{-1}(a -_\vertexSet i) = \emptyset\\
\stagecostsmall{}(i,u,w)&\text{otherwise, with }u \in (\pi_h\times\pi_v)^{-1}(a -_\vertexSet i),w),
\end{cases}
\end{displaymath}
where $(\pi_h\times\pi_v)^{-1}(a -_\vertexSet i) = \emptyset$ if and only if the decision $a$ is out of reach from node $i$; that is, if there is no $u \in \inputspace{}$ such that $i + (\pi_h(u),\pi_v(u)) = a$ or equivalently, if $i$ and $a$ are not adjacent in the grid. Notice that if $(\pi_h\times\pi_v)^{-1}(a -_\vertexSet i) \neq \emptyset$, then it is a singleton.

After $N$ steps, the swarm incurs the optimal transport cost to the desired configuration, with cost $\terminalcostsmall(x,y)$ defined via a large coefficient (e.g., $+1000$) multiplied by the Manhattan distance between $x$ and $y$. Notice that if the setting were deterministic, and the coefficient was $\infty$, with the convention $\infty \cdot 0 = 0$, we would enforce that the agents reach the target destination. In a stochastic setting, this is not a reasonable formulation. Our results encompass both hard and soft final constraints.
\end{example}

In the deterministic formulation (i) of \cref{example:gridworld}, we can apply \cref{theorem:lifting:optimaltransport}: to solve the optimal control problem in probability spaces it suffices to solve the easy dynamic programming problem on the resulting graph and then lift the solution to find the optimal control law for the swarm via optimal transport. Clearly, it is possible that two quadcopter are in the same location. Practically, we implemented the approach described in \cref{remark:spp:transportmap:existence}. 

Some snapshots of a realization of the resulting simulation are provided in \cref{fig:example:gridworld:nonnoisy}, where the time evolution is color-coded from red (start) to blue (end).

\begin{remark}
The dynamic programming table needed for the lifting can also be computed efficiently with the Floyd-Warshall algorithm \cite[$\S$ 25.2]{Cormen2009}.
\end{remark}

\cref{theorem:lifting:optimaltransport} not only gives us a feedback law in the particles space, but also in the probability measures space: (possibly) at each state one can compute again the optimal coupling and update the particles trajectory. This is relevant not only for robustness to unmodelled noise on both microscopic and macroscopic scales, but it also opens the possibility of time-varying targets:
whenever the target distribution $\mu^\ast$ changes, the optimal transport problem is solved again.

Therefore, the toy example \cref{example:gridworld} becomes fundamentally important for a wide number of applications related to fleet control, such as \gls*{acr:amod} \cite{Zardini2021} or warehouse automation.
The case where a forecast for the future target distribution is available is part of ongoing research and will be included in future work.

In general, the noisy formulation (ii) of \cref{example:gridworld} requires the solution of the dynamic programming equation as described in \cref{chapter:dp:finitehorizon}. However, one could also approach the problem as in the deterministic formulation (i), solving a stochastic dynamic programming recursion in the ground space. The resulting trajectory of the swarm in this case is shown in \cref{fig:example:gridworld:noisy}. However, as discussed in \cref{section:lifting:noise}, this might be suboptimal. Moreover, with noisy dynamics, the need for feedback in probability space becomes fundamental to avoid further suboptimality:
\begin{example}\label{example:cornercase:worseningwithoutfeedback}
Let $N = 2$ and $\statespace{k} = \statespace{} = \{0, 1\}$ with the group structure introduced in \cref{example:noise:groupstructure}. The input space is $\inputspace{k} = \inputspace{} = \statespace{}$, and the dynamics is the integrator dynamics: $\dynamics{k}(x, u, w) = \dynamics{}(x, u, w) = u + w$. Then, the cost penalizes switching attempts:
\begin{displaymath}
\stagecostsmall{k}(x, u, w) = \stagecostsmall{}(x, u, w) = 
\begin{cases}
0\quad&\text{if } x = u\\
1&\text{otherwise}.
\end{cases}
\end{displaymath}
We enforce the final distribution constraint, namely:
\begin{displaymath}
\terminalcostsmall(x, y) = \begin{cases}
0\quad&\text{if } x = y\\
\infty&\text{otherwise}.
\end{cases}
\end{displaymath}
Finally, the distribution of the noise is time-varying: 
\begin{displaymath}
\nu_0 = \frac{1}{2}\delta_0 + \frac{1}{2}\delta_1\qquad\text{and}\qquad\nu_1 = \delta_0,
\end{displaymath}
where $\nu_1 = \delta_0$ to ensure the well-posedness of the problem, since we enforce the final distribution constraint as a hard constraint.

Then, clearly the optimal control law is:
\begin{displaymath}
u_0(x,y) = x\qquad\text{and}\qquad u_1(x,y) = y.
\end{displaymath}

Let $\mu_0 = \mu_\ast = \frac{1}{2}\delta_0 + \frac{1}{2}\delta_1$. Then, any optimal transport map is optimal. Take for instance $T(x) = x$. If at the first stage the realization of the noise make the system switch, at the second stage we will incur a cost of $2$ without feedback, whereas the re-dispatch would lead to a cost of $0$. In this corner-case, there is an infinite worsening without feedback in the probability space (i.e., without solving an additional optimal transport problem).
\end{example}

\begin{remark}
\Cref{example:cornercase:worseningwithoutfeedback} provides a perspective on the role of optimal transport in the control law: it acts as a feedback component in the probability space.
\end{remark}

To use the results of \cref{section:dp}, we need lower semi-continuity of the costs. This property is preserved by the definitions in~\eqref{equation:lifting:terminalcost} and~\eqref{equation:lifting:stagecost}:
\begin{lemma}[Semi-continuity is preserved by the lifting]\label{lemma:lifting:lsc}
Let $l: \statespace{}\times\inputspace{}\times\statespace{N}$ such that $l_y \in \lsc{\statespace{}\times\inputspace{}}{\nonnegativeRealsBar}$. Then, for any $\mu^\ast \in \spaceProbabilityBorelMeasures{\statespace{N}}$, the map
\begin{displaymath}
\begin{aligned}
L_{\mu^\ast}&: \spaceProbabilityBorelMeasures{\statespace{}}\times\C{0}{\statespace{}\times\statespace{N}}{\inputspace{}} \to \nonnegativeRealsBar\\
(\mu, u) &\mapsto L_{\mu^\ast}(\mu,u) = \inf_{\plan{\gamma}\in\setPlans{\mu}{\mu^\ast}} \int_{\statespace{}\times\statespace{N}} l(x, u(x,y), y) \d\plan{\gamma}(x,y)
\end{aligned}
\end{displaymath}
is lower semi-continuous.
\end{lemma}
\begin{proof}
Consider the arbitrary converging sequences
\begin{displaymath}
\begin{aligned}
&(\mu_n)_{n \in \naturals} \subset \spaceProbabilityBorelMeasures{\statespace{}}, \mu_n \narrowconvergence \mu \in \spaceProbabilityBorelMeasures{\statespace{}}\quad\text{and}\\
&(u_m)_{m \in \naturals} \subset \C{0}{\statespace{}\times\statespace{N}}{\inputspace{}}, u_m \to u \in \C{0}{\statespace{}\times\statespace{N}}{\inputspace{}}.
\end{aligned}
\end{displaymath}
We want to prove that
\begin{displaymath}
\liminf_{n \to \infty} \inf_{\plan{\gamma_n}\in\setPlans{\mu_n}{\mu^\ast}} \int_{\statespace{}\times\statespace{N}} l(x, u_n(x,y), y) \d\plan{\gamma_n}(x,y) = L_{\mu^\ast}(\mu_n,u_n) \geq L_{\mu^\ast}(\mu,u).
\end{displaymath}
Let $\Lambda \subseteq \naturals$ be the subsequence yielding the $\liminf$. Then we have:
\begin{displaymath}
\begin{aligned}
L_{\mu^\ast}(\mu_n,u_n)
&= \liminf_{n \to \infty}\inf_{\plan{\gamma_n}\in\setPlans{\mu_n}{\mu^\ast}} \int_{\statespace{}\times\statespace{N}} l(x, u_n(x,y), y) \d\plan{\gamma_n}(x,y) \\
\overset{\heartsuit}&{=} \liminf_{n \to \infty}\sup_{\phi_n, \psi_n} \int_{\statespace{}} \phi_n(x)\d\mu_n(x) - \int_{\statespace{N}}\psi_n(y)\d\mu^\ast(y)\\
\overset{\clubsuit}&{\geq} \liminf_{n \to \infty}\int_{\statespace{}} \phi^\ast(x)\d\mu_n(x) - \int_{\statespace{N}}\psi^\ast(y)\d\mu^\ast(y)\\
\overset{\diamondsuit}&{=} \int_{\statespace{}} \phi^\ast(x)\d\mu(x) - \int_{\statespace{N}}\psi^\ast(y)\d\mu^\ast(y)\\
\overset{\heartsuit}&{\geq} \inf_{\plan{\gamma}\in\setPlans{\mu}{\mu^\ast}} \int_{\statespace{}\times\statespace{N}} l(x, u(x,y), y) \d\plan{\gamma}(x,y)-\varepsilon \\
&= L_{\mu^\ast}(\mu,u)-\varepsilon,
\end{aligned}
\end{displaymath}
where in:
\begin{itemize}
    \item[$\heartsuit$] we used~\cref{theorem:kantorovichduality} (Kantorovich duality);
    \item[$\clubsuit$] we used the hypothesis of $l$ being $\lowersemicontinuous$: $l(x, u_n(x,y), y) \geq l(x, u(x,y), y)$ and thus, for sufficiently large $n$, the bounded $\varepsilon$-optimizers $\phi^\ast, \psi^\ast$ for the cost $l(x,u(x,y),y)$ are viable as candidate $\phi_n, \psi_n$, yielding the desired inequality; and
    \item[$\diamondsuit$] we used the definition of narrow convergence, since $\phi^\ast$ and $\psi^\ast$ are bounded.
\end{itemize}
Finally, we let $\varepsilon\to0$ to conclude. 
\end{proof}

Clearly, the results provided in~\cref{lemma:lifting:lsc} are comprehensive of our setting and thus, the cost terms in~\eqref{equation:lifting:terminalcost} and~\eqref{equation:lifting:stagecost} are lower semi-continuous.

\bibliographystyle{siamplain}
\bibliography{references}